\documentclass[12pt,
    a4paper,
]{article} 

\usepackage{hyperref}
\usepackage{mymacros}
\usepackage{enumitem}
\usepackage{booktabs}
\usepackage{mathrsfs}
\usepackage{graphicx}
\usepackage{mathdots}

\usepackage{bm}

\usepackage{tikz}
\usetikzlibrary{
    arrows.meta,
    intersections,
    calc,
    decorations.pathreplacing,
    decorations.pathmorphing,
    math,
    cd,
}

\tikzset{
    moment polytope/.style={thick, fill=lightgray, draw=black}
}

\author{Matthew~R.~Buck}
\title{Lagrangian spheres and cyclic quotient T-singularities}
\date{\today}

\newcommand{\JTS}{J_{T^*S^2}}
\newcommand{\Jref}{J_{\mathrm{ref}}}

\newcommand{\splitcurve}[1]{({#1}^{(0)}, \ldots, {#1}^{(N)})}
\newcommand{\toplevel}[1]{#1_+^{(0)}}
\newcommand{\bottomlevel}[1]{#1_-^{(0)}}

\newcommand{\selfastint}[1]{i_U(#1,#1)}

\newcommand{\constraint}[1]{\mathbf{c}_{#1}}

\newcommand{\Eplus}{\mathcal{E}^+}
\newcommand{\Eminus}{\mathcal{E}^-}
\newcommand{\FEplus}{\mathcal{G}^+}
\newcommand{\FEminus}{\mathcal{G}^-}
\newcommand{\Eplusminus}{\mathcal{E}^\pm}
\newcommand{\FEplusminus}{\mathcal{G}^\pm}

\newcommand{\cotangentPlanesModuli}{\mathcal{M}}
\newcommand{\cylindersModuli}{\mathcal{M}_\mathrm{cyl}}
\newcommand{\cylindersModuliBar}{\overline{\mathcal{M}}_\mathrm{cyl}}
\newcommand{\weinsteinModuli}[1]{\overline{\mathcal{M}}_L({#1})}

\newcommand{\TSlefFib}{\pi_{T^*S^2}}
\newcommand{\refFib}{\pi_{\mathrm{ref}}}

\newcommand{\fibrationSectionL}{S}
\newcommand{\fibrationSectionR}{S'}

\newcommand{\Mod}{\mathrm{Mod}}

\newcommand{\building}[1]{F_{#1}}

\newcommand{\weinsteinNbhd}{Y_-}

\newcommand{\Cinfloc}{C_\mathrm{loc}^\infty}

\newcommand{\pqtwist}{\tau_{p,q}}

\newcommand{\infinityCurve}[1]{\mathbf{u}_\infty^{#1}}

\newcommand{\XJ}[1]{X_{#1}}

\newcommand{\DinftyJ}[1]{D_{\infty,#1}}

\numberwithin{equation}{section}

\usepackage[
    backend=biber,
    url=true,
    doi=true,
    sorting=nyt,
    natbib=true
]{biblatex}
\usepackage{bibentry}
\begin{document}

\maketitle
\section*{\centering Abstract}

We study the Lagrangian isotopy classification of Lagrangian spheres
in the Milnor fibre, \(B_{d,p,q}\), of the cyclic quotient
surface T-singularity \(\frac{1}{dp^2}(1,dpq-1)\). We prove that there is a
finitely generated group of symplectomorphisms such that the orbit of a fixed
Lagrangian sphere exhausts the set of Lagrangian isotopy classes. Previous
classifications of Lagrangian spheres have been established in simpler
symplectic \(4\)-manifolds, being either monotone, or simply connected, whilst
the family studied here satisfies neither of these properties. We construct
Lefschetz fibrations for which the Lagrangian spheres are isotopic to matching
cycles, which reduces the problem to a computation involving the mapping class
group of a surface. These fibrations are constructed using the techniques of
\(J\)-holomorphic curves, and Symplectic Field Theory, culminating in the
construction of a \(J\)-holomorphic foliation by cylinders of \(T^*S^2\). Our
calculations provide evidence towards the symplectic mapping class group of
\(B_{d,p,q}\) being generated by Lagrangian sphere Dehn twists and another
type of symplectomorphism arising as the monodromy of the
\(\frac{1}{p^2}(1,pq-1)\) singularity.

\tableofcontents

\section{Introduction}
In this paper we study the Lagrangian isotopy classification of
Lagrangian spheres in the family of 4-dimensional symplectic manifolds
\(B_{d,p,q}\). These manifolds are the Milnor fibres of the cyclic quotient
singularities \linebreak\({\frac{1}{dp^2}(1,dpq-1)}\), where \(d,p,q\) are
positive integers such that \(p > q \ge 1\) are coprime. When \(d=1\), there
are no Lagrangian spheres at all, whereas for \(d > 1\), we
prove that there is a finitely generated subgroup \(G\) of the symplectic
mapping class group and a fixed Lagrangian sphere \(L_0\) such that for any other Lagrangian
sphere \(L \subset B_{d,p,q}\) there is a symplectomorphism \(\phi \in G\)
such that \(L\) is Lagrangian isotopic to \(\phi(L_0)\).

First, we recall some elementary concepts. Let \(n > 1\) be an integer and
let \(a \ge 1\) be coprime to \(n\). Consider the action of the group of
\(n\)-th roots of unity, \(\Gamma_n\), on \(\C^2\) with weights \((1,a)\):
\[
    \mu \cdot (x,y) = (\mu x, \mu^a y).
\]
The quotient space \(\C^2/\Gamma_n\) is a singular manifold called the cyclic
quotient surface singularity of type \(\frac{1}{n}(1,a)\). The \(A_{dp-1}\)
singularity is the variety
\[
    \{(z_1,z_2,z_3) \in \C^3 \mid z_1z_2 = z_3^{dp}\},
\]
and the group \(\Gamma_p\) acts on it with weights \((1,-1,q)\). The cyclic
quotient singularity \(\frac{1}{dp^2}(1,dpq-1)\) is analytically isomorphic to
the quotient \(A_{dp-1}/\Gamma_p\) via the map
\[
    \C^2/\Gamma_{dp^2} \to A_{dp-1}/\Gamma_p : (x,y) \mapsto
    (x^{dp},y^{dp},xy).
\]
The symplectic manifold \(B_{d,p,q}\) is defined to be the Milnor fibre of
this singularity. Its symplectic structure is inherited from \(\C^3\) since
the \(\Gamma_p\) action is by symplectomorphisms. As a result, we can
explicitly write
\[
    B_{d,p,q} := \left.\left(z_1z_2 - \prod_{i=1}^d (z_3^p -
	i)\right)\right/\Gamma_p.
\]

Following \S{}7.4 of \cite{evansLTF}, \(B_{d,p,q}\) can be equipped with a
Hamiltonian system that gives it the structure of an almost toric manifold.
A fundamental action domain for this system is shown in
Figure~\ref{fig:bdpqFundamentalDomain}.
\begin{figure}\centering \begin{tikzpicture}[
    moment polytope,
    >=Rays,
]
    \filldraw[>=To,<->] (0,3) -- node [left] {\((0,1)\)} (0,0) --
	node [below right] {\((dp^2, dpq - 1)\)} (8,3) ;
    \draw[style=dotted] (0,0) -- (2,1) node {\(\times\)};
    \draw[style=dotted] (2,1) -- (3,1.5) node {\(\times\)} node [above] {\((p,q)\)};
\end{tikzpicture}
    \caption{
	A fundamental action domain for \(B_{d,p,q}\). The labels indicate the
	primitive integer direction of each arrow and ray. Drawn here is the
	case \((d,p,q) = (2,2,1)\).
    }
    \label{fig:bdpqFundamentalDomain}
\end{figure}
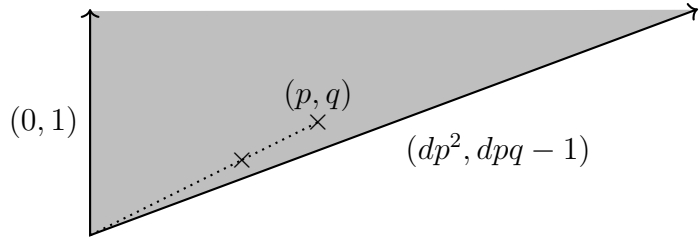
This picture determines \(B_{d,p,q}\) up to symplectomorphism, and is useful
for visualising its topology. Over the dotted lines connecting the crosses live \(d-1\)
Lagrangian 2-spheres, and over the line connecting the unique vertex of the
wedge to a cross lives a Lagrangian CW-complex called a \((p,q)\)-pinwheel.
This is a Lagrangian immersion of a 2-disc, which is an embedding on the
interior, and maps \(p\)-to-\(1\) on the boundary
\cite{khodorovskiy2013symplecticrationalblowup, evansSmith18markov}.
We shall refer to these \(d-1\) spheres as the
\emph{standard} Lagrangian spheres in \(B_{d,p,q}\), and denote their
associated Dehn twists by \(\tau_1,\ldots,\tau_{d-1}\) \cite{seidelLectures}.
The CW-complex formed by the topological wedge sum of these \(d-1\) spheres
and the \((p,q)\)-pinwheel is called the \emph{Lagrangian skeleton} of
\(B_{d,p,q}\). 

Similarly to the case of Lagrangian spheres, one can define a
symplectomorphism associated to a Lagrangian \((p,q)\)-pinwheel, which we
write as \(\tau_{p,q}\). The construction is carried out in
Section~\ref{ch:mcg}. The only thing to note for now is that \(\tau_{p,q}\) is
a compactly-supported symplectomorphism of \(B_{d,p,q}\) with support in a
neighbourhood of the pinwheel itself.

We are now ready to state the main result.
\begin{theorem}[]
    \label{thm:intro:main}
    Let \(L \subset B_{d,p,q}\) be a Lagrangian sphere. Then \(L\) is
compactly-supported Hamiltonian isotopic to a sphere obtained by applying a
word of symplectomorphisms generated by \(\tau_1,\ldots,\tau_{d-1}\), and
\(\tau_{p,q}\), to a standard sphere.
\end{theorem}
For a more precise statement, see Theorem~\ref{thm:main}.

\subsection{Historical context}

Classifications of Lagrangian spheres have been achieved in various cases:
starting with Hind \cite{hind2004Spheres} who proved that every Lagrangian
sphere in \({S^2 \times S^2}\) is Hamiltonian isotopic to the
anti-diagonal \(\bar{\Delta} = \{(x,-x) \in S^2 \times
S^2\}\). Evans \cite{evans2010delPezzo} then proved a similar theorem for
Lagrangian spheres in certain del~Pezzo surfaces.\footnote{A symplectic del~Pezzo
surface is either \(S^2 \times S^2\) or \(\CP^2\) blown-up in \(n < 9\)
generic points equipped with an anticanonical K\"ahler form
\cite[Definition~1.3]{evans2010delPezzo}. Evans proves that no Lagrangian
knotting occurs in the \(n=2,3,4\) cases.} These theorems give examples of
symplectic manifolds where Lagrangian knotting does not occur, that is,
two Lagrangian spheres are Lagrangian isotopic if, and only if, they are
smoothly isotopic.

On the other hand, Seidel proved \cite{seidel99SympKnotting} that Lagrangian
knotting occurs in general. The proof revolves around a special
symplectomorphism associated to a Lagrangian sphere \(L\) called a generalised
Dehn twist\footnote{In the case \(n=1\) this is the classical Dehn twist.}
\(\tau_L\). Many of Seidel's papers give the explicit construction of the Dehn
twist, and we refer to \cite{seidelLectures}. The squared twist \(\tau_L^2\)
is always smoothly isotopic to the identity map, but, depending on the ambient
symplectic manifold, it may, or may not, be symplectically
isotopic. For example, in \(T^*S^2\) equipped with the
canonical symplectic form \(\omega_\mathrm{can}\), any iterate \(\tau^k\) of
the Dehn twist about the zero-section is \emph{not} isotopic to the identity
map through symplectomorphisms. In fact, more is true: \(\tau\) generates the
\emph{symplectic mapping class group} of \((T^*S^2,\omega_\mathrm{can})\),
which is defined to be the group of connected components of the group of
compactly-supported symplectomorphisms, see \cite{seidel1998automorphisms}.
However, the squared Dehn twist about the antidiagonal in \(S^2 \times S^2\)
is symplectically isotopic to the identity map.

Despite this, classifications of Lagrangian spheres in manifolds where
knotting occurs have been achieved. Consider the Milnor fibre of the \(A_n\)
surface singularity. That is, the complex manifold \(W_n\) given by the
equation
\[
    W_n = \{(z_1,z_2,z_3) \mid z_1^2 + z_2^2 + z_3^{n+1} = 1\} \subset \C^3.
\]
Equipped with the restriction of the symplectic form \(\omega_{\C^3}
= \frac{i}{2}\sum_{i=1}^3 \ud z_i \wedge \ud\bar{z_i}\), this is a symplectic
4-manifold. In \cite{hind12Stein}, Hind proves that any Lagrangian sphere in
\(W_1\) or \(W_2\) is Lagrangian isotopic to one obtained from a finite set of
``standard'' Lagrangian spheres by applying Dehn twists about these standard
spheres. Wu \cite{wu14spheresInAnSingularities} then extended this result to
all \(W_n\). One can rephrase these results as saying that the only symplectic
knotting of Lagrangian spheres in the \(A_n\) Milnor fibres comes from Dehn
twists. The main result here (Theorem~\ref{thm:intro:main}) proves the
corresponding result for the family \(B_{d,p,q}\).

\subsection{Outline of the proof}

At its core, the idea of the proof is to use the theory of Lefschetz fibrations
and matching cycles to construct Lagrangian isotopies of spheres. The primary
step is to construct a Lefschetz fibration for which an arbitrary Lagrangian
sphere \(L \subset B_{d,p,q}\) is Lagrangian isotopic to  a matching cycle. We
achieve this by neck stretching around \(L\), which necessitates an
understanding of how \(J\)-holomorphic curves behave under deformations of the
almost complex structure \(J\). To facilitate this analysis, we first
compactify \(B_{d,p,q}\) to simplify the holomorphic curve analysis.

The inspiration for this method is the following local example of a Lefschetz
fibration on \(T^*S^2\).
\begin{example}
    \label{ex:quadricLefFib}
    Let \(Q\) be the affine quadric \((z_1^2 + z_2^2 + z_3^2 = 1) \subset
\C^3\) equipped with the restriction of the standard symplectic form on
\(\C^3\). This is symplectomorphic to \((T^*S^2,\omega_\mathrm{can})\)
\cite[Lemma~18.1]{seidelThesis} via a symplectomorphism that identifies the
zero-section with the real locus of \(Q\). The projection \(\pi_Q : Q \to \C :
\pi(z) = z_3\) is a genus 0 Lefschetz fibration with exactly two singular
fibres \(\pi_Q^{-1}(\pm 1)\). Moreover, the matching cycle of the path
\(\gamma : [-1,1] \to \C : \gamma(t) = t\) is exactly the real locus.
\end{example}

Via Weinstein's Lagrangian neighbourhood theorem, we can compare Lefschetz
fibrations on \(B_{d,p,q}\) to \(\pi_Q\) by restricting them to a
neighbourhood of \(L\). We use a process called neck stretching to construct a
sequence of Lefschetz fibrations \(\pi_k : B_{d,p,q} \to \C^\times\) which
converge (in a suitable sense) to \(\pi_Q\) in a neighbourhood of \(L\).

In Section~\ref{sec:compactification} we find a suitable compactification
\(X_{d,p,q}\) (of a suitable subset) of \(B_{d,p,q}\) whose symplectic
geometry reflects enough of that of \(B_{d,p,q}\) to be able to construct the
Lagrangian isotopy in \(X_{d,p,q}\) and then pull it back to \(B_{d,p,q}\).
In Section~\ref{sec:fibreModuliSpaces} we construct the Lefschetz fibrations
by considering the moduli space of genus 0 curves living in a distinguished
homology class \(F \in H_2(X_{d,p,q};\Z)\) satisfying \(F^2 = F \cdot F = 0\).
For each suitable almost complex structure \(J\), there is a natural map
\(\pi_J : X_{d,p,q} \to \overline{\mathcal{M}}_{0,0}(X_{d,p,q},F;J)\) which
sends a point to the unique stable \(F\)-curve it lies on. This has three
types of fibre: a regular fibre (which is an element of
\(\mathcal{M}_{0,0}(X_{d,p,q},F;J)\)), a nodal fibre with exactly two
transversely intersecting components, and a so-called \emph{exotic fibre} with
potentially many components. After removing the exotic fibre, the map becomes
a Lefschetz fibration with base \(\C\).

The method of producing the compactification \(X_{d,p,q}\) comes from the
almost toric structure on \(B_{d,p,q}\). Using a technique called symplectic
cut (originally due to Lerman \cite{lerman95symplecticCut}) one can
essentially ``chop off'' the non-compact end of \(B_{d,p,q}\) to obtain a
closed symplectic orbifold (see Figure~\ref{fig:bdpqCut}). We do this in
Lemma~\ref{lem:compactify}, and resolve the orbifold singularities
to produce a smooth manifold. The requisite knowledge for these techniques is
presented lucidly in \cite{evansLTF}.

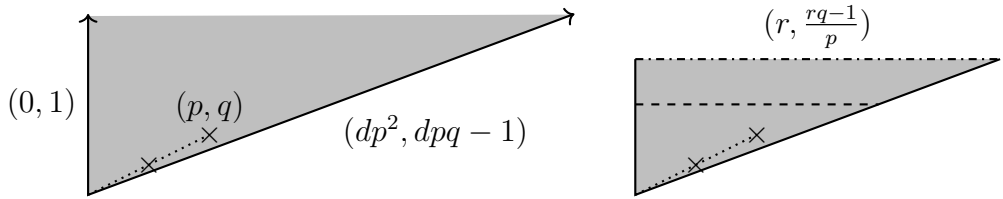
\begin{figure}\centering \begin{tikzpicture}[
    moment polytope,
    >=Rays,
    scale=0.4,
]
    \filldraw[>=To,<->] (0,6) -- node [left] {\((0,1)\)} (0,0) --
	node [below right] {\((dp^2, dpq - 1)\)} (16,6); 
    \draw[style=dotted] (0,0) -- (2,1) node {\(\times\)} -- (4,2) node {\(\times\)} node [above] {\((p,q)\)};
    \begin{scope}[shift={(18,0)}]
	\filldraw (0,4.5) -- (0,0) -- (12,4.5);
	\draw[style=dotted] (0,0) -- (2,1) node {\(\times\)} -- (4,2) node
	    {\(\times\)};
	\draw[style=dash dot] (0,4.5) -- node [above] {\((r,\frac{rq-1}{p})\)}
	    (12,4.5);
	\draw[dashed] (0,3) -- (8,3);
    \end{scope}
\end{tikzpicture}
    \caption{
	The symplectic cut chopping off the non-compact end of
	\(B_{d,p,q}\) is indicated by the dash-dot line, which has
	primitive integer direction \(\left(r,\frac{rq-1}{p}\right)\), where \(r\) is the
	unique integer \(0 < r < p\) such that \(rq \equiv 1 \mod p\). The
	preimage of the area below the horizontal dashed line is a compact
	symplectic submanifold with boundary whose symplectic
	completion is symplectomorphic to \(B_{d,p,q}\). Drawn here is the most
	basic example \((d,p,q)=(2,2,1)\). However, this is not indicative of
	the general case: when \(q > 1\) the dash dot line corresponding to
	the symplectic cut will have positive gradient. See
	Figure~\ref{fig:B322} for another example.
    }
    \label{fig:bdpqCut}
\end{figure}
\begin{figure}\centering \begin{tikzpicture}[
    moment polytope,
    scale=0.5,
]
    \filldraw (0,2) -- (0,0) -- node [below right] {\((18,11)\)} (18,11) -- node
	[above left] {\((2,1)\)} cycle;
    \draw[dotted] (0,0) --  (3,2) node {\(\times\)} -- (6,4) node
	{\(\times\)};
\end{tikzpicture}
    \caption{
	The result of cutting \(B_{2,3,2}\). As is clear, the pictures
	quickly become unwieldy in this form. See
	Figure~\ref{fig:BdpqTransformation} for an integral affine
	transformation of the fundamental action domain of \(B_{d,p,q}\) that
	behaves better under increasing the values \((d,p,q)\).
    }
    \label{fig:B322}
\end{figure}
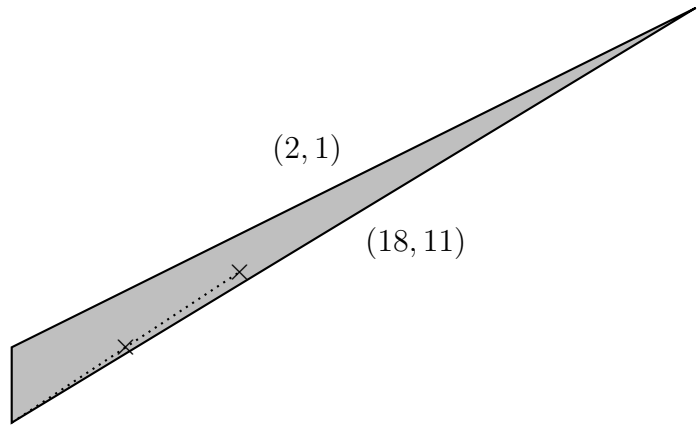

In the spirit of the previously mentioned theorems on Lagrangian isotopy
classifications \cite{hind2004Spheres, evans2010delPezzo}, we use the neck
stretching (or the splitting) technique. We refer to
\cite[\S{}1.3]{eghIntroSFT}, \cite[\S{}3.4]{behwzSFTCompactness}, or
\cite[\S{}2.7]{cieliebakMohnkeSFTCompactness} for the full neck stretching
setup, as we will only need a special case of it here, although we will review
the necessary details in Section~\ref{sec:neck-stretch}. Given a neck
stretching sequence \(J_k\) around a Lagrangian sphere \(L\), we use the SFT
compactness theorem to extract limits of sequences of \(J_k\)-holomorphic
curves. An analysis of these limits is conducted in
Section~\ref{sec:buildingAnalysis}, which is the most technical part of the
paper. We show that, in the limit, we obtain a \(J\)-holomorphic foliation of
\(T^*S^2\) by cylinders (copies of \(\C^\times = \C \backslash \{0\}\)), whose
leaves form the fibres of a Lefschetz fibration akin to \(\pi_Q\) of
Example~\ref{ex:quadricLefFib}. This analysis is largely driven by
intersection theory of punctured curves (reviewed in
Section~\ref{sub:intersection-theory}) and Hind's foliations
\cite{hind2004Spheres} by \(J\)-holomorphic planes of \(T^*S^2\)
(reviewed in Section~\ref{sub:cotangentPlanes}).

The crucial consequence of the SFT limit analysis is that our
Lagrangian sphere \(L\) will be a matching cycle of a limiting Lefschetz
fibration. In Section~\ref{ch:isotopy} we apply a uniform convergence argument
to show that, for sufficiently large neck stretches (sufficiently large
\(k\)), \(L\) is Lagrangian isotopic to a matching cycle of one of the
Lefschetz fibrations \(\pi_{J_k}\).

Section~\ref{ch:mcg} examines the relationship between the mapping class group
of the punctured annulus and Lagrangian knotting of spheres in \(B_{d,p,q}\).
In particular, Section~\ref{sec:dpqSymplectomorphisms} shows that the
symplectic monodromy of the \({\frac{1}{p^2}(1,pq-1)}\) singularity gives a
notion of Dehn twisting about a Lagrangian\linebreak\((p,q)\)-pinwheel. The proof of
the main theorem is completed in Section~\ref{sec:mainTheoremProof}.

\subsection{Acknowledgements}

This work forms the bulk of the author's PhD thesis, which was supported by
the Engineering and Physical Sciences Research Council. Naturally, I am deeply
indebted to my PhD supervisor, Jonny Evans, whom I thank wholeheartedly.
Thanks to my colleagues at Lancaster University for facilitating a stimulating
environment over the past four years. Thanks also to Nikolas Adaloglou and
Johannes Hauber for our many discussions. I am also thankful to Richard
Siefring for pointing me in the right direction when I was confused about some
aspects of their work on intersection theory.

\section{A compactification of \(B_{d,p,q}\) and its Lef\-schetz fibrations}
This section achieves two goals. Firstly, we use an almost toric structure on
\(B_{d,p,q}\) to construct a compactification \(X = X_{d,p,q}\) via a series
of symplectic cuts (see Lemma~\ref{lem:compactify}). Secondly, we show that
for every almost complex structure \(J\) on \(X\) satisfying some
constraints (but importantly including any almost complex structure arising
from neck stretching) there exists a \(J\)-holomorphic foliation of \(X\)
whose leaves form a Lefschetz fibration (Proposition~\ref{prop:lefFibXJ}).

\subsection{Spiel on base diagrams}
\label{sec:momentPolytopes}

In this section we use a little of the theory of (almost) toric base diagrams.
As mentioned in the introduction, the requisite knowledge can be found in
\cite{evansLTF} and the references therein. We state the main definition of a
\emph{Delzant polytope} here for convenience (taken from Definition~3.5 in
\cite{evansLTF}).
\begin{definition}
    A rational convex polytope \(P\) (which we will call a \emph{moment
polytope} or simply a \emph{polytope}) is a subset of \(\R^n\) defined as the
intersection of a finite collection of half spaces \(\{ x \in \R^n \mid
\sum_{i=1}^n \alpha_i x_i \le b\}\) with \(\alpha_i \in \Z\) and \(b \in \R\).
We say that \(P\) is a \emph{Delzant polytope} if it is a convex rational
polytope such that every point on a \(k\)-dimensional facet has a
neighbourhood isomorphic (via a transformation of the form \(Ax + b\) with \(A
\in SL_n(\Z)\) and \(b \in \R^n\)) to a neighbourhood of the origin in the
polytope \([0,\infty)^{n-k} \times \R^k\). A vertex of a polytope is called
\emph{Delzant} if the germ of the polytope at that vertex is Delzant.
\end{definition}

\begin{remark}
    \begin{enumerate}
        \item The Delzant condition is important since it means that the
            symplectic manifold corresponding to the polytope is smooth.
        \item Importantly, as we only work with \(4\)-dimensional symplectic
manifolds (corresponding to \(n=2\) in the above definition), a vertex \(v\)
of a polygon is Delzant if, and only if,\footnote{Note that, for \(n\)-dimensional
    facets (interior points) the Delzant condition is trivial, and for \(n-1\)
    dimensional facets (faces) the condition is automatically satisfied since
    they lie on hypersurfaces of the form \(\{x \in \R^n \mid \sum_{i=1}^n
\alpha_i x_i = b\}\). Therefore, checking the vertices of a \(2\)-dimensional
polygon is enough to verify whether it is Delzant or not.} the primitive integer vectors \(u\)
and \(w\) pointing along the edges that meet at \(v\) form a matrix \(u \wedge
w\) with determinant \(\pm 1\).
    \end{enumerate}
\end{remark}

Let \(p,q \in \Z\) be such that \(\gcd(p,q) = 1\), then write \(0 < r <
p\) for the unique integer such that \(qr \equiv 1 \mod p\). Consider the
singular toric manifold with the non-Delzant moment polygon, \(\Pi(p,q)\),
shown in Figure~\ref{fig:pqPolytope}.
\begin{figure}[ht]\centering
    \begin{tikzpicture}[
        moment polytope,
    ]
        \filldraw (0,2) -- (0,0) node [left] {\(\frac{1}{p}(1,q)\)}
    	-- node [below right] {\((p,q)\)}
    	(2,1) node [right] {\(\frac{1}{p}(1,p-r)\)} -- (2,2);
    
    \end{tikzpicture}
    \caption{
	The moment polygon \(\Pi(p,q)\).
    }
    \label{fig:pqPolytope}
\end{figure}
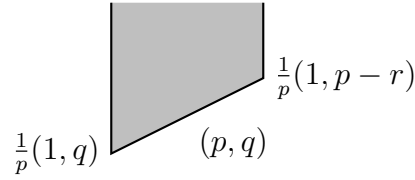
As in \cite[\S{}4.5]{evansLTF}, we can use symplectic cuts
to resolve the cyclic quotient singularities
present in the symplectic orbifold corresponding to the polygon \(\Pi(p,q)\).
The minimal resolution of \(\Pi(p,q)\) is a smooth toric manifold with moment
polygon sketched in Figure~\ref{fig:minRes}.
\begin{figure}[ht]\centering
    \begin{tikzpicture}[
        moment polytope,
        scale=1.5,
    ]
        \coordinate (X-1) at (0,0);
        \foreach \i [remember=\i as \ilast (initially -1)] in {0,...,6}
    	\coordinate (X\i) at ($ (X\ilast) + (7*\i:1cm) $);
        \coordinate (Y) at ($ (X6) + (0,1) $);
        \filldraw let \p1 = (Y) in
    	(0,\y1) -- (0,0) -- node [below] {\(-x_1\)} (X0)
    	\foreach \i/\itext/\pos in
    	{1/\cdots/below,2/-x_m/below,3/{(p,q)}/below right,
    	4/-y_1/below right,5/\iddots/below right,6/-y_k/below right}
    	    { -- node [\pos] {\(\itext\)} (X\i)} -- +(0,1);
        \fill[color=black] (0,0) circle [radius=2pt]
    	\foreach \i in {0,...,6} {(X\i) circle [radius=2pt]};
    \end{tikzpicture}
    \caption{
	The minimal resolution of \(\Pi(p,q)\). The vertices are marked for
        clarity. The labels \(-x_i\), and \(-y_j\) represent the
        self-intersection numbers of the symplectic spheres that live above
        the corresponding edges. The label \((p,q)\) indicates the primitive
        integer direction of its edge.
    }
    \label{fig:minRes}
\end{figure}
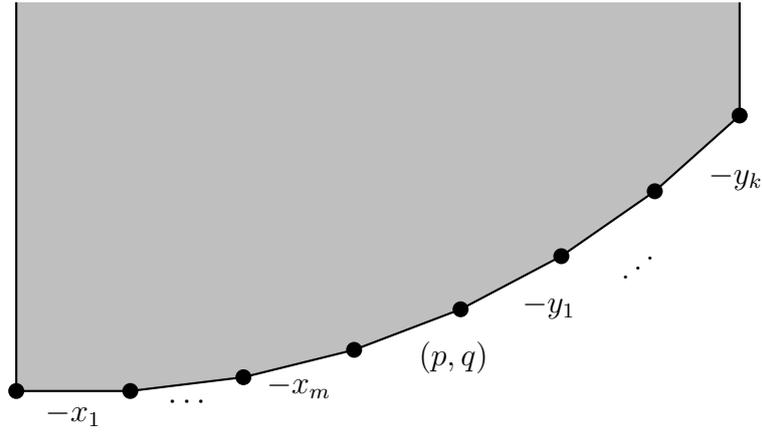

Denote the continued fractions \(\frac{p}{q} = [x_1, \ldots, x_m]\) and
\(\frac{p}{p-r} = [y_1, \ldots, y_k]\). We use the Hirzebruch-Jung convention
for continued fractions, meaning:
\[
    [x_1,\ldots,x_m] := x_1 - \frac{1}{x_2 - \frac{1}{\ddots -\frac{1}{x_m}}}.
\]
Recall that a zero continued fraction (ZCF) is one which
evaluates to zero:
\[
    [a_1,\ldots,a_l] = 0.
\]
Blowing up \([a_1,\ldots,a_l]\) corresponds to replacing it with any one of
the following:
\[
    [1,a_1+1,\ldots,a_l],\, [a_1,\ldots,a_i+1,1,a_{i+1}+1,\ldots,a_l],
    \text{ or } [a_1,\ldots,a_l+1,1],
\]
all of which are ZCFs themselves \cite[Example~9.10]{evansLTF}. The reverse
operation is called \emph{blow down}, and any ZCF admits a blow down to
\([1,1]\) \cite[Lemma~9.11]{evansLTF}. The following lemma collects some facts
that we will use when constructing the compactification. 
\begin{lemma}[]
    \label{lem:ZCF}
    \begin{enumerate}
	\item The continued fraction \(\chi = [x_1, \ldots, x_m, 1, y_1,
	    \ldots, y_k]\) is a ZCF. Then, since \(x_i,y_i \ge 2\),  there
	    is a \emph{unique} sequence of blow ups from the ZCF \([1,1]\) to
	    \(\chi\).
	\item The symplectic sphere corresponding to the edge marked
	    \((p,q)\) in the polygon of Figure~\ref{fig:minRes} has self
	    intersection \(-1\).
    \end{enumerate}
\end{lemma}
\begin{proof}
    (1). The fact that \(\chi = 0\) follows from a simple computation using the
relation \(qr \equiv 1 \mod n\). Since
\[
    \chi = x_1 - \frac{1}{x_2 - \frac{1}{\ddots
	-\frac{1}{1-[y_1,\ldots,y_k]^{-1}}}},
\]
we calculate that
\[
    1-[y_1,\ldots,y_k]^{-1} = 1 - \frac{p-r}{p} = \frac{r}{p}.
\]
Since the continued fraction of \(\frac{p}{r}\) is just the reverse of that of
\(\frac{p}{q}\), we find that
\[
    \frac{1}{1 - [y_1,\ldots,y_k]^{-1}} = \frac{p}{r} = [x_m,\ldots,x_1],
\]
which we use to see
\[
    \chi = x_1 - \frac{1}{x_2 - \frac{1}{\ddots - \frac{1}{x_m -
    [x_m,\ldots,x_1]}}}.
\]
Since
\[
    x_m - [x_m,\ldots,x_1] = \frac{1}{[x_{m-1},\ldots,x_1]},
\]
it follows that \(\chi = 0\).

As is well-known, see for example \cite[Lemma~9.11]{evansLTF},
every ZCF can be obtained from \([1,1]\) via iterated blow up, and the
uniqueness follows from the observation that \(\chi\) has a unique entry equal
to \(1\). Therefore, the initial blow up must be \([1,1]
\to [2,1,2]\), and subsequent blow ups must be adjacent to the unique \(1\)
entry in the ZCF.\footnote{For example, \([2,1,2] \to [3,1,2,2]\) is
permissible, whilst \([2,1,2] \to [1,3,1,2]\) isn't.}

(2). We prove this by direct computation. Consider the part of the polygon
shown in Figure~\ref{fig:pqNeighbourhoodAndCyclicQuot}(a).
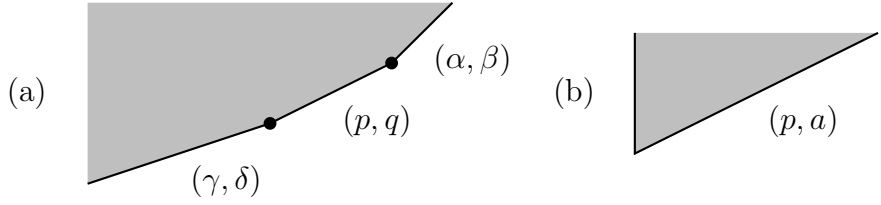
\begin{figure}\centering \begin{tikzpicture}[
   moment polytope, 
   scale=0.8,
   ]
   \draw (-1,1.5) node {(a)};
   \fill (0,0) -- ++(3,1) -- ++(2,1) -- ++(1,1) -- (0,3) -- cycle;
   \draw (0,0) -- node [below right] {\((\gamma,\delta)\)} ++(3,1)
       -- node [below right] {\((p,q)\)} ++(2,1)
       -- node [below right] {\((\alpha,\beta)\)} ++(1,1);
   \fill[black] (3,1) circle [radius=3pt] ++(2,1) circle [radius=3pt];
   \begin{scope}[shift={(9,0.5)},scale=2]
       \draw (-0.5,0.5) node {(b)};
       \filldraw (0,1) -- (0,0) -- node [below right] {\((p,a)\)} (2,1);
   \end{scope}
\end{tikzpicture}
    \caption{
        (a) A neighbourhood of the edge \((p,q)\) in the polygon of
	Figure~\ref{fig:minRes}. (b) the moment polygon \(\pi(p,a)\).
    }
    \label{fig:pqNeighbourhoodAndCyclicQuot}
\end{figure}
Recall that each of the vertices were the result of resolving a singularity 
corresponding to the polygon \(\pi(p,a)\) in
Figure~\ref{fig:pqNeighbourhoodAndCyclicQuot}(b)
for some \(0 < a < p\). In particular, let \(A \in SL_2(\Z)\) be the unique
matrix that maps a vertex of the form shown in the right panel of
Figure~\ref{fig:pqNeighbourhoodAndCyclicQuot} to the \(\frac{1}{p}(1,p-r)\)
vertex in Figure~\ref{fig:pqPolytope}. That is, \(A\) satisfies
\[
    (0,1)A = (-p,-q), \qquad\text{and}\qquad (0,1)A^{-1} = (p,a).
\]
Then, since \((\alpha,\beta) = (1,0)A\) we have that\footnote{Indeed,
    \[
        A^{-1} = \begin{pmatrix} * & * \\ p & a \end{pmatrix}
    \]
    so that
    \[
        \alpha = (1,0)A \begin{pmatrix} 1 \\ 0\end{pmatrix}
        = (1,0) \begin{pmatrix} a & * \\ -p & * \end{pmatrix}
        \begin{pmatrix} 1 \\ 0 \end{pmatrix} = a.
\]} \(0 < \alpha=a< p\). A similar argument shows that \(0 < \gamma < p\).

Both vertices in Figure~\ref{fig:pqNeighbourhoodAndCyclicQuot}(a) are Delzant,
which means
\begin{equation}
    \label{eq:delzantCorners}
    \begin{aligned}
        -\delta p + \gamma q &= \det \begin{pmatrix} p & q \\ -\gamma & -\delta
    	\end{pmatrix} = 1, \text{ and}\\
        -\alpha q + \beta p &= \det \begin{pmatrix} \alpha & \beta \\ -p & -q
    	\end{pmatrix} = 1.
    \end{aligned}
\end{equation}
Reducing modulo \(p\) yields \(\gamma \equiv -\alpha \mod p\), which combined
with the fact that \(0 < \alpha,\gamma < p\), means \(\gamma = p-\alpha\).
Lemma~3.20 of \cite{evansLTF}, asserts that the self intersection of the
\((p,q)\)-edge is given by
\[
    \det \begin{pmatrix} -\gamma & -\delta \\ \alpha & \beta \end{pmatrix} = 
	-\gamma\beta + \delta\alpha.
\]
Therefore, combined with equations \eqref{eq:delzantCorners}, this implies that
\[
    -\gamma\beta + \delta\alpha = \frac{-\gamma - \alpha}{p} = -1,
\]
which completes the proof.
\end{proof}

\begin{remark}
    A more conceptual way to prove statement (2) above is to use the fact that
there is a unique series of blow ups
\[
    [1,1] \to [x_1, \ldots, x_m, 1, y_1, \ldots, y_k]
\]
and perform the corresponding symplectic blow ups as in
Figure~\ref{fig:11polygonBlowUp}. Indeed, Example~9.10 and Corollary~9.13 of
\cite{evansLTF} show that the combinatorics of these two procedures are the
same.
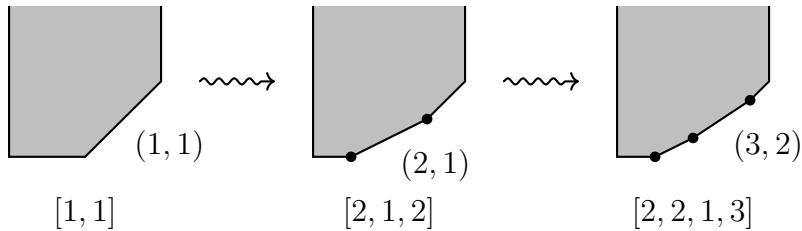
\begin{figure}[ht]\centering \begin{tikzpicture}[
    moment polytope,
]
    \filldraw (0,2) -- (0,0) -- (1,0) -- node [below right] {\((1,1)\)} (2,1)
	-- (2,2);
    \draw[->,decorate,
	decoration={snake,amplitude=.4mm,segment length=2mm,post length=1mm}]
	(2.5,1) -- (3.5,1);
    \filldraw (4,2) -- (4,0) -- (4.5,0) -- node [below right] {\((2,1)\)}
	(5.5,0.5) -- (6,1) -- (6,2);
    \draw[->,decorate,
	decoration={snake,amplitude=.4mm,segment length=2mm,post length=1mm}]
	(6.5,1) -- (7.5,1);
    \filldraw (8,2) -- (8,0) -- (8.5,0) -- (9,0.25) -- node [below right]
	{\((3,2)\)} (9.75,0.75) -- (10,1) -- (10,2);
    \fill[fill=black] (4.5,0) circle [radius=2pt]
	(5.5,0.5) circle [radius=2pt];
    \fill[fill=black] (8.5,0) circle [radius=2pt]
	(9,0.25) circle [radius=2pt]
	(9.75,0.75) circle [radius=2pt];
    \draw (1,-0.75) node {\([1,1]\)} (5,-0.75) node {\([2,1,2]\)}
	(9,-0.75) node {\([2,2,1,3]\)};
\end{tikzpicture}
    \caption{
	An example of blowing up the toric manifold associated to the ZCF
	\([1,1]\). The symplectic cuts made are labelled with the primitive
	integer direction they point in. The vertices are marked for clarity.
    }
    \label{fig:11polygonBlowUp}
\end{figure}
Thus, we obtain two families of toric moment polygons, one from
resolving all the polygons \(\Pi(p,q)\), and the other from blowing
up ZCFs. Note in particular that the second family of polygons will
necessarily have a unique edge with self intersection \(-1\). An inductive
argument on the length of the blow-ups of the ZCFs shows that these families
coincide, from which it follows that the \((p,q)\) edge in
Figure~\ref{fig:minRes} must be a \(-1\)-curve, since all the other edges are
of self intersection at most \(-2\).
\end{remark}

\subsection{Compactification construction}
\label{sec:compactification}

The complex projective plane \(\CP^2\) equipped with the Fubini-Study form
\(\omega_\mathrm{FS}\) has a Hamiltonian torus action with moment polygon
Figure~\ref{fig:CP2polygon}(a).
\begin{figure}\centering \begin{tikzpicture}[
    moment polytope,
]
    \draw (-1,-1.5) node {(a)};
    \filldraw (0,0) -- (3,-3) -- (3,0) -- cycle;
    \begin{scope}[shift={(9,-3)}]
        \draw (-4,1.5) node {(b)};
	\coordinate (X) at (0,1.7);
	\fill (0,0) -- (1,1.5) -- (1,3) -- (-1,3) -- (-1,0.5) -- cycle;
	\draw (-1,0.5) -- (0,0) -- (1,1.5) (1,3) -- (-1,3);
	\draw (-1,0.5) -- (-1,3) (1,1.5) -- (1,3);
	\fill[color=black]
	    (1,1.5) circle [radius=3pt] node [right] {\(\frac{1}{p}(1,q)\)}
	    (-1,0.5) circle [radius=3pt] node [left] {\(\frac{1}{p}(1,p-r)\)};
	\draw[style=dotted] (0,0) -- (X) node {\(\times\)};
	\draw[style=dotted] (X) -- +(0,0.5) node {\(\times\)};
    \end{scope}
    
\end{tikzpicture}
    \caption{
        (a) The moment polygon of \((\CP^2,\omega_\mathrm{FS})\). (b)
        The almost toric base diagram of the compactification \(X\) (modulo
        resolving the marked singularities as in Figure~\ref{fig:minRes}).
    }
    \label{fig:CP2polygon}
\end{figure}
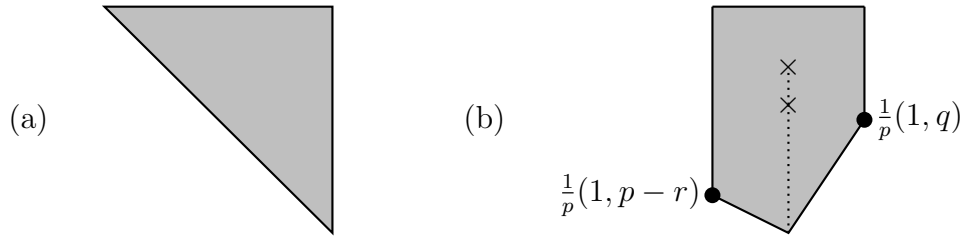
The next result shows that we can perform a series of alterations to this
polygon, called symplectic cuts and non-toric blow ups, to obtain a
symplectic compactification of \(B_{d,p,q}\) determined up to symplectic
deformation\footnote{The reason this is not ``up to symplectomorphism'' is
    because of the freedom of choice over the affine lengths of the edges
introduced by resolving the singularities. However, this is largely
inconsequential.} by the polygon Figure~\ref{fig:CP2polygon}(b). Performing a
symplectic cut to a Delzant vertex in a moment polygon is characterised by
Figure~\ref{fig:symplecticCut}(a).
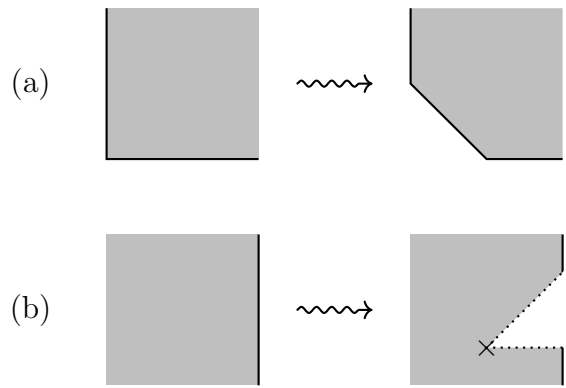
\begin{figure}\centering \begin{tikzpicture}[
    moment polytope
]
    \draw (-1,1) node {(a)};
    \fill (0,2) -- (0,0) -- (2,0) -- (2,2) -- cycle;
    \draw (0,2) -- (0,0) -- (2,0);
    \draw[->,decorate,
	decoration={snake,amplitude=.4mm,segment length=2mm,post length=1mm}]
	(2.5,1) -- (3.5,1);
    \begin{scope}[shift={(4,0)}]
        \fill (0,2) -- (0,1) -- (1,0) -- (2,0) -- (2,2) -- cycle;
        \draw (0,2) -- (0,1) -- (1,0) -- (2,0);
    \end{scope}
    \begin{scope}[shift={(0,-3)}]
        \draw (-1,1) node {(b)};
        \fill (0,0) -- (2,0) -- (2,2) -- (0,2) -- cycle;
        \draw (2,0) -- (2,2);
    \draw[->,decorate,
        decoration={snake,amplitude=.4mm,segment length=2mm,post length=1mm}]
        (2.5,1) -- (3.5,1);
    \begin{scope}[shift={(4,0)}]
        \fill (0,0) -- (2,0) -- (2,2) -- (0,2) -- cycle;
        \draw (2,0) -- (2,2);
        \fill[white] (2.1,0.5) -- (1,0.5) -- (2.1,1.6);
        \draw[dotted] (2,0.5) -- (1,0.5) node {\(\times\)} -- (2,1.5);
    \end{scope}
    \end{scope}
\end{tikzpicture}
    \caption{
        (a) A symplectic cut at a Delzant vertex. (b) A non-toric blow up
        at an edge of a moment polygon.
    }
    \label{fig:symplecticCut}
\end{figure}
The non-toric blow up transforms a polygon as in
Figure~\ref{fig:symplecticCut}(b). Both of these operations correspond to the
symplectic blow up. The reader may consult Sections~4.4 and 9.1 of
\cite{evansLTF} for further details.

\begin{lemma}[]
    \label{lem:compactify}
    There exists a number\footnote{The actual value of \(\mu\) is
    irrelevant, but one can calculate it to be \(\sqrt{2}\) times the
    Euclidean distance of the diagonal edge from the top-right vertex in
the polygon shown in Figure~\ref{fig:CP2polygon}(a).}
\(\mu > 0\) and a symplectic compactification \((X,\omega) = (X_{d,p,q},
\omega_{d,p,q})\) of \(B_{d,p,q}\) such that \((X,\omega)\) is a blow up of
\((\CP^2,\mu\omega_\mathrm{FS})\). The blow ups compute a basis
\(\{H,\fibrationSectionL,E_i,\mathcal{E}_j \mid 0 \le i \le n, 1 \le j \le
d\}\) of \(H_2(X)\) such that the following holds:
\begin{enumerate}
    \item \(H\) is the class of a line in \(\CP^2\), \fibrationSectionL{} and
        \(E_i\) are exceptional curves corresponding to symplectic cuts, and
        \(\mathcal{E}_j\) are exceptional curves corresponding to non-toric
        blow ups;
    \item \(n=m+k-1\) is equal to the number of blow ups
        \[
            [1,1] \to [x_1, \ldots, x_m, 1, y_1, \ldots, y_k]
        \]
        determined by Lemma~\ref{lem:ZCF};
    \item the \(\omega\)-areas of the non-toric curves \(\mathcal{E}_j\) are
	all equal, that is, there exists \(l > 0\) such that, for each \(1 \le
	j \le d\), \[ \omega(\mathcal{E}_j) = l; \] and,
    \item define the \emph{homology class of a smooth fibre} by
        \(F := H - \fibrationSectionL{}\), then \(\omega(F) = 2l\), and
        \(\omega(F) > \omega(E_n)\).
\end{enumerate}
\end{lemma}
\begin{proof}
    Consider the almost toric base diagram of \(B_{d,p,q}\) shown in\linebreak
Figure~\ref{fig:BdpqTransformation}(a).\footnote{Note that the figures drawn
here are for the case \((d,p,q) = (2,2,1)\) but the process works for all
triples.} Choose \(A \in SL_2(\Z)\) such that \((0,1)A = (p,q)\). 
\begin{figure}[ht]\centering
    \begin{tikzpicture}[%
    moment polytope,
    scale=0.67,
    >=Rays,
]
    \coordinate (O) at (15,0);
    \draw (4,-1) node {(a)};
    \filldraw[>=To,<->] (0,3) -- node [left] {\((0,1)\)} (0,0) --
	node [below right] {\((dp^2, dpq - 1)\)} (8,3) ;
    \draw[style=dotted] (0,0) -- (2,1) node {\(\times\)};
    \draw[style=dotted] (2,1) -- (3,1.5) node {\(\times\)} node [above] {\((p,q)\)};
    \begin{scope}[shift={(O)}]
        \draw (-2,-1) node {(b)};
	\filldraw[>=To,<->] (-6,3) -- node [below left] {\((-p,r)\)} (0,0) --
	    (2,3) node [above] {\((p, dp - r)\)};
	\draw[color=black, style=dashed, >=To, <->]
	    (-3,3) node [above] {\((-1,1)\)} -- (0,0)
	    -- node [below right] {\((1,1)\)} (3,3);
	\draw[style=dotted] (0,0) -- (0,1) node {\(\times\)};
	\draw[style=dotted] (0,1) -- (0,1.5) node {\(\times\)} node [above] {\((0,1)\)};
    \end{scope}
\end{tikzpicture}
    \caption{
	The figure from left to right shows the effect of the transformation
        \(A^{-1}\). The dashed rays in (b) are included to show the extremes
        that the solid rays cannot cross since \(0 < r < p\) and \(d > 1\).
    }
    \label{fig:BdpqTransformation}
\end{figure}
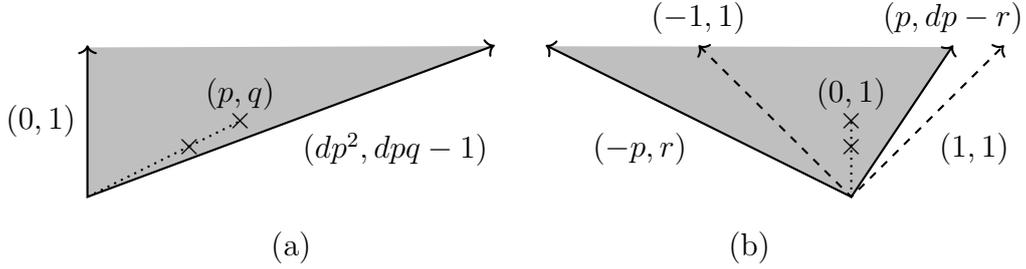
Writing \(A = \begin{pmatrix} r &s \\ p & q\end{pmatrix}\), the condition
\(\det A = 1\) ensures that \(qr \equiv 1 \mod p\) and we may choose \(A\)
such that \(0 < r < p\). Applying \(A^{-1}\) to the base diagram then
yields Figure~\ref{fig:BdpqTransformation}(b). We make a
series of symplectic cuts, shown as dash dot lines in
Figure~\ref{fig:BdpqCuts}: first a horizontal one at some level above the
focus-focus critical values; and then vertical cuts on either
side.\footnote{As the diagram makes clear, the vertical cuts can be chosen so
that \(\omega(F) = 2l\).} Doing so introduces two cyclic quotient
singularities, which are marked in the figure. Taking the minimal resolution
yields the desired compactification \(X\).
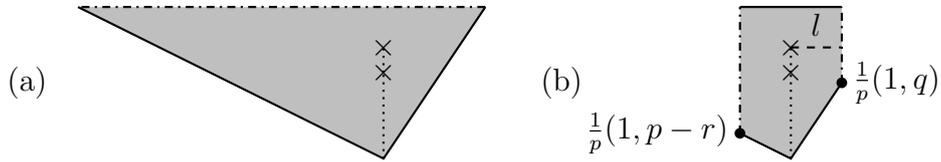
\begin{figure}[ht]\centering
\begin{tikzpicture}[
    moment polytope,
    scale=0.67,
]
    \draw (-7,1.5) node {(a)};
    \coordinate (X) at (0,1.7);
    \filldraw (-6,3) -- (0,0) -- (2,3);
    \draw[style=dash dot] (2,3) -- (-6,3);
    \draw[style=dotted] (0,0) -- (X) node {\(\times\)};
    \draw[style=dotted] (X) -- +(0,0.5) node {\(\times\)};

    \begin{scope}[shift={(8,0)}]
        \draw (-4.5,1.5) node {(b)};
	\coordinate (X) at (0,1.7);
	\fill (0,0) -- (1,1.5) -- (1,3) -- (-1,3) -- (-1,0.5) -- cycle;
	\draw (-1,0.5) -- (0,0) -- (1,1.5) (1,3) -- (-1,3);
	\draw[style=dash dot] (-1,0.5) -- (-1,3) (1,1.5) -- (1,3);
	\fill[color=black]
	    (1,1.5) circle [radius=3pt] node [right] {\(\frac{1}{p}(1,q)\)}
	    (-1,0.5) circle [radius=3pt] node [left] {\(\frac{1}{p}(1,p-r)\)};
	\draw[style=dotted] (0,0) -- (X) node {\(\times\)};
	\draw[style=dotted] (X) -- +(0,0.5) node {\(\times\)};
	\draw[style=dashed] (X) ++(0,0.5) -- node [above] {\(l\)} +(1,0);
    \end{scope}
\end{tikzpicture}
    \caption{
	Performing symplectic cuts (dash dot lines) to compactify \(B_{d,p,q}\). After
    performing the vertical cuts, we obtain a manifold with two cyclic
    quotient singularities, which are marked as bullets in the figure, along
    with their type. The dashed line labelled \(l\) indicates the affine displacement of
    the right hand vertical cut from the monodromy eigenline.
    }
    \label{fig:BdpqCuts}
\end{figure}

    Since \(X\) contains a symplectic sphere with non-negative self
intersection (take the top edge in the base diagram, for example), a
well-known result of McDuff \cite{mcduff90RationalRuled} shows that \(X\) is
symplectomorphic to a blow-up of \(\CP^2\). Thus, the 
remainder of the proof is to choose a sequence of blow-ups that satisfies the
claims of the lemma. To this end, we rotate the branch cut by \(90^\circ\)
anticlockwise (see \cite[Section~7.2]{evansLTF}), to obtain\footnote{Note that
    we continue to draw the base diagram of the singular manifold
    (Figure~\ref{fig:BdpqCuts}(b)) prior to taking the minimal resolution.
    This is to make the pictures easier to draw and interpret. Drawing the
resolved manifold would involve long chains of edges as in
Figure~\ref{fig:minRes}, the precise directions and lengths of which add
nothing to the proof, so we omit them.} Figure~\ref{fig:nonToricBlowDown}(a).
The \(d\) bites in the diagram correspond to the exceptional loci of \(d\)
non-toric blow ups. In particular, they represent a collection of \(d\)
disjoint symplectic \(-1\)-spheres. Moreover, each of them has \(\omega\)-area
equal to \(l\), since the focus-focus critical values are all the same affine
length \(l\) from the right-hand vertical edge. Label the homology classes of
these spheres by \(\mathcal{E}_j\) and blow them down to obtain
Figure~\ref{fig:nonToricBlowDown}(b).
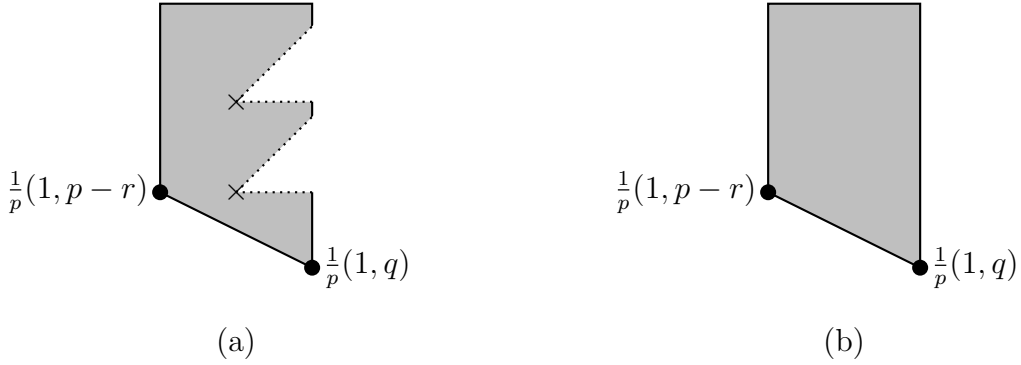
\begin{figure}[ht]\centering
    \begin{tikzpicture}[moment polytope, >=Rays]
        \draw (0,-1.5) node {(a)};
        \filldraw (-1,0.5) -- (1,-0.5) -- (1,3) -- (-1,3) -- cycle;
	\fill[color=black]
	    (-1,0.5) circle [radius=3pt] node [left] {\(\frac{1}{p}(1,p-r)\)}
	    (1,-0.5) circle [radius=3pt] node [right] {\(\frac{1}{p}(1,q)\)};
	\fill[color=white]
	    (1.1,0.5) -- (0,0.5) -- (1.1,1.6)
	    (1.1,1.7) -- (0,1.7) -- (1.1,2.8);
	\begin{scope}[style=dotted]
	    \draw (1,0.5) -- (0,0.5) node {\(\times\)};
	    \draw (1,1.7) -- (0,1.7) node {\(\times\)};
	    \draw (0,0.5) -- (1,1.5) (0,1.7) -- (1,2.7);
	\end{scope}
	\begin{scope}[shift={(8,0)}]
        \draw (0,-1.5) node {(b)};
	\filldraw (-1,0.5) -- (1,-0.5) -- (1,3) -- (-1,3) -- cycle;
	\fill[color=black]
	    (-1,0.5) circle [radius=3pt] node [left] {\(\frac{1}{p}(1,p-r)\)}
	    (1,-0.5) circle [radius=3pt] node [right] {\(\frac{1}{p}(1,q)\)};
	\end{scope}
    \end{tikzpicture}
    \caption{
	Non-toric blow downs.
    }
    \label{fig:nonToricBlowDown}
\end{figure}
Observe that the singularities of the resultant manifold are modelled on a
reflection about the vertical axis of those of\footnote{This is not a typo;
reflection in the vertical acts as inverse modulo \(p\) on the singularity:
\(\frac{1}{p}(1,q) \mapsto \frac{1}{p}(1,r)\).} \(\Pi(p,r)\). Therefore, the
minimal resolution is given by (the reflection of) Figure~\ref{fig:minRes}.
Then Lemma~\ref{lem:ZCF} implies that there is a unique sequence of toric
boundary blow downs to the polygon in
Figure~\ref{fig:blowDownToP2}(a), which is just \(\CP^2 \# 2\overline{\CP^2}\).
\begin{figure}[ht]\centering \begin{tikzpicture}[moment polytope]
    \draw (-2,1.5) node {(a)};
    \filldraw[shift={(0,0.5)}] (0,0)
	-- node [below] {\(E_0\)} (1,0)
	-- node [right] {\(H - E_0\)} (1,2)
	-- node [above] {\(F = H - \fibrationSectionL{}\)} (-1,2)
	-- node [left]  {\(\fibrationSectionL{}\)} (-1,1)
	-- cycle;    
    \draw (4,1.5) node {(b)};
    \filldraw[shift={(8,0)}] (0,0) -- (0,3) -- node [above] {\(H\)} (-3,3) --
	cycle;
\end{tikzpicture}
    \caption{
        (a) The toric manifold obtained by blowing down \(X_{d,p,q}\). This is
	exactly \(\CP^2 \# 2\overline{\CP^2}\). The edges are labelled with
	the corresponding homology classes. (b) The moment polygon of
	\(\CP^2\), obtained by blowing down the edges labelled \(\fibrationSectionL{}\) and
	\(E_0\).
    }
    \label{fig:blowDownToP2}
\end{figure}
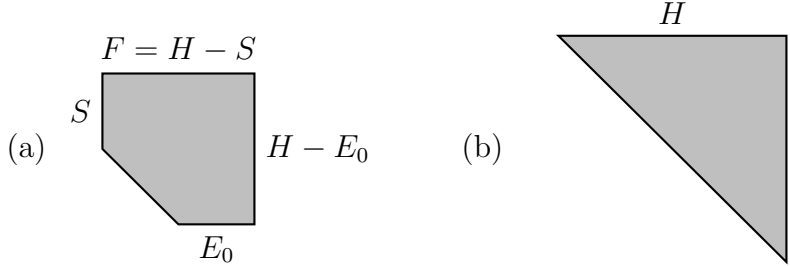
Reversing this process, that is blowing up instead of down, yields the desired
sequence of blow ups from \((\CP^2,\mu\omega_\mathrm{FS})\) to \((X,\omega)\).
Labelling the homology classes of the  toric boundary blow ups as
\(\fibrationSectionL\) and \(E_i\), and the non-toric ones as
\(\mathcal{E}_j\), we obtain the claimed basis
\[
    \{H,\fibrationSectionL,E_i,
    \mathcal{E}_j \mid 0 \le i \le n, 1 \le j \le d\}.
\]

We have established properties (1)--(3), so it remains to show (4). The
inequality \(\omega(H - \fibrationSectionL{}) > l\) holds as \(\omega(H -
\fibrationSectionL{})\) is given by the affine length of the horizontal edge of
the polygon in Figure~\ref{fig:BdpqCuts}(b). The second
inequality follows from considering Figure~\ref{fig:affineLengths},
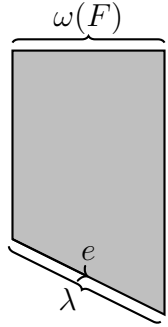
\begin{figure}[ht]\centering \begin{tikzpicture}[moment polytope]
    \filldraw (-1,0.5) -- (1,-0.5) -- (1,3) -- (-1,3) -- cycle;
    \draw[decorate, decoration={brace,mirror,raise=3pt}] (1,3) -- node
	[above=2pt] {\(\omega(F)\)}
	(-1,3);
    \draw[->] (-1,0.5) -- (0,0) node [above=2pt] {\(e\)};
    \draw[decorate, decoration={brace,mirror,raise=3pt}] (-1,0.5) -- node [below left]
	{\(\lambda\)} (1,-0.5);
\end{tikzpicture}
    \caption{
        The moment polygon of Figure~\ref{fig:nonToricBlowDown}(b) with affine
        lengths indicated. The edge \(e\) points in the \((p,-r)\) direction.
        Since \(E_n\) is obtained by symplectically cutting the bottom edge,
        we have the inequality \(\omega(E_n) < \lambda\).
    }
    \label{fig:affineLengths}
\end{figure}
which is the same polygon as in Figure~\ref{fig:nonToricBlowDown}(b). The edge
\(E_n\) is obtained by symplectic cutting the bottom edge \(e\) of
Figure~\ref{fig:affineLengths}, so its affine length \(\omega(E_n)\) is
strictly smaller than that of \(e\), which we denote by \(\lambda\). Since
\(e\) points in the direction \((p,-r)\), we must have that \(\lambda{}p =
\omega(H-\fibrationSectionL{})\), since the top edge is horizontal. As \(p
> 1\), we obtain that \(\lambda < \omega(H-\fibrationSectionL{})\). This
completes the proof.
\end{proof}

\begin{remark}
    The toric boundary \(D \subset X\) is a \emph{symplectic divisor} of
\((X,\omega)\), meaning that is a union of (real) codimension \(2\) symplectic
submanifolds of \((X,\omega)\). Furthermore, it is a cycle of transversely
intersecting spheres, which we represent by the dual intersection
graph of its components in Figure~\ref{fig:dualIntGraphD}(a).
\begin{figure}[ht]\centering
\begin{tikzpicture}[
    colorstyle/.style={
	circle, draw=black, fill=black, thick, inner sep=0pt, minimum size=4pt,
    outer sep=0pt},
    thick,
    ]
    \tikzmath{
	\sep=0.5;
    }
    \node (A) at (0,5) [colorstyle, label=above:\(F\)]{};
    \node (B) at ($ (A) + (1,-\sep)$) [colorstyle,
	label=above right:\(\fibrationSectionR{}\)]{};
    \node (bl) at ($ (B) - (0,\sep) $) [colorstyle]{};
    \node (dotsR) at ($ (bl) + (0,-\sep)$) [inner sep=0pt, outer sep=0pt,
	label=right:\(\vdots\)]{};
    \node (b1) at ($ (dotsR) + (0,-\sep) $) [colorstyle]{};
    \node (E) at ($ (b1) + (-1,-\sep) $) [colorstyle, label=below:\(E_n\)]{};
    \node (am) at ($ (E) + (-1,\sep) $) [colorstyle]{};
    \node (dotsL) at ($ (am) + (0,\sep) $) [inner sep=0pt, outer sep=0pt,
	label=left:\(\vdots\)]{};
    \node (a1) at ($ (dotsL) + (0,\sep) $) [colorstyle]{};
    \node (E-1) at ($ (a1) + (0,\sep) $) [colorstyle, label=above
	left:\(\fibrationSectionL{}\)]{};
    \draw (A) -- (B) -- (bl) -- ($ (dotsR) + (0,0.2)$) ++ (0,-0.4) -- (b1) --
        (E) -- (am) -- ($ (dotsL) - (0,0.2) $) ++ (0,0.4) -- node [left=15pt]
        {(a)} (a1) -- (E-1) -- (A);
    \node (D) at (0,3.75) {\(D =\)};

    \begin{scope}[shift={(6,4)}]
        \draw (-2,0) node {(b)};
	\node (a) at (0,0) [colorstyle]{};
	\node (b) at (0.5,0) [inner sep=0pt, outer sep=0pt] {\(\cdots\)};
	\node (c) at (1,0) [colorstyle]{};
	\node (d) at (2,0) [colorstyle, label=below:\(E_n\)]{};
	\node (e) at (3,0) [colorstyle]{};
	\node (f) at (3.5,0) [inner sep=0pt, outer sep=0pt] {\(\cdots\)};
	\node (g) at (4,0) [colorstyle]{};
        \draw node [left=5pt] {\(D_\infty = \)} (a) -- ($ (b) - (0.3,0) $) ++
            (0.6,0) -- (c) -- (d) -- (e) -- ($ (f) - (0.3,0) $) ++ (0.6,0) --
            (g);
        \draw[decorate, decoration={brace,raise=7pt}] (a) -- node [above=8pt]
            {\(m\)} (c);
        \draw[decorate, decoration={brace,raise=7pt}] (e) -- node [above=8pt]
            {\(k\)} (g);
    \end{scope}
    
\end{tikzpicture}
    \caption{
        (a) The dual intersection graph of the compactifying divisor \(D
	\subset X\). Vertices are labelled with the respective homology
	classes of the component symplectic spheres. The integers \(m\)
	and \(k\) correspond to the lengths of the continued fractions of
	\(\frac{p}{p-r}\) and \(\frac{p}{q}\) respectively, as in
	Lemma~\ref{lem:ZCF}. (b) The subgraph consisting of spheres
	introduced by resolving the singularities in the compactification
	process.
    }
    \label{fig:dualIntGraphD}
\end{figure}
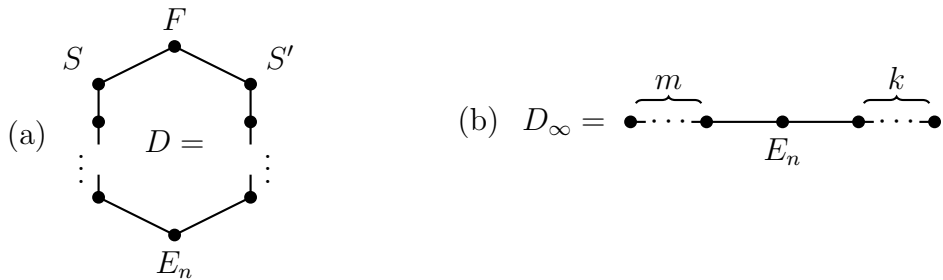
Removing the components \(\{\fibrationSectionL{}, F, \fibrationSectionR{}:= H
- E_0 - \sum_{j=1}^d \mathcal{E}_j\}\) yields a subgraph representing the part
of the toric boundary obtained by resolving the singularities in
Figure~\ref{fig:nonToricBlowDown}, which we will call \(D_\infty\).
\end{remark}

\subsection{Lagrangian spheres in \(B_{d,p,q}\)}
\label{sec:spheresInBdpq}

We classify the homology classes in \(H_2(B_{d,p,q};\Z) \cong
\Z^{d-1}\) that support Lagrangian spheres. The point of this is to show that,
in the compactification \(X\), any Lagrangian sphere \(L \subset B_{d,p,q}
\subset X\) is homologous to \(e - e'\) where \(e,e' \in H_2(X;\Z)\) are
classes represented by \(J\)-holomorphic \(-1\)-curves. This will be important
for our neck stretching analysis performed in Section~\ref{ch:neckStretch}.

The homotopy type of \(B_{d,p,q}\) is a \((p,q)\)-pinwheel wedged with
\(d-1\) spheres \cite[Lemma~7.11]{evansLTF}. Moreover, the standard
\(A_{d-1}\) configuration of Lagrangian spheres is a generating set of the
second homology:
\[
    H_2(B_{d,p,q}) = H_2(B_{d,p,q};\Z) = \Z\langle L_i : 1 \le i \le d-1
    \rangle,
\]
Therefore, the intersection form with respect to this basis is
\[
    Q := \begin{pmatrix} -2 & 1 & 0 & \cdots & \cdots & \cdots & 0 \\
	1 & -2 & 1 & 0 & \cdots & \cdots & 0 \\
	0 & 1 & -2 & 1 & \cdots & \cdots & 0 \\
	\vdots & \vdots & \ddots & \ddots & \ddots & \vdots & \vdots \\
	0 & 0 & \cdots & \cdots & 1 & -2 & 1 \\
	0 & 0 & \cdots & \cdots & 0 & 1 & -2
	\end{pmatrix}.
\]
The following argument is standard, and we include it only for completeness.
\begin{lemma}[]
    \label{lem:homologyClassOfSpheresBdpq}
    The set \(\mathcal{L}\) of homology classes represented by Lagrangian
\linebreak spheres in \(H_2(B_{d,p,q})\) is given by
\begin{equation}
    \label{eq:lagrHomology}
    \mathcal{L} = \{ \pm(L_i + L_{i+1} + \ldots + L_{i+k}) \mid i \ge 1,
    k \ge 0, i+k \le d-1\}.
\end{equation}
\end{lemma}
\begin{proof}
    Weinstein's Lagrangian tubular neighbourhood theorem
\cite{weinstein71Lagrangians} ensures that any Lagrangian sphere has
self-intersection \(-2\). So, our first task is to determine the set of
\(-2\)-classes in \(H_2(B_{d,p,q})\). Consider the vector space
\[
    E := \left\{x = (x_1,\ldots,x_{d}) \in \R^d \,\left|\, \sum_{i=1}^d x_i = 0
    \right.\right\}
\]
equipped with the restriction of the standard Euclidean inner product
\(g_E := \langle \cdot,\cdot \rangle|_E\). Consider the standard Euclidean
basis \(e_i\) of \(\R^d\) and define a map \(\varphi : H_2(B_{d,p,q}) \otimes
\R \to E\) by
\[
    \varphi(L_i) = e_i - e_{i+1}.
\]
This is a \(\Z\)-linear isometric isomorphism \((H_2(B_{d,p,q}) \otimes \R,
-Q) \cong (E,g_E)\). Consequently, it carries the set of \(-2\)-classes
bijectively onto
\[
    R := \{x \in E \mid g_E(x,x) = 2\},
\]
which is the set of vectors with exactly two non-zero entries: one of which is
\(1\), and the other \(-1\). In other words,
\[
    R = \{e_i - e_j \mid i \ne j\}.
\]
Observe that this is just the set of roots of the Lie algebra
\(\mathfrak{sl}_d\C\), see \cite[\S{}15.1]{fultonHarris} for example.
Therefore, the set of \(-2\)-classes is \(\varphi^{-1}(R)\), and so
\[
    \mathcal{L} \subset \varphi^{-1}(R).
\]

It remains to show that each homology class in \(\varphi^{-1}(R)\) supports a
Lagrangian sphere. First, recall that the Picard-Lefschetz formula states
that, for the 4-dimensional Dehn twist \(\tau_L\) associated to a Lagrangian
2-sphere \(L\), the action of \(\tau_L\) on homology is given by
\[
    (\tau_L)_\ast(A) = \begin{cases}
        A + (A \cdot L)L, &\text{ if }A \in H_2,\\
	A,		  &\text{ if }A \in H_k,\,k\ne2,
    \end{cases}
\]
see \cite{seidelLectures} for example.
Next, recall the action of the Weyl group \(W\) on \(R\) is generated by reflections
\(\sigma_\alpha\)
through the hyperplanes orthogonal to the roots \(\alpha \in R\). In
particular, the action is determined by its effect
on the set of simple roots:
\[
    \{\alpha_i := e_i - e_{i+1} = \varphi(L_i) \mid 1 \le i < d\}.
\]
The reflection \(\sigma_{\alpha_i}\) acts as
\[
    \sigma_{\alpha_i}(\alpha_{i+1}) = \alpha_i + \alpha_{i+1} =
    \sigma_{\alpha_{i+1}}(\alpha_i),
\]
and as the identity on non-adjacent roots \(\alpha_j\), \(|j-i| > 1\). Now,
observe that the map \(\varphi\) intertwines the action of the Weyl group on
\(R\) and that of \(\tau_L\) on \(H_2(B_{d,p,q})\). Therefore, as \(W\) acts
transitively on \(R\), so must \(\tau_L\) on \(\varphi^{-1}(R)\). In
particular, each class in \(\varphi^{-1}(R)\) supports a Lagrangian sphere,
and so we conclude that \(\mathcal{L} = \varphi^{-1}(R)\), which finishes the
proof.
\end{proof}

\begin{lemma}[]
    \label{lem:homologyCompactification}
    Under the embedding \(B_{d,p,q} \hookrightarrow X\), the
generators \(L_i\) of the second homology \(H_2(B_{d,p,q})\) are identified
with \(\mathcal{E}_i - \mathcal{E}_{i+1} \in H_2(X)\) (up to sign).
\end{lemma}
\begin{proof}
    Zooming in to Figure~\ref{fig:symplecticCut}(b), one can see that each
sphere \(L_i\) has a neighbourhood \(N_i \subset X_{d,p,q}\) that looks like
that shown in Figure~\ref{fig:sphereNbhd}.
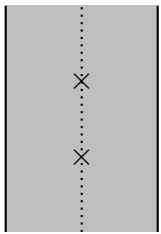
\begin{figure}[ht]\centering \begin{tikzpicture}[
    moment polytope,
    rotate=90
]
    \fill (0,1) -- (0,-1) -- (3,-1) -- (3,1) -- cycle;
    \draw (0,1) -- (3,1) (0,-1) -- (3,-1);
    \draw[dotted] (0,0) -- (1,0) node {\(\times\)} -- (2,0) node {\(\times\)}
	-- (3,0);
\end{tikzpicture}
    \caption{
	A neighbourhood of one of the standard Lagrangian spheres \(L_i
	\subset X\).
    } 
    \label{fig:sphereNbhd}
\end{figure}
The second homology of \(N_i\) is freely generated by \(F,\mathcal{E}_i,
\mathcal{E}_{i+1} \in H_2(X)\), so we must have
\[
    L_i = \lambda F + \mu_i \mathcal{E}_i + \mu_{i+1} \mathcal{E}_{i+1}.
\]
Note that the left edge of the diagram in Figure~\ref{fig:sphereNbhd} is part
of the \(\fibrationSectionL{}\) sphere, which is a non-zero class \(S_i\) in
\(H_2(N_i, \partial N_i)\) satisfying
\[
    F \cdot S_i = 1, \qquad\text{and}\qquad \mathcal{E}_j \cdot S_i = 0,
\]
for \(j = i,i+1\). Therefore, since \(L_i\) is disjoint from \(S_i\), we find
that \(\lambda = 0\). Now, \(L_i\) must be a \(-2\) class, and so
\[
    -2 = L_i^2 = -\mu_i^2 - \mu_{i+1}^2,
\]
which implies that \(|\mu_i| = 1 = |\mu_{i+1}|\). Finally, the Lagrangian
condition ensures that
\[
    0 = \omega(L_i) = \mu_i\omega(\mathcal{E}_i) +
    \mu_{i+1}\omega(\mathcal{E}_{i+1}) = l(\mu_i + \mu_{i+1}),
\]
which shows that the coefficients \(\mu_i\) and \(\mu_{i+1}\) must have
opposite sign. Hence, the result follows.
\end{proof}

Combining the above with Lemma~\ref{lem:homologyClassOfSpheresBdpq}, we have
proved the following:
\begin{corollary}[]
    \label{cor:homologyClassOfSpheresXdpq}
    Any Lagrangian sphere \(L \subset B_{d,p,q} \subset X\) has homology class
of the form
\[
    [L] = \mathcal{E}_i - \mathcal{E}_j,
\]
for some \(1 \le i \ne j \le d\).
\end{corollary}

\subsection{Moduli spaces of fibre curves}
\label{sec:fibreModuliSpaces}

In this section, we will examine the \(J\)-holomorphic curves of \(X\) that
represent the homology class \(F = H-S\).
Consider the moduli space\footnote{The notation used here is adopted from
\cite{mcduffSalamonCurves}. We'll often suppress \(X\) in the notation
\(\mathcal{M}_{0,0}(X,F;J)\).} \(\mathcal{M}_{0,0}(X,F;J)\) of genus 0
\(J\)-holomorphic curves in the class \(F\).
The present goal is to prove that (after removing a single exotic stable
curve) the forgetful map
\begin{equation}
    \label{eq:lefFib}
    \overline{\mathcal{M}}_{0,1}(X,F;J) \to \overline{\mathcal{M}}_{0,0}(X,F;J)
\end{equation}
gives rise to a Lefschetz fibration for all almost complex structures \(J\) in
some suitable subset \(\mathcal{J}(\Delta)\) of the space of compatible
structures \(\mathcal{J}(X,\omega)\), defined as follows.
\begin{definition}
    Let \(\Delta\) be a symplectic divisor in \((X,\omega)\), and define
\[
    \mathcal{J}(\Delta) := \left\{ J \in \mathcal{J}(X,\omega) \, \left| \,
	\parbox{6.5cm}{each irreducible component of \(\Delta\) admits a
	\(J\)-holomorphic representative}\right.\right\}.
\]
\end{definition}

The reason we're interested in this space of almost complex structures is that
the Lagrangian spheres we consider live in the complement of a divisor in
\(X\). Under neck stretching, the almost complex structure only changes in a
small neighbourhood of the Lagrangian, and so, supposing the initial almost
complex structure \(J_0\) is a member of \(\mathcal{J}(\Delta)\), then \(J_t\)
will also be for all \(t \ge 0\).

Consider the divisor \(D'\) obtained by excising the single component
corresponding to the edge of the toric boundary with homology class \(E_n\).
That is, \(D'\) is the subgraph of the dual intersection graph of \(D\)
(Figure~\ref{fig:dualIntGraphD}(a)) consisting of every vertex except the bottom
one. We are primarily concerned with the case where \(\Delta = D'\) in the
above definition. We will show that, for each \(J \in \mathcal{J}(D')\), there
exists a \(J\)-holomorphic curve in the class \(E_n\). Note that this does not
necessarily mean that \(J \in \mathcal{J}(D)\) since the aforementioned curve
may not have the same image as the \(E_n\)-component\footnote{Although the
reader will lose little by ignoring this fact.} of \(D\).
Indeed, \emph{a priori} the Lagrangian sphere \(L\) may
intersect the \(E_n\)-component of \(D\), whereas
Corollary~\ref{cor:En-disjoins-L} shows that, for a sufficiently long neck
stretch \(J_t\), the unique \(J_t\)-holomorphic curve of class \(E_n\) is
disjoint from \(L\).

The existence of the \(J\)-holomorphic curve in the class \(E_n\) will then be
used to show that, for all \(J \in \mathcal{J}(D')\), the curves in the Gromov
compactification \(\overline{\mathcal{M}}_{0,0}(F;J)\) come in the following
types:
\begin{itemize}
    \item a smooth curve of class \(F\), called a \emph{smooth fibre curve};
    \item a nodal curve with exactly 2 components with homology classes
	\(\mathcal{E}_j, F - \mathcal{E}_j\), which will correspond to the
	Lefschetz critical fibres; and
    \item an \emph{exotic} nodal curve \(\infinityCurve{J}\) with \(n+2\)
	genus 0 components, all, except one (the component covering the
	unique \(J\)-holomorphic curve in the class \(E_n\)), of which cover
	those of \(D_\infty \cap D'\).
\end{itemize}
Moreover, the unique \(J\)-holomorphic curves in the classes
\(\fibrationSectionL{}\) and \(\fibrationSectionR{}\) are sections of the
fibration \eqref{eq:lefFib} --- see Corollary~\ref{cor:c00FmoduliDescription}
and Proposition~\ref{prop:lefFibXJ} for further details.

\subsubsection{Almost complex structures on symplectic divisors}
\label{subsub:acsOnDivisors}

That \(\mathcal{J}(D') \ne \emptyset\), follows from
straightforward extensions of results on the symplectic neighbourhood theorem.
See Appendix~\ref{app:acsOnDivisor} for the detailed construction. To apply
the results therein, we need to check that
the components of \(D'\) intersect symplectically orthogonally. However, this
follows from standard facts on toric geometry, see \cite[\S{}3.2]{evansLTF}
for example. Indeed, each intersection point between components of \(D\) is
modelled on a Delzant corner of the moment polytope of \(X\), which is fibred
symplectomorphic to \((\C^2,\omega_{\C^2})\) with its usual toric structure.
The components of \(D'\) correspond to the coordinate planes in \(\C^2\),
which intersect symplectically orthogonally. See
Figure~\ref{fig:divisorComponentsOrthogonal} for a cartoon picture of what
happens.
\begin{figure}[ht]\centering \begin{tikzpicture}[
    moment polytope,
    scale=1.5,
]
    \fill (0,1) -- (0,0) -- (1,-1) -- (1,1) -- cycle;
    \draw (0,1) -- node [left] {\(D_i\)} (0,0) -- node [below left] {\(D_j\)} (1,-1);

    \draw[->,decorate,
	decoration={snake,amplitude=.4mm,segment length=2mm,post length=1mm}]
	(2,0) -- (4,0)
	node [above,text width=3cm,align=center,midway] { fibred }
	node [below,text width=4cm,align=center,midway] { symplectomorphic };
    \begin{scope}[shift={(5,-0.5)}]
	\fill (0,1) -- (0,0) -- (1,0) -- node [right] {\(\cong \C^2\)} (1,1) -- cycle;
	\draw (0,1) -- (0,0) -- (1,0);
    \end{scope}
\end{tikzpicture}
    \caption{
	Intersecting components of the divisor \(D'\) do so symplectically
	orthogonally, since any Delzant corner of a moment polytope is
	integral affine isomorphic to the standard one. The fibred
	symplectomorphism sends the pieces \(D_i\) and \(D_j\) to the
	coordinate planes in \(\C^2\).
    }
    \label{fig:divisorComponentsOrthogonal} 
\end{figure}
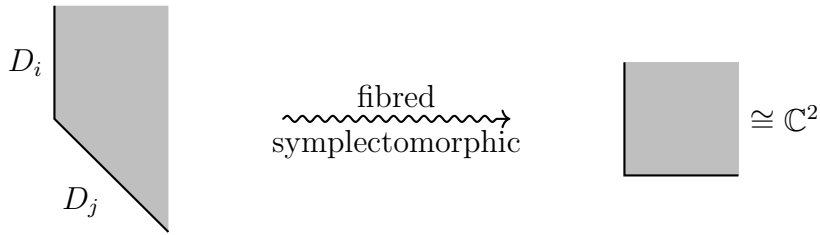

\subsubsection{Non-bubbling under variations of \(J\)}
\label{sec:nonBubbling}

Choose \(J \in \mathcal{J}(D')\). In this section we show that
the class \(E_n\) admits a \(J\)-holomorphic representative. To this end,
recall that the intersection pairing on \(H_2(X) = H_2(X;\Z)\) is
non-degenerate. Given any subset \(\mathcal{A} \subset H_2(X)\) we denote the
orthogonal complement (with respect to the intersection pairing) of the
submodule \(\Z\langle \mathcal{A} \rangle\) it generates by
\(\mathcal{A}^\perp\). By non-degeneracy, \(\dim \mathcal{A}^\perp = \codim
\mathcal{A}\). This first result is a basic application of McDuff's positivity
of intersections \cite{mcduffPosInts}.

\begin{lemma}[]
    \label{lem:containBoundary}
    Let \(A \in H_2(X)\) be any homology class represented by a nodal
\(J\)-holomorphic sphere, that is, there exists \(\mathcal{A} \subset
H_2(X)\) such that \(A = \sum_{A_\alpha \in \mathcal{A}}
A_\alpha\) where
each \(A_\alpha \in \mathcal{A}\) has a \(J\)-holomorphic representative. Let
\(C \subset D\) be a connected subgraph\footnote{In the sense of the dual
intersection graph of \(D\)} of \(D\) and suppose
that  \(A \cdot C_i = 0\) for each \(J\)-holomorphic component \(C_i\) of
\(C\), and that,
for some \(i\), there exists \(k_i \in \N\) such that \(k_i[C_i] \in
\mathcal{A}\). Then, for all \(i\), there exist positive integers \(k_i\)
such that \(k_i[C_i] \in \mathcal{A}\), that is, any nodal curve representing
\(A\) contains \(C\).
\end{lemma}
\begin{proof}
    We induct on the number, \(m\), of irreducible components of \(C\). The base
case \(m=1\) is trivial, so assume the result holds for some \(m > 0\). In
other words, \(C = \bigcup_{1 \le i \le m+1}(C_i)\) and, for \(1 \le i \le m\), there
exist positive integers \(k_i > 0\) such that \(k_i[C_i] \in \mathcal{A}\).
Note that, by construction, \(C_{m+1}\) has at most two neighbours in \(C\),
each of which it intersects exactly once positively. Therefore, we have that
\[
    C_{m+1} \cdot \sum_{i=1}^m k_i C_i > 0.
\]
Combining this with the assumption that \(A \cdot C_{m+1} = 0\), we obtain
\[
    \left( A - \sum_{i=1}^m k_i [C_i] \right) \cdot C_{m+1} < 0.
\]
Positivity of intersections then implies that one of the remaining components
\(A_\alpha \in \mathcal{A} \backslash \{k_i [C_i] \mid 1 \le i \le m\}\)
covers \(C_{m+1}\). Thus, there exists \(k_{m+1} > 0\) such that
\(k_{m+1} [C_{m+1}] \in \mathcal{A}\).
\end{proof}

\begin{lemma}[]
    \label{lem:enBubble}
    For any \(J \in \mathcal{J}(D')\), any stable curve of class \(E_n\) is
smooth: \(\overline{\mathcal{M}}_{0,0}(E_n;J) = \mathcal{M}_{0,0}(E_n;J)\).
That is, a \(J\)-holomorphic curve of class \(E_n\) cannot bubble under
variations of \(J \in \mathcal{J}(D')\).
\end{lemma}
\begin{proof}
    Consider the connected subgraph \(C \subset D\) consisting of \(D\) minus
\(E_n\) and its neighbours. Said another way, \(C\) is a the maximal subgraph
none of whose components intersect \(E_n\). Then, \(C\) consists of all the
vertices except the bottom three in Figure~\ref{fig:dualIntGraphD}. Denote the
set of homology classes of the components of \(C\) as \(\mathcal{C}\).

Suppose that \(\mathbf{u} = (u_\alpha) \in
\overline{\mathcal{M}}_{0,0}(E_n;J)\) is a stable \(J\)-holomorphic curve in
the class \(E_n\), and that, for some components \(u_\alpha\) of
\(\mathbf{u}\) and \(C_i\) of \(C\), we have \(A_\alpha := [u_\alpha] = k_i
[C_i]\). Since \(F \in \mathcal{C}\), and, for all \(c \in \mathcal{C}\),
\(E_n \cdot c = 0\), we may apply Lemma~\ref{lem:containBoundary} to deduce
that \(A_\beta = k_\beta F\) for some component \(u_\beta\) of \(\mathbf{u}\),
which contradicts \(\omega(F) > \omega(E_n) \ge \omega(A_\beta)\). Therefore,
for all \(\alpha\), \(c \in \mathcal{C}\), and \(k \in \N\), we have
\(A_\alpha \ne kc\). Then, the fact that \(E_n \in \mathcal{C}^\perp\)
combined with positivity of intersections, implies that \(A_\alpha \in
\mathcal{C}^\perp\) for all \(\alpha\). Indeed, \(0 = E_n \cdot c =
\sum_\alpha A_\alpha\cdot c\) and \(A_\alpha \cdot c \ge 0\).

    A simple calculation shows that \(\mathcal{C}\) is a rank \(n+2\)
linearly independent set, and so, since \(\dim H_2(X) = n+d+3\), to
explicitly describe \(A_\alpha\) it suffices to find a rank \(d + 1\) linearly
independent set in \(\mathcal{C}^\perp\). Let \(E^-,E^+\) denote the
neighbours of \(E_n \in D'\) and note that \(E^+ - E^- \in
\mathcal{C}^\perp\), since they both intersect \(E_n\) once positively. Then,
it is easily checked that
\[
    \mathcal{C}^\perp = \Z\langle E_n, E^+ - E^-,
	\mathcal{E}_j - \mathcal{E}_{j+1} : 1 \le j < d\rangle.
\]
Recall that during the compactification process of Lemma~\ref{lem:compactify},
the curves \(E^\pm\) were introduced through symplectic cutting the toric
boundary during the resolution of singularities procedure. Since we have choice over
the affine lengths of these cuts, they can be made equal. That is,
\(\omega(E^+) = \omega(E^-)\). Furthermore, by Lemma~\ref{lem:compactify},
\(\omega(\mathcal{E}_j) = l\) does not depend on \(j\), so we find that, for
some integers \(\lambda_n,\lambda_\pm,\mu_j \in \Z\),
\[
    \omega(A_\alpha) = \omega\left(\lambda_nE_n + \lambda_\pm(E^+ - E^-) +
    \sum_{j=1}^{d-1} \mu_j(\mathcal{E}_j - \mathcal{E}_{j+1})\right) =
    \lambda_n\omega(E_n).
\]
That is, the \(\omega\)-area of \(u_\alpha\) is an integer multiple of
\(\omega(E_n)\), and since
\[
    \sum_\alpha \omega(A_\alpha) = \omega(E_n)
\]
this is only possible if it is in fact equal to \(\omega(E_n)\). Since
non-constant \(J\)-holomorphic curves have positive area, this implies that
\(A_\alpha = E_n\) and hence, \(\mathbf{u} = u_\alpha\) is a smooth curve.
\end{proof}

\begin{corollary}[]
    \label{cor:EnAlwaysJhol}
    For any \(J \in \mathcal{J}(D')\), the class \(E_n\) admits a
\(J\)-holomorphic representative.
\end{corollary}
\begin{proof}
    Let \(\mathcal{J}(D',E_n)\) denote the subset of \(\mathcal{J}(D')\) of
almost complex structures that admit a \(J\)-holomorphic representative of
\(E_n\). Section~\ref{subsub:acsOnDivisors} says that \(\mathcal{J}(D) \subset
\mathcal{J}(D',E_n)\) is non-empty, and so, since \(c_1(E_n) > 0\), we may
apply automatic transversality \cite{hoferLizanSikorav97genericity} to show
that \(\mathcal{J}(D',E_n)\) is open in \(\mathcal{J}(D')\). Moreover, by the
non-bubbling result of Lemma~\ref{lem:enBubble}, \(\mathcal{J}(D',E_n)\) is
also a closed subset. Indeed, pick a sequence \(J_\nu \in
\mathcal{J}(D',E_n)\) converging to \(J \in \mathcal{J}(D')\) and a
corresponding sequence of curves \(u_\nu \in \mathcal{M}_{0,0}(E_n;J_\nu)\).
Gromov compactness ensures that there is a convergent subsequence \(u_\nu \to
\mathbf{u}\), for some \(\mathbf{u} \in \overline{\mathcal{M}}_{0,0}(E_n;J)\).
However, Lemma~\ref{lem:enBubble} states that \(\mathbf{u}\) must actually be
a smooth curve, and so we find that \(J \in \mathcal{J}(D',E_n)\). Thus, since
\(\mathcal{J}(D')\) is connected, we must have that \(\mathcal{J}(D',E_n) =
\mathcal{J}(D')\).
\end{proof}

\subsubsection{The universal \(J\)-holomorphic curve}
\label{sec:universalCurve}

For \(J \in \mathcal{J}(D')\), let \(E_J\) be the image of the unique
\(J\)-holomorphic curve in the homology class \(E_n\), and let \(\DinftyJ{J}\)
be the divisor consisting of \((D_\infty \cap D') \cup E_J\). See
Figure~\ref{fig:DJ} for a sketch of the curve configuration in a neighbourhood
of \(E_J\). Define \(\XJ{J} := X \backslash \DinftyJ{J}\).
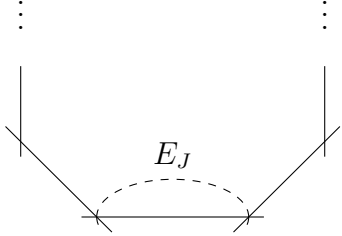
\begin{figure}\centering \begin{tikzpicture}
    \draw (-2,2) node [above=10pt] {\(\vdots\)};
    \draw (2,2) node [above=10pt] {\(\vdots\)};
    \draw (-2,0.8) -- (-2,2);
    \draw (-2.2,1.2) -- (-0.8,-0.2);
    \draw (-1.2,0) -- (1.2,0);
    \draw (2.2,1.2) -- (0.8,-0.2);
    \draw (2,0.8) -- (2,2);
    \draw[dashed] (1,0) arc [x radius=1, y radius=0.5, start angle=0, end
	angle=180] node [above, midway] {\(E_J\)};
    \draw[dashed] (1,0) arc [x radius=1, y radius=0.5, start angle=0, end
	angle=-20] ;
    \draw[dashed] (-1,0) arc [x radius=1, y radius=0.5, start angle=180, end angle=200];
\end{tikzpicture}
    \caption{
	Part of the divisors \(D_\infty\) and \(\DinftyJ{J}\). The only
	difference between the two is the central component: \(\DinftyJ{J}\)
	contains the unique \(J\)-holomorphic curve \(E_J\) in the class
	\(E_n\), whereas the \(E_n\)-component of \(D_\infty\) is part of the
	toric boundary divisor of \(X\).
    }
    \label{fig:DJ}
\end{figure}

Following Wendl \cite[\S{}7.3.3]{wendl2018lowDimCurves}, we define the
\emph{universal \(J\)-holomorphic curve} to be the forgetful map
\[
    \pi_J : \overline{\mathcal{M}}_{0,1}(\XJ{J},F;J) \to
    \overline{\mathcal{M}}_{0,0}(\XJ{J},F;J).
\]
Our goal is to prove that these moduli spaces can be endowed with smooth
structures that realise \(\pi_J\) as a smooth Lefschetz fibration. The only
notable difference between our situation and that in
\cite{wendl2018lowDimCurves} is that the target manifold \(\XJ{J}\) is not closed.
However, due to the specifics of our situation, we can work around this and
show that the required results of \cite[\S{}7.3.4]{wendl2018lowDimCurves} hold
here.

We prove the following results first under a mild genericity condition \(J \in
\mathcal{J}_\mathrm{reg}(D')\) explained in
Definition~\ref{def:genericityModD} below, and then extend them to
all \(J \in \mathcal{J}(D')\) by an automatic transversality and non-bubbling
argument as in Corollary~\ref{cor:EnAlwaysJhol}.
\begin{remark}
    The reason we invoke genericity in the first place is that it gives us a
convenient argument to deduce the structure of
\(\overline{\mathcal{M}}_{0,0}(\XJ{J},F;J)\). One may wonder why we don't just
assume genericity henceforth and skip the automatic transversality argument.
The reason is that in chapter~\ref{ch:neckStretch} we will have to modify our
almost complex structures a few times, and knowing the results of this section
apply to all \(J \in \mathcal{J}(D')\) makes those arguments more
straightforward.
\end{remark}

\begin{lemma}[\emph{c.f.}~Lemma~7.43 of \cite{wendl2018lowDimCurves}]
    \label{lem:mBar00HomeoToC}
    The moduli space \(\overline{\mathcal{M}}_{0,0}(\XJ{J},F;J)\) is homeomorphic
to the complex plane \(\C\).
\end{lemma}
\begin{lemma}[\emph{c.f.}~Lemma~7.45 of \cite{wendl2018lowDimCurves}]
    \label{lem:universalCurveLefFib}
    The moduli spaces \(\overline{\mathcal{M}}_{0,1}(\XJ{J},F;J)\) and
\(\overline{\mathcal{M}}_{0,0}(\XJ{J},F;J)\) admit smooth structures, making them
open manifolds of dimension 4 and 2 respectively, such that \(\pi_J\) is a
Lefschetz fibration with genus zero fibres and exactly one critical point in
each singular fibre.
\end{lemma}

\begin{definition}
    \label{def:genericityModD}
    Fix \(J_0 \in \mathcal{J}(D')\) and denote \(D_{J_0} := D' \cup E_{J_0}\).
Let \(U_{J_0} = X \backslash D_{J_0} \subset \XJ{J_0}\) and consider the
subset of almost complex structures \linebreak \(\mathcal{J}_\mathrm{reg}(U_{J_0},J_0)
\subset \mathcal{J}(D')\) defined to be those that satisfy
\[
    J_{D_{J_0}} = J|_{X \backslash U_{J_0}} = J_0|_{X \backslash U_{J_0}} =
    J_0|_{D_{J_0}},
\]
and every \(J\)-holomorphic curve that maps an injective point\footnote{An
    injective point of a curve \(u : \Sigma \to X\) is one where
\(u^{-1}(u(z)) = \{z\}\) and \(\ud u(z)\) is injective.} into
\(U_{J_0}\) is Fredholm regular.\footnote{A curve \(u\) is Fredholm regular when the
    linearisation of the Cauchy-Riemann operator at the curve is surjective.
    Said another way, this means the moduli space near \(u\) is cut out
transversely and is thus a smooth manifold.} This is a Baire subset\footnote{A Baire
subset is one which contains a countable intersection of open dense sets.}
\cite[Theorem~4.8]{wendl2014curvesLectures} (see also
\cite[Theorem~A.4]{wendl2020Contact3folds} and Remark~3.2.3 of
\cite{mcduffSalamonCurves}) and so we can choose a ``small perturbation'' of
\(J_0\) that lives in \(\mathcal{J}_\mathrm{reg}(U_{J_0},J_0)\). We shall
abuse notation and denote this set by \(\mathcal{J}_\mathrm{reg}(D')\) since
the choice of \(J_0\) doesn't really matter.
\end{definition}

\begin{remark}
    One should think of the set \(U_{J_0}\) above as the set in which almost
complex structures in \(\mathcal{J}_\mathrm{reg}(D')\) are allowed to vary.
\end{remark}

We first prove a result that restricts the form of the non-smooth stable
curves in \(\overline{\mathcal{M}}_{0,0}(\XJ{J},F;J)\). 
\begin{lemma}[]
    \label{lem:constrainGenericNodal}
    For every \(J \in \mathcal{J}_\mathrm{reg}(D')\), every non-smooth stable
curve \linebreak \(\mathbf{u} \in \overline{\mathcal{M}}_{0,0}(\XJ{J},F;J)\) is a nodal
curve with exactly two transversely intersecting components \(\mathbf{u} =
(u_1,u_2)\) satisfying \(([u_1],[u_2]) = (\mathcal{E}_j,F-\mathcal{E}_j)\) for
some \({1 \le j \le d}\). In particular, there are no multiple covers.
\end{lemma}
\begin{proof}
Let \(A_\alpha\) be
the homology class of an irreducible non-constant component \(u_\alpha\) of
\(\mathbf{u} = (u_\alpha)_\alpha\). Since \(u_\alpha\) is disjoint from
\(\DinftyJ{J_0}\), positivity of intersections implies that \(A_\alpha \cdot
D_\beta = 0\) for each irreducible component \(D_\beta\) of
\(\DinftyJ{J_0}\). We write this condition as \(A_\alpha \in
\DinftyJ{J_0}^\perp\). Since the components of \(\DinftyJ{J_0}\) form a rank
\(n+2\) linearly independent set, and \(\dim H_2(X) = n+d+3\), we find that
\(\DinftyJ{J_0}^\perp = \Z\langle F, \mathcal{E}_j : 1 \le j \le d\rangle\).
Thus we may write, for some integers \(\lambda,\mu_j \in \Z\),
\begin{equation}
    \label{eq:homologyClassOfStableFCurve}
    A_\alpha = \lambda F + \sum_{j=1}^d \mu_j \mathcal{E}_j.
\end{equation}

We claim that \(c_1(A_\alpha) > 0\) for all non-constant components of
\(\mathbf{u}\). Indeed, this will follow from the Fredholm regularity
condition of \(J \in \mathcal{J}_\mathrm{reg}(D')\), provided that each
non-constant component passes through \(U_{J_0}\). Now, by the open mapping
theorem, a \(J\)-holomorphic map \(u_\alpha : S^2 \to D_{J_0}\) is either
constant, or covers a component of \(D_{J_0}\). However, since \(A_\alpha \in
\DinftyJ{J_0}^\perp\), the latter case would imply that \(u_\alpha\) covers
the \(F\)-component of \(D_{J_0}\), which is only possible if \(\mathbf{u}\)
is smooth. Therefore, in the non-smooth case, we conclude that every component
\(u_\alpha\) passes through \(U_{J_0}\). However, \emph{a priori},
\(u_\alpha\) may not be simple, so we pass to its underlying simple curve, say
\(u_\alpha'\). This satisfies the condition that it maps an injective point
into \(U_{J_0}\), and so, it is Fredholm regular, implying that
\(c_1(u_\alpha') > 0\), and thus \(c_1(u_\alpha) = c_1(A_\alpha) > 0\).

Combining the above with the arguments of Proposition~4.8, Lemma~4.12, and
\S{}4.3 of \cite{wendl2018lowDimCurves}, we deduce that \(\mathbf{u}\) is
a nodal curve with exactly two embedded components of index 0 intersecting in
a single node. Let us relabel the components as \(\mathbf{u} = (u_1,u_2)\).
Since the index of a curve is given by
\begin{equation}
    \label{eq:closedIndexFormula}
    \ind(u) = 2c_1([u]) - 2 \ge 0,
\end{equation}
we find that \(c_1(A_i) = 1\) for both components \((u_1,u_2)\) of
\(\mathbf{u}\). As \(u_i\) is embedded, the adjunction formula
\cite[Theorem~2.6.4]{mcduffSalamonCurves} says
\[
    c_1(A_i) = A_i^2 + 2,
\]
implying that \(A_i^2 = -1\) and so \(\sum_j \mu_j^2 = 1\), from which we
deduce that exactly one \(\mu_j\) is non-zero, and furthermore, this
coefficient is \(\pm1\). The equation
\[
    2\lambda + \mu_j = c_1(A_i) = 1
\]
ensures that either \(A_i = \mathcal{E}_j\), or, \(A_i = F - \mathcal{E}_j\).
Finally, the condition
\[
    \sum_{i=1}^2 A_i = F
\]
implies that the other component \(u_{i+1}\) of \(\mathbf{u}\) satisfies
\([u_{i+1}] = A_{i+1} = F - A_i\), from which the result follows.
\end{proof}

\begin{remark}
    We call the types of nodal curves arising from
Lemma~\ref{lem:constrainGenericNodal} \emph{curves of type} \((\mathcal{E}_j,
F-\mathcal{E}_j)\).
\end{remark}

\begin{proof}[Proof of Lemma~\ref{lem:mBar00HomeoToC}]
    A topological manifold structure on
\(\overline{\mathcal{M}}_{0,0}(\XJ{J},F;J)\) is constructed via a gluing
argument, such as that in the proof of Lemma~7.43 in
\cite{wendl2018lowDimCurves}. Our application of gluing is valid since 
Lemma~\ref{lem:constrainGenericNodal} showed that the only non-smooth
stable curves are nodal curves with exactly two transversely intersecting
components.

It remains to show that
\(\overline{\mathcal{M}}_{0,0}(\XJ{J},F;J)\) is homeomorphic to \(\C\). To this
end, recall the evaluation map
\[
    \mathrm{ev} :
    \overline{\mathcal{M}}_{0,1}(\XJ{J},F;J) \to \XJ{J} :
    \mathrm{ev}([(\mathbf{u},x)]) = \mathbf{u}(x).
\]
In our case, \(\mathrm{ev}\) is bijective, which follows by a standard
foliation and compactness argument (for example, combine
\cite[Proposition~2.53]{wendl2018lowDimCurves}, with an argument such as the
end of the proof of Theorem~1.16 of \cite{wendl2020Contact3folds}). Moreover,
since it's a proper map,\footnote{Preimages of compact sets are compact.} and
\(\XJ{J}\) is locally compact and Hausdorff, it follows that \(\mathrm{ev}\)
is a homeomorphism onto \(\XJ{J}\). Now fix a parametrisation \(v : S^2 \to
X\) of the divisor component of class \(\fibrationSectionL{}\), and remove the
unique point\footnote{This is the point \(\fibrationSectionL\cap D_\infty\).}
from the domain that maps to \(\DinftyJ{J}\). Then, since
\(\fibrationSectionL{} \cdot F = 1\), the composition \(\pi_J \circ
\mathrm{ev}^{-1} \circ v : \C \to \overline{\mathcal{M}}_{0,0}(\XJ{J},F;J)\)
is the required homeomorphism.
\end{proof}

\begin{proof}[Proof of Lemma~\ref{lem:universalCurveLefFib}]
    The smooth structures on the moduli spaces are constructed through the
usual implicit function theorem argument \cite[\S{}3]{mcduffSalamonCurves} on
the set of Fredholm regular curves, combined with the gluing argument at the
nodes given in \cite[Lemma~7.45]{wendl2018lowDimCurves}. The same result
there also shows that \(\pi_J\) also has a Lefschetz fibration structure at
the nodal points. Since these arguments are inherently local, they also apply
to our situation of the non-compact target manifold \(\XJ{J}\).
\end{proof}

\begin{corollary}[]
    \label{cor:eulerCharM01Bar}
    The moduli space \(\overline{\mathcal{M}}_{0,1}(\XJ{J},F;J)\) is homeomorphic
to\linebreak \(Y \# k\CP^2\), where \(Y\) is a ruled surface over
\(\overline{\mathcal{M}}_{0,0}(\XJ{J},F;J) \cong \C\), and \(k \ge 0\) is the
number of nodal fibres. In particular, the Euler characteristic satisfies
\begin{equation}
    \chi(\overline{\mathcal{M}}_{0,1}(\XJ{J},F;J)) = k + 2.
\end{equation}
\end{corollary}
\begin{proof}
    The homeomorphism claim is Corollary~7.46 of \cite{wendl2018lowDimCurves},
and the Euler characteristic formula then follows from a calculation using
elementary properties of \(\chi\). Specifically, \(Y\) is fibred by 2-spheres
over \(\C\), and so
\[
    \chi(Y) = \chi(S^2)\chi(\C) = 2,
\]
and connect summing with \(\CP^2\) increases \(\chi\) by
\(1\).
\end{proof}

\begin{corollary}[]
    \label{cor:genericNodalCurves}
    For every \(J \in \mathcal{J}_\mathrm{reg}(D')\), there are exactly \(d\)
non-smooth stable curves in \(\overline{\mathcal{M}}_{0,0}(\XJ{J},F;J)\) given
exactly by the curves of type \((\mathcal{E}_j, F-\mathcal{E}_j)\) for all \(1
\le j \le d\).
\end{corollary}
\begin{proof}
    Since the evaluation map \(\mathrm{ev} :
\overline{\mathcal{M}}_{0,1}(\XJ{J},F;J) \to \XJ{J}\) is a homeomorphism, and
by Corollary~\ref{cor:eulerCharM01Bar}, we have that
\[
    \#\text{nodal curves} + 2 = \chi(\overline{\mathcal{M}}_{0,1}(\XJ{J},F;J)) =
    \chi(\XJ{J}) = d+2,
\]
and so there must be exactly \(d\) nodal curves in \(\XJ{J}\), all of which are of
type \((\mathcal{E}_j, F-\mathcal{E}_j)\) by
Lemma~\ref{lem:constrainGenericNodal}. Since a nodal curve of this type is
unique, that is, for each \(1 \le j \le d\), there is at most one nodal curve
of type \((\mathcal{E}_j, F-\mathcal{E}_j)\), then there is exactly one for
each integer \(j\).
\end{proof}

It is at this moment that we drop the genericity condition. The next
result is a generalisation of Lemma~\ref{lem:constrainGenericNodal}, and it
will be used to show that all of the previous results that relied on
genericity continue to hold.
\begin{lemma}[]
    \label{lem:constrainNodal}
    For every \(J \in \mathcal{J}(D')\), a non-smooth stable curve in
\linebreak
\(\overline{\mathcal{M}}_{0,0}(\XJ{J},F;J)\) must be of type \((\mathcal{E}_j,
F-\mathcal{E}_j)\) for some \(1 \le j \le d\).
\end{lemma}
\begin{proof}
    The argument is similar to that of
Corollary~\ref{cor:EnAlwaysJhol}. We will show that the set of almost complex
structures \(\mathcal{J}(D',A)\) that support a \(J\)-holomorphic curve in the
class \(A \in H_2(X)\), where \(A\) is one of the classes \(\mathcal{E}_j,
F-\mathcal{E}_j\), is equal to the whole space
\(\mathcal{J}(D')\) by automatic transversality and non-bubbling results. The
fact that \(\mathcal{J}(D',A)\) is non-empty follows immediately from
Corollary~\ref{cor:genericNodalCurves}, so we need only prove the non-bubbling
results.

Suppose that \(J \in \mathcal{J}(D')\), assume that \(A = \mathcal{E}_j\),
and let \(\mathbf{u} \in \overline{\mathcal{M}}_{0,0}(\mathcal{E}_j;J)\) be a
stable curve of class \(\mathcal{E}_j\). Choose a non-constant component
\(u_\alpha\) of \(\mathbf{u}\) and note, as in the proof of
Lemma~\ref{lem:enBubble}, that Lemma~\ref{lem:containBoundary} ensures that
\[
    A_\alpha = [u_\alpha] \in (\DinftyJ{J} \cup \fibrationSectionL{})^\perp.
\]
Since, \(D_\infty \cup \fibrationSectionL{}\) forms a rank \(n+3\) linearly
independent set in \(H_2(X)\), we find that \((\DinftyJ{J} \cup
\fibrationSectionL{})^\perp = \Z\langle \mathcal{E}_k : 1 \le k \le
d\rangle\), and so, for some integers \(\mu_k \in \Z\)
\[
    A_\alpha = \sum_{k=1}^d \mu_k \mathcal{E}_k.
\]
Recall that the areas of the classes \(\mathcal{E}_k\) are all equal to \(l\).
Thus, we have that \(\omega(A_\alpha)\) is an integer multiple of \(l\):
\[
    \omega(A_\alpha) = l\sum_{k=1}^d \mu_k.
\]
However, this is only possible if \(\omega(A_\alpha) = l\), and so we conclude
that \(\mathbf{u} = u_\alpha\) is actually a smooth curve.

Repeating the above argument with \(\mathcal{E}_j\) replaced with
\(F-\mathcal{E}_j\) and \(\fibrationSectionL{}\) replaced with
\(\fibrationSectionR\) shows that \(F-\mathcal{E}_j\) curves also don't undergo
bubbling. Hence, by the argument of Corollary~\ref{cor:EnAlwaysJhol},
\(\mathcal{J}(D',A) = \mathcal{J}(D')\). Summarising, we have shown that the
classes \(\mathcal{E}_j, F-\mathcal{E}_j\) are \(J\)-holomorphic for all \(J
\in \mathcal{J}(D')\). It remains to show that any non-smooth stable curve must
be of type \((\mathcal{E}_j, F-\mathcal{E}_j)\).

Suppose that we have a stable curve \(\mathbf{u} \in
\overline{\mathcal{M}}_{0,0}(\XJ{J},F;J)\) that is not of type \((\mathcal{E}_j,
F-\mathcal{E}_j)\) for all \(1 \le j \le d\). Then, since the classes
\(\mathcal{E}_j\) and \(F-\mathcal{E}_j\) are \(J\)-holomorphic, a simple
positivity of intersections argument, akin to Lemma~\ref{lem:containBoundary},
implies that \(\mathbf{u}\) has no components in common with \(\mathcal{E}_j\)
and \(F-\mathcal{E}_j\). In particular, for any component \(u_\alpha\) of
\(\mathbf{u}\), we have that
\[
    A_\alpha \in \left(\DinftyJ{J} \cup \bigcup_{j=1}^d \mathcal{E}_j
    \right)^\perp = \Z\langle F \rangle,
\]
and so \(\mathbf{u}\) must actually be a smooth curve. This completes the
proof.
\end{proof}

\begin{corollary}[]
    \label{cor:c00FmoduliDescription}
    For every \(J \in \mathcal{J}(D')\) the compactifying curves
in\linebreak \(\overline{\mathcal{M}}_{0,0}(X,F;J)\), that is, the elements of
\(\overline{\mathcal{M}}_{0,0}(X,F;J) \backslash \mathcal{M}_{0,0}(X,F;J)\),
are given exactly by:
\begin{itemize}
    \item for each \(1 \le j \le d\), a nodal curve of type \((\mathcal{E}_j,
	F-\mathcal{E}_j)\), and
    \item a stable curve \(\infinityCurve{J}\) whose components cover those of
	\(\DinftyJ{J}\). We call this the \emph{exotic curve} in
	\(\overline{\mathcal{M}}_{0,0}(X,F;J)\).
\end{itemize}
In particular, \(\overline{\mathcal{M}}_{0,0}(X,F;J)\) is homeomorphic to
\(S^2\).
\end{corollary}
\begin{proof}
    To prove that the curves of type \((\mathcal{E}_j, F-\mathcal{E}_j)\) do
indeed exist, we wish to repeat the argument of
Corollary~\ref{cor:genericNodalCurves}. We achieve this by noting that
Lemmata~\ref{lem:mBar00HomeoToC} and \ref{lem:universalCurveLefFib} continue
to hold in the non-generic case, since the only way they depend on
genericity is to prove that non-smooth stable curves consist only of nodal
curves with exactly two components intersecting transversely, but
Lemma~\ref{lem:constrainNodal} shows that this continues to hold in the
non-generic case. Hence, the result of Corollary~\ref{cor:genericNodalCurves}
continues to hold.

To see that the stable curve \(\infinityCurve{J} \in
\overline{\mathcal{M}}_{0,0}(X,F;J)\) exists, we use the fact that
\(\mathrm{ev} : \overline{\mathcal{M}}_{0,1}(\XJ{J},F;J) \to \XJ{J}\) is a
homeomorphism to pick a sequence of smooth curves \(u_\nu \in
\mathcal{M}_{0,0}(X,F;J)\) passing through points \(x_\nu \in X\) converging
to \(x \in \DinftyJ{J}\). Gromov compactness allows us to extract a convergent
subsequence \(u_\nu \to \infinityCurve{J}\). Since \(F^2 = 0\), the image of
\(\infinityCurve{J}\) must be disjoint from all the curves in
\(\overline{\mathcal{M}}_{0,1}(\XJ{J},F;J)\), so it follows that
\(\im\infinityCurve{J} \subset \DinftyJ{J}\). Unique continuation
\cite[\S{}2.3]{mcduffSalamonCurves} of \(J\)-holomorphic curves implies that
each of the non-constant components of \(\infinityCurve{J}\) have the same
image as a component of \(\DinftyJ{J}\). Hence, \(\infinityCurve{J}\) covers
\(\DinftyJ{J}\).

Finally, the Gromov topology on \(\overline{\mathcal{M}}_{0,0}(X,F;J)\)
realises it as the one point compactification of
\(\overline{\mathcal{M}}_{0,0}(\XJ{J},F;J) \cong \C\). Whence we obtain that
\(\overline{\mathcal{M}}_{0,0}(X,F;J)\) is homeomorphic to \(S^2\).
\end{proof}

\subsection{The Lefschetz fibrations \(\pi_J : \XJ{J} \to \C\)}
\label{sec:compactLefFib}

The results of section~\ref{sec:universalCurve} allow us to construct
Lefschetz fibrations on \(\XJ{J} = X \backslash \DinftyJ{J}\) for all \(J \in
\mathcal{J}(D')\). This section proves the following proposition.

\begin{proposition}[]
    \label{prop:lefFibXJ}
    For every \(J \in \mathcal{J}(D')\), there exists a Lefschetz
fibration \(\pi_J : \XJ{J} \to \C\) which has smooth fibres in the class \(F\),
and exactly \(d\) nodal fibres of type \((\mathcal{E}_j, F - \mathcal{E}_j)\).
\end{proposition}

This almost immediately follows from \S{}\ref{sec:universalCurve}, however
there are a few smoothness technicalities to check. The aim is to show that
the map
\[
    \XJ{J} \to \overline{\mathcal{M}}_{0,0}(\XJ{J},F;J) : x \mapsto \text{the curve
    passing through }x
\]
is a smooth Lefschetz fibration. Precisely, this map is equal to \(\pi_J \circ
\mathrm{ev}^{-1}\), where \(\pi_J : \overline{\mathcal{M}}_{0,1}(\XJ{J},F;J) \to
\overline{\mathcal{M}}_{0,0}(\XJ{J},F;J)\) is the universal curve from the
previous section. We abuse notation and denote this composition by \(\pi_J\)
also, as we won't make further reference to the universal curve. A detailed
proof of the following can be found in Lemma~6.29 of \cite{LvHMWII}. We
recall it here as our subsequent arguments crucially depend on it.
\begin{lemma}[]
    \label{lem:c00smoothAtlas}
    There exists a unique smooth structure on
\(\overline{\mathcal{M}}_{0,0}(\XJ{J},F;J)\) such that the map \(\pi_J : \XJ{J} \to
\overline{\mathcal{M}}_{0,0}(\XJ{J},F;J)\) is smooth except perhaps at images in
\(\XJ{J}\) of the nodes of nodal curves in
\(\overline{\mathcal{M}}_{0,0}(\XJ{J},F;J)\).
\end{lemma}
\begin{proof}
    We construct a smooth atlas on \(\overline{\mathcal{M}}_{0,0}(\XJ{J},F;J)\)
making use of the smoothness and embeddedness properties of the curves in this
moduli space. Choose a point \(p \in \XJ{J}\) that is \emph{not} the image of a
node, and fix an embedded open 2-disc \(D_p\) passing through \(p\) that is
transverse to the tangent space of the curve \(u_p \in
\overline{\mathcal{M}}_{0,0}(\XJ{J},F;J)\) passing through \(p\). After possibly
shrinking \(D_p\), we may assume that nearby curves in some open
neighbourhood \(U \subset \overline{\mathcal{M}}_{0,0}(\XJ{J},F;J)\) of \(u_p\)
intersect \(D_p\) exactly once, transversely. Therefore, we have a
homeomorphism \(U \to D_p\). We consider \(U \to D_p\) to be a smooth chart by
smoothly identifying \(D_p\) with the open unit disc in \(\C\). We take the
atlas to be the maximal one containing the collection of all such charts, all
of which are smoothly compatible since the foliation
\(\overline{\mathcal{M}}_{0,0}(\XJ{J},F;J)\) is smooth. The claims about
smoothness of \(\pi_J\) and uniqueness of the corresponding atlas follow by
construction.
\end{proof}

\begin{remark}
    As in Lemma~\ref{lem:mBar00HomeoToC},
\(\overline{\mathcal{M}}_{0,0}(\XJ{J},F;J)\) equipped with the smooth structure
of Lemma~\ref{lem:c00smoothAtlas} is diffeomorphic to the complex plane
\(\C\). Indeed, as a 2-dimensional topological manifold, there is a unique
smooth structure up to diffeomorphism. However, it is not clear that these
smooth structures are \emph{compatible} with one another. Indeed, smoothness
of gluing maps is a subtle question in general
\cite[Remark~2.40]{wendl2018lowDimCurves}
\cite[Remark~6.30]{abouzaidMcLeanSmithKuranishi}, and one should be careful
when discussing smooth structures on spaces of stable \(J\)-holomorphic
curves.
\end{remark}

Next, we want to show that \(\pi_J\) is smooth over the nodal points. Firstly,
we note a result sketched in the appendix of \cite{wendl2018lowDimCurves},
which explains how moduli spaces of square zero spheres that
degenerate to a single curve with exactly one transverse node ``look like''
Lefschetz singular fibres.
\begin{lemma}[Corollary A.3 of \cite{wendl2018lowDimCurves}]
    \label{lem:lefCoords}
    Each nodal point in \(\XJ{J}\) admits a complex coordinate chart in which the
leaves of the foliation \(\overline{\mathcal{M}}_{0,0}(\XJ{J},F;J)\) are
identified with the fibres of the map \(\pi_0 : \C^2 \to \C : \pi_0(z_1,z_2) =
z_1z_2\).
\end{lemma}

\begin{corollary}[]
    The map \(\pi_J : \XJ{J} \to \overline{\mathcal{M}}_{0,0}(\XJ{J},F;J)\) is a
smooth Lefschetz fibration. The nodal fibres are given exactly by the \(d\)
nodal curves of type \((\mathcal{E}_j,F-\mathcal{E}_j)\).
\end{corollary}
\begin{proof}
    Let \(V \subset \C\) be an open neighbourhood of \(0 \in \C\) and choose a
local section \(s : V \to \C^2\) of \(\pi_0\). Note that \(\im s\) is
necessarily disjoint from the unique nodal point \(0 \in \C^2\). Denote one of
the coordinate charts given by Lemma~\ref{lem:lefCoords} by \(\psi : U \to
\C^2\). The fact that \(s\) is a section and \(\psi\) maps curves in
\(\overline{\mathcal{M}}_{0,0}(\XJ{J},F;J)\) to fibres of \(\pi_0\) ensures that
the composition \({\phi := \pi_J \circ \psi^{-1} \circ s : V \to
\overline{\mathcal{M}}_{0,0}(\XJ{J},F;J)}\) is one of the charts of the atlas
constructed in Lemma~\ref{lem:c00smoothAtlas}. Again, as \(\psi\) maps curves
to fibres, we have that
\begin{equation}
    \label{eq:piJlocalDescription}
    \pi_J = \phi \circ \pi_0 \circ \psi,
\end{equation}
that is, the following diagram commutes:
\begin{equation}
\label{eq:commutative}
\begin{tikzcd}
    U \arrow[r,"\psi"] \arrow[d,"\pi_J"] & \C^2 \arrow[d,"\pi_0"] \\
    \overline{\mathcal{M}}_{0,0}(\XJ{J},F;J) & V \arrow[l,"\phi"]
\end{tikzcd}.
\end{equation}
Since the all the maps on the right-hand-side of
\eqref{eq:piJlocalDescription} are smooth everywhere, we must
have that \(\pi_J\) is smooth at the node. Moreover, by the commutativity of
\eqref{eq:commutative}, the local charts \(\psi\) and \(\phi\) give \(\pi_J\)
the structure of a Lefschetz singularity at the node. The claim about the
nodal fibres follows directly from Corollary~\ref{cor:c00FmoduliDescription}.
\end{proof}

\begin{proof}[Proof of Proposition~\ref{prop:lefFibXJ}]
    All that remains to do is identify the moduli\linebreak spaces
\(\overline{\mathcal{M}}_{0,0}(\XJ{J},F;J)\) with \(\C\) for all \(J \in
\mathcal{J}(D')\). In fact, we consider the subset \(\mathcal{J}' \subset
\mathcal{J}(D')\) of almost complex structures that are fixed along \(T\fibrationSectionL{}
\subset TX|_{\fibrationSectionL{}}\). Then, similarly to the proof of
Lemma~\ref{lem:constrainGenericNodal}, we can fix a parametrisation \(v : S^2
\to X\) of the \(\fibrationSectionL{}\) curve that is \(J\)-holomorphic for all \(J \in
\mathcal{J}'\) and such that \(v(0)\) is in the \(F\)-component of \(D\) and
\(v(\infty) = S \cap \DinftyJ{J}\). The condition \(\fibrationSectionL{} \cdot F = 1\)
ensures that the composition \(\pi_J \circ v\) is a homeomorphism, and
moreover, it is a diffeomorphism by construction of the smooth structure on
\(\overline{\mathcal{M}}_{0,0}(\XJ{J},F;J)\) since the transverse disk \(D_p\) can
be taken to be \(v(S^2\backslash \{\infty\})\). This completes the proof.
\end{proof}

\begin{remark}
    \label{rem:lefFibMarkedPoints}
    \begin{enumerate}
	\item One can visualise the map \(\pi_J : \XJ{J} \to \C\) as a Riemann
sphere with \(d+2\) marked points. Indeed, for each \(J \in \mathcal{J}(D')\),
denote that curves of type \((\mathcal{E}_j, F-\mathcal{E}_j)\) by
\(\mathbf{u}_j^J\). Then we define a marked surface by
\[
    (S^2, (0,\infty, v^{-1}(\mathbf{u}_1^J), \ldots, v^{-1}(\mathbf{u}_d^J))).
\]
As \(J\) varies, the \(d\) marked points corresponding to Lefschetz singular
fibres move around, whilst the points \(0\) and \(\infty\) remain fixed.
	\item The domain \(\XJ{J}\) of \(\pi_J\) depends on \(J\). However, as we
shall see later, the Lagrangian isotopy from \(L\) to a matching cycle of
\(\pi_J\) occurs away from the curve \(\infinityCurve{J}\). Therefore, by
excising a suitable tubular neighbourhood of \(\infinityCurve{J}\), we can fix
a subset of \(X\) on which all of the maps \(\pi_J\) are
defined (see the discussion preceding
Corollary~\ref{cor:isotopyToMatchingCycle}.
    \end{enumerate}
\end{remark}

\section{Constructing a foliation adapted to a Lagrangian sphere}
\label{ch:neckStretch}
In this section we apply the neck stretching technique to produce a foliation
on \(T^*S^2\) by \(J\)-holomorphic cylinders with the property that those
cylinders intersecting the zero section do so along smooth circles. It is in
this sense that the fibration is adapted to a Lagrangian sphere (since any
Lagrangian corresponds to the zero section of its Weinstein neighbourhood).

First, we recall some further details on the neck stretching setup. We then
construct an almost complex structure \(\JTS\) on \(T^*S^2\) that is suitable
for neck stretching, and satisfies the property that it is
\emph{anti-invariant} under fibre-wise multiplication by \(-1\). This is the
property that powers the intersection along circles argument.

The intersection theory for punctured \(J\)-holomorphic curves is reviewed in
Section~\ref{sub:intersection-theory}, as well as explaining how this works in
the case of Morse-Bott degenerate asymptotics. This theory is then used in
Section~\ref{sub:cotangentPlanes} to prove some basic facts about foliations
of \(T^*S^2\) by \(J\)-holomorphic planes constructed by Hind
\cite{hind2004Spheres}.

The bulk of the section takes places in Section~\ref{sec:buildingAnalysis}
where we apply intersection theory to analyse the holomorphic buildings that
arise as limits of curves under neck stretching. Finally,
Section~\ref{sec:foliationConstruction} uses the limit analysis to construct a
neck stretching sequence \(J_k\) such that all the curves in
\(\overline{\mathcal{M}}_{0,0}(F;J_k)\) that pass through the Weinstein
neighbourhood of \(L\) converge to their respective holomorphic buildings as
\(k \to \infty\).

\subsection{Neck-stretching, holomorphic buildings, and SFT compactness}
\label{sec:neck-stretch}

We briefly recall the neck stretching construction of
\cite{cieliebakMohnkeSFTCompactness}, specialised to the case of stretching
about a Lagrangian submanifold.
Consider a Lagrangian \(L\) of a closed symplectic manifold \((X,\omega)\) and
equip \(L\) with a Riemannian metric. This allows us to define the sphere
bundle of covectors of a fixed length:
\[
    T_R^*L := \{ v \in T^*L \mid |v| = R\}.
\]
Along with the restriction of the canonical 1-form \(\lambda_\mathrm{can}\), 
the pair \((T_R^*L,\lambda_\mathrm{can})\) is an example of a contact
manifold. To condense the notation, we set \((M,\lambda) := (T_R^*L,
\lambda_\mathrm{can})\). A contact-type hypersurface in a symplectic manifold
admits a tubular neighbourhood symplectomorphic to \(((-\epsilon,\epsilon)
\times M, \ud(e^r\lambda))\), which we call the \emph{neck region}, or simply
the \emph{neck}.

The symplectisation \((\R \times M, \ud(e^r\lambda))\), carries a
special class of almost complex structures suitable for SFT. Consider the
vector field \(\partial_r\) given by the \(\R\)-coordinate, and define the
\emph{Reeb vector field} \(R_\lambda\) by
\[
    \lambda(R_\lambda) \equiv 1, \text{ and }
    \iota_{R_\lambda}\ud\lambda \equiv 0.
\]
\begin{definition}
    We say that an almost complex structure \(J\) on a symplectisation \(\R
\times M\) is \emph{cylindrical} when
\begin{enumerate}
    \item \(J\partial_r = R_\lambda\),
    \item \(J\) is invariant under \(\R\)-translation, and
    \item the restriction \(J|_\xi\) of \(J\) to the contact structure \(\xi =
	\ker\lambda\) is a compatible almost complex structure on the
	symplectic vector bundle \((\xi,\ud\lambda) \to M\).
\end{enumerate}
\end{definition}

Since the neck region is symplectomorphic to part of a symplectisation, we can
consider the class of almost complex structures \(J\) on \(X\) that restrict
to a cylindrical almost complex structure \(J_M\) on the neck
\([-\epsilon,\epsilon] \times M\). These are sometimes called almost complex
structures that are \emph{adapted} to the hypersurface \(\{0\} \times M
\subset X\), but we will often abuse language and call them cylindrical. The
process of neck stretching is to start with one of these adapted almost
complex structures and replace the neck with larger and larger portions
\([-t-\epsilon,\epsilon] \times M\) of \(\R \times M\) along with an almost
complex structure that is cylindrical along this longer neck. More precisely,
following \cite[\S{}2.7]{cieliebakMohnkeSFTCompactness} we excise the neck
region to obtain a compact manifold\footnote{In the notation of
    \cite{cieliebakMohnkeSFTCompactness} this is \(\bar{X}_0\).} \(Y := X
\backslash (-\epsilon,\epsilon) \times M\) with boundary \(\partial Y = M_-
\sqcup M_+\), where \(M_\pm = \{\pm\epsilon\} \times M \subset X\). Form the
\emph{stretched} manifold \(X_t\) by
\[
    X_t := Y \cup_{M_-\sqcup M_+} [-t-\epsilon,\epsilon]\times M,
\]
which is diffeomorphic to \(X\). Define the almost complex structure \(J_t\)
on \(X_t\) by
\begin{equation}
    \label{eq:neckStretch}
    J_t = \begin{cases}
	J, &\text{ on }Y\\
	J_M, &\text{ on } [-t-\epsilon,\epsilon] \times M.
    \end{cases}
\end{equation}
The family of almost complex structures \(J_t\) is called a \emph{neck
stretch} of \(J\).

A neck stretch \(J_t\) has no well-defined limit on \(X\) as \(t \to \infty\).
However, due to the \(\R\)-invariant nature on the neck, \(J_t\) can be seen to
converge (on compact subsets) to almost complex structures on either the
symplectisation \(\R \times M\), or the \emph{completed} symplectic manifold
\[
    \hat{Y} = Y \cup_{M_-} [-\epsilon,\infty) \times M \cup_{M_+}
    (-\infty,\epsilon] \times M.
\]
Taking a monotonic sequence \(t_k \to \infty\) and the corresponding neck
stretch \(J_k := J_{t_k}\), we consider sequences \(f_k : \Sigma \to X_k\) of
\(J_k\)-holomorphic curves with uniformly bounded energy.\footnote{The
    definition of energy is a little technical in general
    \cite[\S{}6.1]{behwzSFTCompactness}. However, in our case, as we are neck
    stretching around a contact-type hypersurface in a closed manifold, the
    energy bound will be automatically satisfied by considering sequences of
homologous curves.} To make sense of a limit of \(f_k\) as \(k\to\infty\) we
need to recall the notion of a \emph{holomorphic building}. Essentially, this
is a map \(F : \Sigma^* \to X^*\) from a (potentially disconnected and
punctured) Riemann surface \(\Sigma^* = \bigsqcup_{\nu=0}^N \Sigma^{(\nu)}\)
to the manifold
\[
    X^* = \hat{Y} \sqcup \bigsqcup_{\nu=1}^N \R \times M,
\]
where the restriction \(F^{(\nu)} = F|_{\Sigma^{(\nu)}}\) to each
\emph{level} maps into
\[
    X^{(\nu)} = \begin{cases}
        \hat{Y}, &\text{ if }\nu=0,N+1\\
	\R \times M, &\text{ if }\nu=1,\ldots,N.
    \end{cases}
\]
The limits of the neck stretch sequence \(J_k\) fit together into an almost
complex structure \(J^*\) on \(X^*\), with respect to which \(F\) is
\(J^*\)-holomorphic.

Each level \(F^{(\nu)} : \Sigma^{(\nu)} \to X^{(\nu)}\) is a finite-energy
punctured curve, with each puncture asymptotic to a Reeb orbit in
\((M,\lambda)\). The set of punctures \(\Gamma\) is partitioned into positive
and negative sets \(\Gamma^\pm\) determined by whether the asymptote
corresponding to a puncture lives in either the convex or concave ideal
boundary of \(X^{(\nu)}\) respectively. Crucially, the asymptotics in
adjoining levels of a \(J^*\)-holomorphic building agree. That is, the \(\nu\)
level connects to the \((\nu+1)\) level, and the set of positive asymptotics
of \(F^{(\nu)}\) \emph{is equal to} the set of negative asymptotics of
\(F^{(\nu+1)}\). Amongst other things, this means we can glue \(\Sigma^*\)
together into a compact Riemann surface \(\bar{\Sigma}\) and \(F\)
correspondingly glues together into a continuous map \(\bar{F} : \bar{\Sigma}
\to X\), whence we obtain the homology class \([\bar{F}] \in H_2(X)\) of the
building.\footnote{Sometimes we will drop the bar notation and just write
\([F]\) when no confusion is possible.}

A \emph{stable} \(J^*\)-holomorphic building is one where none of the levels
\(F^{(\nu)}\) restrict to constant maps on any sphere component with fewer
than 3 nodal points or punctures, and moreover, that no level \(F^{(\nu)}\) is
comprised solely of a union of trivial cylinders over Reeb orbits. We will
produce stable \(J^*\)-holomorphic buildings by using the SFT compactness
theorem \cite[Theorem~10.3]{behwzSFTCompactness}
\cite[Theorem~1.1]{cieliebakMohnkeSFTCompactness}:
\begin{theorem}[SFT compactness]
    Given a neck stretching family \(J_t\) and a sequence of
\(J_k\)-holomorphic curves \(f_k : \Sigma \to X_k\) with uniformly bounded
energy, there exists a subsequence converging to a stable \(J^*\)-holomorphic
building \(F\).
\end{theorem}
The nature of the convergence is technical (see Definition~2.7 of
\cite{cieliebakMohnkeSFTCompactness} for it in full) and we note just a couple
of its properties. For each level \(\nu\), up to reparametrisation, there
exists a sequence of translations of the maps \(f_k\) in the neck
region so that they converge to the level \(F^{(\nu)}\) in the \(\Cinfloc\)
topology. Additionally, the convergence ensures that of the homology classes:
\[
    [f_k] \to [\bar{F}] \in H_2(X).
\]
In particular, if \([f_k] = A\) is independent of \(k\), then \([\bar{F}] =
A\).

In our case \(M=T_R^*L\) is a separating hypersurface in
\(X\), meaning that \(Y = X \backslash (-\epsilon,\epsilon) \times M\) has two
connected components\footnote{The sign of \(Y_\pm\) corresponds to
that of the boundary \(\partial Y_\pm = \{\pm\epsilon\} \times M\).} \(Y_\pm\).
This means that \(X^{(0)} = \hat{Y}\) decomposes into two components
\(X_\pm^{(0)}\), which are the completions of \(Y_\pm\). The \(0\)-level of a
holomorphic building then naturally splits into two components called the
\emph{top} and \emph{bottom} levels: \(\Sigma^{(0)} = \Sigma_+^{(0)} \sqcup
\Sigma_-^{(0)}\) and \(F_\pm^{(0)} := F|_{\Sigma_\pm^{(0)}}\) maps into
\(X_\pm^{(0)}\). As the completions of the symplectic manifolds \(Y_\pm\),
\(X_\pm^{(0)}\) come equipped with symplectic forms \(\omega^\pm\) satisfying:
\[
    (\toplevel{X},\omega^+) \cong (X \backslash L,\omega) \text{ and }
    (\bottomlevel{X},\omega^-) \cong (T^*L,\ud\lambda_\mathrm{can}).
\]
Moreover, the stretched manifolds \(X_t\) can be equipped with natural
symplectic structures \cite[Example~2.4]{cieliebakMohnkeSFTCompactness}:
\[
    \omega_t = \begin{cases}
	\omega, &\text{ on }Y_+,\\
	\ud(e^r\lambda), &\text{ on } [-t-\epsilon,\epsilon] \times M,\\
	e^{-t}\omega, &\text{ on } Y_-,
    \end{cases}
\]
which converge (on compact subsets and after rescaling by \(e^t\) in the case
of \(\bottomlevel{X}\)) to the symplectic forms \(\omega\) on
\(\toplevel{X} \cong X \backslash L\) and \(\ud\lambda_\mathrm{can}\) on
\(\bottomlevel{X} \cong T^*L\).

Returning to the scenario of a Lagrangian sphere \(L\) in the
compactification \(X = X_{d,p,q}\), since \(L\) is
disjoint from the divisor \(D' \subset X\), a cylindrical almost complex
structure \(J\) defined on the Weinstein neighbourhood\footnote{Up to
    rescaling by the Liouville flow on the neck, we can assume that \(R=1\).}
\(T_{\le 1}^*L \subset X\) of \(L\), can be extended to an element of
\(\mathcal{J}(D')\). Indeed, choose a tubular neighbourhood \(N\) of \(D'\)
disjoint from \(T_{\le 1}^*L\) and define \(J\) on \(T_{\le 1}^*L \sqcup N\)
and then extend it arbitrarily to the whole of \(X\). Since neck stretching only
alters almost complex structures in the neck region, we have that \(J_t \in
\mathcal{J}(D')\) for all \(t \ge 0\). Therefore, the results of
Section~\ref{sec:fibreModuliSpaces} apply to the neck stretch \(J_t\).

\begin{remark}
    The neck stretching setup allows us to think about points in the Weinstein
neighbourhood \(x \in Y_- = T_{\le 1}^*L\) as either
living in \(X\), or in the completion \(\bottomlevel{X} \cong T^*L\). This is
useful for phrasing statements like: given a \(J^{(0)}\)-holomorphic curve
\(F^{(0)}\), that forms part of a \(J^*\)-holomorphic building \(F\), passing
through \(x \in \bottomlevel{X}\), there exists a sequence of \emph{closed}
\(J_k\)-holomorphic curves \(f_k^x\) passing through \(x\) that converges to
\(F\). This phrasing will be used repeatedly in Section~\ref{ch:isotopy}.
\end{remark}

\subsection{An almost complex structure adapted to a Lagrangian sphere}
\label{sub:complex-structures}

In this section we construct a family of \(SO(3)\)-invariant,
\(\omega_\mathrm{can}\)-compatible almost complex structures on \(T^*S^2\)
that can be taken to be cylindrical on the complement of an arbitrarily small
neighbourhood of the zero section. The construction is similar to
Lemma~4.12 of \cite{seidel00graded}, however we make it explicit here to
ensure the anti-invariance under fibre-wise multiplication by \(-1\) property.

First we recall a result taken from Seidel's thesis \cite{seidelThesis}. Let
$\pi : \C^3 \to \C$ be the map $\pi(z) = z_1^2 + z_2^2 + z_3^2$ and denote by
$E$ the pullback of $\pi$ via the inclusion $\R \hookrightarrow \C$. Equip $E$
with the restriction of the standard K\"ahler structure
$(\Omega=\omega_{\C^3},i)$ on $\C^3$.

\begin{lemma}[{\cite[Lemma 18.1]{seidelThesis}}] 
    \label{lem:seidelIsomorphism}
    The restriction of $(E,\Omega)$ to $(0,1]$ is isomorphic to the trivial
symplectic fibre bundle $(0,1] \times (T^*S^2,\omega_{\mathrm{can}})$. An
explicit isomorphism is given by:
\[
    \Psi: z = x + \mathrm{i}y \mapsto (\pi(z),|x|^{-1}x, -|x|y).
\]
\end{lemma}

Here we have employed the usual embedding $T^*S^2 = \{ (q,p) \in \R^3 \times
\R^3 \mid |q| = 1, \langle q,p \rangle = 0 \}$, and the fact that
\(\omega_\mathrm{can} = \sum_{i=1}^3 \ud p_i\wedge\ud q_i\). Under this
identification, fibre-wise multiplication by \(-1\) \((q,p) \mapsto (q,-p)\) is
identified with complex conjugation. 
\begin{remark}
    In fact, $\Psi$ restricts to an \emph{exact} symplectomorphism $\Psi_s :
(E_s,-y_i\ud x_i) \to (\{s\} \times T^*S^2, p_i\ud q_i)$ on each fibre $E_s =
\pi^{-1}(s)$. This is critical for our situation since cylindrical
almost complex structures depend on the Liouville structures present. Note
also that $\Psi$ is equivariant with respect to the diagonal Hamiltonian
$SO(3)$-action on $\R^3 \times \R^3 = \C^3$.
\end{remark}

The almost complex structure induced by the standard one \(i\) is not
cylindrical with respect to the contact-type hypersurface \(T_1^*S^2\),
therefore, we explicitly construct one that is.
\begin{proposition}[]
    \label{prop:T*S2-acs}
    There exists a compatible almost complex structure $\JTS$ on
$T^*S^2$ that is \emph{anti-invariant} under the anti-symplectomorphism
$\varphi$ induced by complex conjugation on $E_s$, and is cylindrical away
from the zero-section. Moreover, $\JTS$ can be made cylindrical outside of an
arbitrarily small neighbourhood of $0_{S^2}$.
\end{proposition}
\begin{proof}
    Consider the manifold $W := ([0,1] \times T^*S^2) \backslash (\{0\}
\times 0_{S^2})$. Observe that $\Psi$ extends to a isomorphism
$E|_{[0,1]} \backslash \{0\} \to W$ that is exact on the fibres. Pushing forward via
$\Psi_s$ yields a family of almost complex structures $J_s$ on $T^*S^2$,
defined only away from the zero-section in the case of $J_0$. Each of these is
anti-invariant under $\varphi$ and $J_0$ is cylindrical on
$(T^*S^2 \backslash 0_{S^2}, \lambda_{\mathrm{can}}) = (\R \times T^*_1 S^2,
e^r\lambda_{\mathrm{can}}|_{T^*_1 S^2})$.

Choose a smooth function $\rho : \R \to \R$ that satisfies:
\[
\begin{aligned}
	\rho|_{(-\infty,\frac{1}{4}]} &= 1 \\
	\rho|_{[\frac{3}{4},\infty)} &= 0 \\
	\rho'|_{(\frac{1}{4}, \frac{3}{4})} &< 0.
\end{aligned}
\]
Then $\JTS(q,p) := J_{\rho(|p|)}(q,p)$ is the desired almost complex structure.
By altering $\rho$ the last claim of the result follows.
\end{proof}

\subsection{Intersection theory of punctured curves}
\label{sub:intersection-theory}

Intersection theory for punctured \(J\)-holomorphic curves was first laid out
by Siefring in \cite{siefring11intersection}. In that paper, Siefring assumes
a non-degeneracy of the contact form. However, as we
explain in this section, some of Siefring's results apply more widely (with
few alterations) to the Morse-Bott degenerate case considered in
this work.

\subsubsection{Asymptotic constraints}
\label{sec:asymptoticConstraints}

Morse-Bott degenerate Reeb orbits come in families of orbifolds of orbits
sharing the same period \(T\). We consider the especially simple
case of \((M,\lambda) = (T_1^*S^2,\lambda_\mathrm{can})\), which has the
property that every simple\footnote{A simple Reeb orbit is one with minimal
period.} Reeb orbit has the same period, and \(M \cong \RP^3\) is smoothly
foliated by them. This implies that the quotient of \(M\) by the \(S^1\)
action of the Reeb flow is a smooth 2-dimensional manifold (diffeomorphic to
\(S^2\)).

Associated to any Reeb orbit \(\gamma\) is a differential operator called an
\emph{asymptotic operator} \(\mathbf{A}_\gamma\). Somewhat counter intuitively,
knowing the precise definition will not be that useful for what
follows.\footnote{However, we note that in a unitary trivialisation of the
    bundle \(\gamma^*\xi\), one can always write \[ \mathbf{A}_\gamma =
-J_0\partial_t - S(t), \] where \(S(t)\) is a loop of symmetric matrices and
\(J_0\) is the standard complex structure on \(\C^n\).} The reader may consult
\cite[\S{}3]{wendlSFT} for the precise definition along with
justifications for the statements we make here. What is
more useful for us is a basic understanding of the spectrum
\(\sigma(\mathbf{A}_\gamma)\): it is a discrete set composed
exclusively of real eigenvalues. In dimension 4 (as we consider
here) this property furnishes the definition of a
certain winding number function defined on \(\sigma(\mathbf{A}_\gamma)\) ---
see Equation~\eqref{eq:extremalWindingNumbers}. 

A non-degenerate asymptotic operator satisfies \(0 \notin
\sigma(\mathbf{A}_\gamma)\), whereas the Morse-Bott operators we handle here
have \(\dim\ker(\mathbf{A}_\gamma) = 2\), which leads to problems in defining
the intersection number and virtual dimensions of moduli spaces of curves. The
root of the problem is that the conventional definition of the Conley-Zehnder
index \(\CZindex{}\) via the spectral flow \cite{conleyZehnder83indexTheory}
is undefined. This is resolved by perturbing the asymptotic operators by a
small number \(\epsilon \in \R \backslash \{0\}\) to make them non-degenerate
and then computing the Conley-Zehnder index of the perturbed operator. The
sign of \(\epsilon\) corresponds\footnote{The relationship between the sign of
    the perturbation and the geometric idea of whether an orbit is constrained
    or not is explained in \S{}3.2 of \cite{wendl2010autoTrans}. Specifically,
    the reader is directed to the discussion of the splitting \(T_u
    \mathcal{B} = W_\Lambda^{1,p,\mathbf{\delta}}(u^*TW) \oplus V_\Gamma
\oplus X_\Gamma\), and the proof of Proposition~3.7.} to whether the orbit is
considered as \emph{constrained} or \emph{unconstrained} in the moduli
problem. That is, whether or not we consider families of curves whose
asymptotic orbits are allowed to move freely.

To make this correspondence precise, let \(\dot{\Sigma} = \Sigma \backslash
\Gamma\) be a punctured Riemann surface, where \(\Gamma = \Gamma^+ \cup
\Gamma^-\) is the finite set of punctures partitioned into positive and
negative subsets. Suppose that \((W,\omega)\) has convex and concave
ends modelled on \((M_+,\lambda_+)\) and \((M_-,\lambda_-)\) respectively, and
let \(f : \dot{\Sigma} \to W\) be a punctured curve. 
Fix \(\epsilon \in \R \backslash \{0\}\) so that
\begin{equation}
    \label{eq:perturbSpectrum}
    [-|\epsilon|,|\epsilon|] \cap \sigma(\mathbf{A}_\gamma) \backslash \{0\} =
	\emptyset,
\end{equation}
then \(\gamma\) is constrained if the signs of the puncture \(z\) and
\(\epsilon\) agree, otherwise, it is unconstrained. We
shall write \(\gamma + \epsilon\) to denote the orbit subject to the
perturbation \(\epsilon\), which corresponds to the asymptotic operator
\(\mathbf{A}_\gamma + \epsilon\).

\begin{remark}
    \label{rem:epsilonSmallEnough}
    In the sequel we will always assume (without always saying so) that
\(\epsilon\) (or often \(\delta\)) has been chosen small enough so that
Equation \eqref{eq:perturbSpectrum} holds for a fixed finite collection of
asymptotic operators. These collections arise in situations when one considers
curves with multiple punctures, or, more generally, a finite collection of
punctured curves. When considering the asymptotic operators associated to the
orbits of a punctured curve, we will often write \(\mathbf{A}_z\) to denote
the operator associated to the asymptotic orbit of the puncture \(z \in
\Gamma\). Furthermore, when it is clear from the context, we sometimes just
write \(\mathbf{A}\).
\end{remark}

Before proceeding further, we remark on some differences in approaches
to defining the Conley-Zehnder index in the literature.

\begin{remark}
    \label{rem:CZindex}
    As already described, one can define an index in the Morse-Bott
degenerate case by making a choice \(\constraint{z} \in \R \backslash
\{0\}\) for each asymptotic operator \(\mathbf{A}_z\) and compute the index
\(\CZindex{}\) of the perturbed \emph{non-degenerate} operator \(\mathbf{A}_z
+ \constraint{z}\), as in Equation (3.5) of
\cite{wendl2010autoTrans}; or, one can compute the Robbin-Salamon index
\(\RSindex{}\) \cite{robbinSalamon93indexForPaths} of the linearised Reeb
flow as done in \cite{bourgeoisThesis}. It turns out that the latter can be
realised as a special case of the former. Indeed, as in
\cite[\S{}5]{bourgeoisThesis}, denote the linearised Reeb flow by \(\Psi^\pm\)
and let \(\mathbf{A}_\pm\) be the asymptotic operator with parallel transport
\(\Psi^\pm\). Then
\begin{align*} \CZindex{\tau}(\mathbf{A}_\pm \pm \delta) &=
    \RSindex{}(\Psi^{\pm'}) \\ &= \RSindex{}(\Psi^\pm) \mp
\frac{1}{2}\dim\ker(\Psi^\pm - I) \\ &= \RSindex{}(\Psi^\pm) \mp
\frac{1}{2}\dim\ker(\mathbf{A}_\pm), \end{align*} where \(\Psi^{\pm'}\) is the
perturbed parallel transport defined in \S{}5.2.3 of \cite{bourgeoisThesis}.
Note that the left-hand side of the first line is what appears in the index
formulas of \cite{wendl2010autoTrans}, and the last line appears in the index
formula of \(\bar{\partial}_1\) of \cite[\S{}5]{bourgeoisThesis}.

In the sequel, we shall often use the notation \(\RSindex(\gamma)\) to
represent the Robbin-Salamon index of the linearised Reeb flow around the Reeb
orbit \(\gamma\). That is, \(\RSindex(\gamma) = \RSindex(\Psi^{\pm})\).
\end{remark}

\subsubsection{Definitions and key results}
    Now let \(g : \dot{\Sigma'} \to W\) be another \(J\)-holomorphic curve. In
defining the intersection number between \(f\) and \(g\), it is important to
be precise about the orbit constraints one considers. Here, we will
usually consider the case where \emph{every} orbit of \(f\) and \(g\) is
unconstrained, although, we will also consider the opposite case, where every
orbit is constrained. To keep the notation light, we will denote these numbers
as \(i_U(f,g)\) and \(i_C(f,g)\) respectively. At times where it is clear from
the context which intersection number we are handling, we shall sometimes
refer to them both as the \emph{Siefring intersection number}, due to
Siefring's original work \cite{siefring11intersection}.

The numbers \(i_U(f,g)\) and \(i_C(f,g)\) are generalisations of the usual
intersection product for closed curves. See
\cite[\S{}4]{wendl2020Contact3folds} for a concise account of the
non-degenerate case, and \cite[\S{}4.1]{wendl2010autoTrans} for the
corresponding Morse-Bott statements that we use here. Moreover, there is a
further generalisation: that of the intersection number of two holomorphic
buildings. The mantra that we shall justify in this section is that the
intersection theory of buildings arising in the splitting (or neck stretching)
construction behaves almost exactly like that of closed curves.

Before stating the main definitions, we recall some terminology. We adopt the
conventions of Wendl, which differ slightly from Siefring's; see Remark~4.7 of
\cite{wendl2020Contact3folds} for a note on the differences. Given a Reeb
orbit \(\gamma\) let \(\tau\) denote a unitary trivialisation of the contact
bundle \(\gamma^*\xi \to S^1\) over it.\footnote{Occasionally we will also use
\(\Phi\) instead of \(\tau\).} As is customary, whenever a trivialisation over
a multiple cover is required, we use the pullback of \(\tau\)
under the covering map \(S^1 \to S^1 : \theta \mapsto k\theta\) and denote
this by \(\tau^k\) (although sometimes we abuse notation and just use
\(\tau\)). Similarly, we write \(\gamma^k\) for the pullback of \(\gamma\)
under the \(k\)-fold covering map, and we say that \(\gamma^k\) is a
\(k\)-fold cover of \(\gamma\).

The reader may consult Sections 3 and 4 of the excellent book
\cite{wendl2020Contact3folds}, or the original paper \cite{HWZpropertiesII}
for full details on the following statements about asymptotic winding numbers.
See also Section~3.1.3 of \cite{siefring11intersection} for a summary of the
situation. Relative to \(\tau\), there is a well-defined winding number
function \(\mathrm{wind}_\mathbf{A}^\tau : \sigma(\mathbf{A}) \to \Z\). The
extremal winding numbers \(\alpha_\pm^\tau\) are defined as
follows:\footnote{Observe that these definitions are sound, no matter the
degeneracy of \(\mathbf{A}\).}
\begin{equation}
    \begin{aligned}
        \label{eq:extremalWindingNumbers}
        \alpha_+^\tau(\mathbf{A}) &=
	\min\{\mathrm{wind}_\mathbf{A}^\tau(\lambda) \mid \lambda \in
	(0,+\infty) \cap \sigma(\mathbf{A}) \} \\
	\alpha_-^\tau(\mathbf{A}) &=
	\max\{\mathrm{wind}_\mathbf{A}^\tau(\lambda) \mid \lambda \in
	(-\infty,0) \cap \sigma(\mathbf{A}) \}.
    \end{aligned}
\end{equation}
We will also write \(\alpha_\pm^\tau(\gamma)\) to mean
\(\alpha_\pm^\tau(\mathbf{A}_\gamma)\). Now suppose that \(\gamma\) is a
simple Reeb orbit, then we define the numbers
\[
	\Omega^\tau_\pm(\gamma^k + \epsilon, \gamma^m + \epsilon') =
	km \min\left\{\frac{\mp\alpha_\mp^\tau(\gamma^k + \epsilon)}{k},
	\frac{\mp\alpha_\mp^\tau(\gamma^m + \epsilon')}{m}\right\}.
\]
This definition is extended to all pairs of Reeb orbits by asserting that
\[
    {\Omega_\pm^\tau(\gamma + \epsilon, \gamma' + \epsilon') = 0}
\]
if \(\gamma\)
and \(\gamma'\) are not covers of the same simple orbit.

The final ingredients needed to state the main definitions are the relative
intersection number \(f \bullet_\tau g\) of two punctured curves, and the
relative first Chern number \(c_1^\tau(E)\) associated to a complex bundle \(E
\to \dot{\Sigma}\) over a punctured curve. The former is defined to be the usual
algebraic count of intersections between \(f\) and the push-off of \(g\) in
the direction determined by \(\tau\) near infinity, and the latter is defined
via algebraic counts of zeros of sections of \(E\) that are non-zero and
constant near infinity. For precise definitions see Sections 4.2 and 3.4 of
\cite{wendl2020Contact3folds} respectively. For a punctured curve \(f :
\dot{\Sigma} \to W\) we'll often use the shorthand \(c_1^\tau(f) :=
c_1^\tau(f^*TW)\).

\begin{definition}
\begin{enumerate}
    \item Let \(f : \dot{\Sigma} \to W\) and \(g : \dot{\Sigma'} \to W\) be
punctured holomorphic curves in an almost complex manifold with cylindrical
ends \(W\), then for any \(\delta > 0\) sufficiently small, their constrained
and unconstrained intersection numbers are defined as\footnote{In this
    equation and in those that follow, the notation \(\pm\) indicates that the
sum appears twice, once with a plus sign, and again with a minus sign.}
\begin{align*}
    i_C(f,g) &:= f \bullet_\tau g - \sum_{(z,w) \in \Gamma^\pm \times
	(\Gamma')^\pm} \Omega_\pm^\tau (\gamma_z \pm \delta, \gamma_w \pm
	\delta), \\
    i_U(f,g) &:= f \bullet_\tau g - \sum_{(z,w) \in \Gamma^\pm \times
	(\Gamma')^\pm} \Omega_\pm^\tau (\gamma_z \mp \delta, \gamma_w \mp
	\delta).
\end{align*}
    \item Let \(F = \splitcurve{F}\) and \(G = \splitcurve{G}\) be holomorphic
buildings in a split manifold, then we define their intersection number to be
    \[
	i(F,G) := \sum_{i=0}^N F^{(i)} \bullet_\tau G^{(i)}.
    \]
    \item The normal Chern number of a curve \(f\) is defined as
    \[
	c_N(f) = c_1^\tau(f) - \chi(\dot{\Sigma}) \pm \sum_{z \in
	\Gamma^\pm} \alpha_{\mp}^\tau(\gamma_z \mp \delta).
    \]
\end{enumerate}
\end{definition}
These definitions are sound since the dependence on the trivialisation
\(\tau\) cancels in each of the equations. In particular see Appendix~C.5 of
\cite{wendl2020Contact3folds} for the justification of this claim for
intersections of buildings.

We now state some key results that will be crucial in our analysis of
holomorphic buildings arising from the neck stretching process. First, the
adjunction inequality, which is an immediate corollary of the adjunction
formula given in \cite[\S{}4]{wendl2010autoTrans}.

\begin{theorem}[Adjunction inequality.]
    \label{thm:adjunction}
    Let \(f : \dot{\Sigma} \to W\) be a punctured curve in a manifold with
cylindrical ends. Then,
\[
	i_U(f,f) \ge c_N(f).
\]
\end{theorem}

Next, we have a statement of continuity of the intersection products with
respect to the topology of the moduli spaces of holomorphic buildings.
This can be derived from Lemma~5.7 of \cite{siefring11intersection}.
\begin{theorem}[]
    \label{thm:intersection-cts}
If \(f_k\) and \(g_k\) are sequences of closed curves in \((X_k,J_k)\)
converging to holomorphic buildings \(F\) and \(G\) respectively under neck
stretching, then \(i(F,G) = f_k \cdot g_k\) for \(k\) large.
\end{theorem}

Finally, the following statement is our justification of the mantra that
holomorphic buildings in split manifolds behave like closed holomorphic
curves. The proof is given in Section~\ref{sec:proofOfBuildingsMantra}. No
novel techniques are involved, so it is only included for completeness.
\begin{proposition}[]
	\label{prop:intersection-additive-covers}
	\begin{enumerate}
	\item 	The intersection product for holomorphic buildings is
		additive over the levels with respect to \(i_U(\cdot,\cdot)\):
		for \(F = \splitcurve{F}\) and \(G = \splitcurve{G}\) we have
		\[ i(F,G) = \sum_{i=0}^N i_U(F^{(i)}, G^{(i)}). \]
	\item The unconstrained intersection number \(i_U(\cdot,\cdot)\) is
	    linear with respect to multiple covers. That is, if \(f\) is a
	    \(d\)-fold cover of \(f'\), then \(i_U(f,g) = d(i_U(f',g))\).
	\end{enumerate}
\end{proposition}

\subsubsection{Proof of Proposition~\ref{prop:intersection-additive-covers}}
\label{sec:proofOfBuildingsMantra}
We prove Part 1, and Part 2 will be picked up along the way as a corollary of
our techniques. The key is to understand how a common orbit \(\gamma\) of the
buildings \(F\) and \(G\) contributes to the sum \(\sum_{i=0}^N i_U(F^{(i)},
G^{(i)})\) as both a positive orbit for \((F^{(i)}, G^{(i)})\) and a negative
one for \((F^{(i+1)}, G^{(i+1)})\). Such an orbit is called a \emph{breaking
orbit}, and the corresponding pair punctures \((z,w)\) are called a
\emph{breaking pair}

Now,
\[
    i(F,G) - \sum_{i=0}^N i_U(F^{(i)},G^{(i)}) = \sum_{\substack{(z,w)\\
    \text{breaking pairs}}}
    \mathrm{Br}(\gamma_z + \delta, \gamma_w + \delta),
\]
where
\[
    \mathrm{Br}(\gamma_z + \epsilon, \gamma_w + \epsilon) :=
    \Omega^\tau_+(\gamma_z - \epsilon, \gamma_w - \epsilon) +
    \Omega^\tau_-(\gamma_z + \epsilon, \gamma_w + \epsilon).
\]
is the \emph{breaking contribution} of the breaking pair \((z,w)\).
Henceforth we suppress the \(\epsilon\) notation in \(\mathrm{Br}(\gamma +
\epsilon, \gamma' + \epsilon)\) since it is independent of the number
\(\epsilon > 0\) chosen, provided that it is small enough. It suffices to show
that \(\mathrm{Br}(\gamma,\gamma') = 0\) for all unconstrained orbits
\(\gamma\) and \(\gamma'\) in Morse-Bott degenerate  families.

Given an asymptotic operator \(\mathbf{A} = \mathbf{A}_\gamma :
\Gamma(\gamma^*\xi) \to \Gamma(\gamma^*\xi)\) of a Reeb
orbit \(\gamma\), denote its \(k\)-fold cover by \(\mathbf{A}^k =
\mathbf{A}_{\gamma^k}\). Consider the map
\begin{equation*}
	\ker(\mathbf{A}) \to \ker(\mathbf{A}^k) : f \mapsto f_k,
\end{equation*}
sending a section \(f \in \ker(\mathbf{A})\) to its pullback \(f_k\) under the
\(k\)-fold covering map \(S^1 \to S^1 : \theta \mapsto k\theta\). The next
lemma shows that it is an linear isomorphism.
\begin{lemma}[]
	Suppose that \(\mathbf{A}\) is the asymptotic operator of a hermitian
line bundle\footnote{Which is always the case for us, since we deal only with
4-dimensional manifolds.} \(E \to S^1\). Then the map \(\ker(\mathbf{A}) \to
\ker(\mathbf{A}^k) : f \mapsto f_k\) is a linear isomorphism.
\end{lemma}
\begin{proof}
	The map is well-defined since, if \(f\) is an eigenfunction of
\(\mathbf{A}\) with eigenvalue \(\lambda\), then \(f_k\) is one of
\(\mathbf{A}^k\) with eigenvalue \(k\lambda\). Injectivity is easily seen, and
surjectivity follows from the fact that the covering multiplicity of an
eigenfunction \(g\) of \(\mathbf{A}^k\) is given by
\cite[Lemma~3.2]{siefring11intersection}
\[
	\cov(g) = \gcd(\mathrm{wind}(\tau^k \circ g), k),
\]
and that the winding function is constant on eigenspaces
\cite[Lemma~3.4]{HWZpropertiesII}. Explicitly, take a non-zero \(f \in
\ker(\mathbf{A})\), then \(f_k \in \ker(\mathbf{A}^k)\) is non-trivial, so
\[
	\mathrm{wind}(\tau^k \circ g) = \mathrm{wind}(\tau^k \circ f_k) = 
		k\mathrm{wind}(\tau \circ f),
\]
and thus \(\cov(g) = k\).
\end{proof}

\begin{lemma}[]
    \label{lem:unconstrained-breaking-vanishes}
    Suppose that \(\gamma\) is a simple Morse-Bott degenerate Reeb orbit.
Then, for all integers \(k,m > 0\), we have that
\(\mathrm{Br}(\gamma^k,\gamma^m) = 0\).
\end{lemma}
\begin{proof}
	Recall the definition
\[
	\Omega^\tau_\pm(\gamma^k + \epsilon, \gamma^m + \epsilon') =
	km \min\left\{\frac{\mp\alpha_\mp^\tau(\gamma^k + \epsilon)}{k},
	\frac{\mp\alpha_\mp^\tau(\gamma^m + \epsilon')}{m}\right\}.
\]
Thus,
\[
	\mathrm{Br}(\gamma^k,\gamma^m) = \min\{m\alpha_+^\tau(\gamma^k +
	\epsilon), k\alpha_+^\tau(\gamma^m + \epsilon)\} -
	\max\{m\alpha_-^\tau(\gamma^k - \epsilon), k\alpha_-^\tau(\gamma^m -
	\epsilon)\},
\]
and the result then follows from the claim that there exists an
integer \(l\) such that, for all \(k > 0\), \(\alpha_+^\tau(\gamma^k +
\epsilon) = kl = \alpha_-^\tau(\gamma^k - \epsilon)\). Indeed, from this
it follows that all the terms involved in the min/max functions above are
equal to \(kml\).

To prove the claim, recall that \(\mathrm{wind}_\mathbf{A}^\tau :
\sigma(\mathbf{A}) \to \Z\) is monotone increasing and attains each value in
\(\Z\) exactly twice (accounting for multiplicities of eigenvalues).
Therefore, since \(\dim\ker(\mathbf{A}^k) = 2\), and \(\delta > 0\) is
chosen such that \(\delta < |\lambda|\) for each non-zero eigenvalue \(\lambda
\in \sigma(\mathbf{A}^k)\), it is clear from the definitions that (compare
Figure~\ref{fig:perturbA})
\begin{figure}\centering \begin{tikzpicture}[
    thick,
]
    \foreach \i/\offset in {1/0.2,0/0,-1/-0.2} {
	\draw (-2,\i) -- (2,\i);
	\foreach \j in {-1,0,1} {
	    \fill ($ (\j,\i)+(\offset,0) $) circle [radius=2pt];
	}
    }
    \draw[dashed] (0,1.5) -- (0,-1.5) node [below] {\(0\)};
    \draw (2,1) node [right=10pt] {\(\sigma(\mathbf{A}^k+\delta)\)};
    \draw (2,0) node [right=10pt] {\(\sigma(\mathbf{A}^k)\)};
    \draw (2,-1) node [right=10pt] {\(\sigma(\mathbf{A}^k-\delta)\)};
\end{tikzpicture}
    \caption{
	The effect of perturbing on the spectrum of an asymptotic operator
	\(\mathbf{A}\). The bullets represent elements of the spectrum.
    }
    \label{fig:perturbA}
\end{figure}
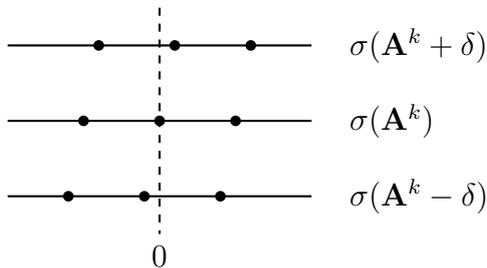

\[
	\alpha_+^\tau(\mathbf{A}^k + \delta) =
	\mathrm{wind}_{\mathbf{A}^k}^\tau(0) =
	\alpha_-^\tau(\mathbf{A}^k - \delta).
\]
The previous lemma ensures that any non-zero \(f \in \ker(\mathbf{A}^k)\) is a
\(k\)-fold cover of a non-zero element of \(\ker(\mathbf{A})\), and so,
\(\mathrm{wind}_{\mathbf{A}^k}^\tau(0) = k\mathrm{wind}_\mathbf{A}^\tau(0)\).
Setting \(l = \mathrm{wind}_\mathbf{A}^\tau(0)\) completes the proof.
\end{proof}

\begin{remark}
    \label{rem:linearity-of-extremal-winding}
    Note that we can extract a proof of Part 2 of
Proposition~\ref{prop:intersection-additive-covers} from the proof of the
claim above through the observation that it shows that\footnote{Note the
	distinction to the non-degenerate case, where this is only an
inequality in general,
\emph{cf}.~\cite[Proposition~C.2]{wendl2020Contact3folds}. See also
\cite[\S{}4.2]{wendl2010autoTrans} for a discussion on covering relations in
the degenerate case, in particular Proposition 4.6 therein, of which this
result is a special case.}
\[
	\alpha_\pm^\tau(\mathbf{A}^k \pm \delta) =
	k\alpha_\pm^\tau(\mathbf{A} \pm \delta).
\]
Indeed, this is because \(\mathrm{wind}^\tau_\mathbf{A}(0) =
\alpha_\pm^\tau(\mathbf{A} \pm \delta)\). Therefore, both components \(f
\bullet_\tau g\) and \(\Omega_\pm^\tau(\gamma^k, \gamma^l)\) scale linearly
with the covering number. More precisely, \((df) \bullet_\tau g = d (f
\bullet_\tau g)\) and \(\Omega_\pm^\tau(\gamma^{dk}, \gamma^{dl}) = d
\Omega_\pm^\tau(\gamma^k, \gamma^l)\). Hence, \(i_U(df,g) = d i_U(f,g)\).
\end{remark}

\subsubsection{Positivity of intersections for curves with degenerate
asymptotics}

Siefring's paper \cite{siefring11intersection} essentially shows that we have
positivity of intersections for punctured curves with Morse-Bott degenerate
asymptotics.
\begin{theorem}[Positivity of intersections of punctured curves]
    \label{thm:punc-pos-ints}
    Suppose that \(f,g\) are punctured curves such that \(f^{-1}(\im g)\) does
not contain a non-empty open set. Then \(i_*(f,g) \ge 0\), where the notation
\(i_*\) indicates either the constrained or unconstrained intersection number.
Moreover, equality occurs if, and only if, \(f\) and \(g\) are disjoint and
remain so after homotopy.
\end{theorem}

We sketch the proof, which essentially depends on two facts. The first is
local positivity of intersections, which is the
same story as the case of closed curves, covered in detail in
\cite{mcduffSalamonCurves}, for example. The second is the asymptotic
representation for a \(J\)-holomorphic half-cylinder. This topic has been
covered extensively in the literature, for example \cite{HWZpropertiesI,
wendl2020Contact3folds} cover the non-degenerate case, and
\cite{HWZpropertiesIV, bourgeoisThesis, behwzSFTCompactness} cover the
Morse-Bott degenerate case.

Despite the non-degeneracy assertion made in \cite{siefring11intersection},
only minor modifications are required to prove
Theorem~\ref{thm:punc-pos-ints}. The technical results on winding numbers that
power the relevant positivity of intersections results are isolated in Section
3.1.4 of \cite{siefring11intersection}. These in turn depend on two key
ingredients: the asymptotic representative of the difference between two
\(J\)-holomorphic half-cylinders (Theorem~3.6 of
\cite{siefring11intersection}\footnote{Which is stated and proved as
Theorem~2.2 in \cite{siefring08asymptotic}.}); and the properties of the
winding function\footnote{Denoted by \(w(\lambda,[\Phi])\) in
\cite{siefring11intersection}.} \(\mathrm{wind}_\mathbf{A}^\tau\)
(Lemma~3.1 of \cite{siefring11intersection}), which is a summary of the
properties proved in Lemmata~3.4--3.7 of \cite{HWZpropertiesII}. The latter of
these is valid irrespective of whether \(\mathbf{A}\) is
degenerate or not, and the former, as Siefring remarks in Section~2.2 of
\cite{siefring08asymptotic}, only depends on the exponential decay of a
half-cylinder. Therefore, one can apply the decay results of\footnote{Note
    that in \cite{HWZpropertiesIV} the manifold of unparametrised Reeb orbits
    is assumed to be a circle. The proof for more general manifolds of orbits
    (including our case of interest, the 2-sphere) can be found in either
    \cite[\S{}3.3]{bourgeoisThesis}, or
\cite[Appendix~A]{behwzSFTCompactness}, or \cite{moraThesis}.}
\cite{HWZpropertiesIV} and \cite[\S{}3.3]{bourgeoisThesis} (see also
\cite[Appendix~A]{behwzSFTCompactness}, or \cite{moraThesis}) to conclude that
Theorem~3.6 of \cite{siefring11intersection}
holds for punctured \(J\)-holomorphic curves with Morse-Bott asymptotics.

To apply the results in \cite{siefring11intersection}, one must be careful
about exactly what is meant by the Conley-Zehnder index that appears
there.\footnote{Denoted by \(\mu^\Phi\) in \cite{siefring11intersection}.}
The properties of \(\alpha_\pm^\tau\), ensure that
\[
	\alpha_\pm^\tau(\mathbf{A} \mp \delta) = \alpha_\pm^\tau(\mathbf{A}),
\]
provided that \(\delta > 0\) is sufficiently small. Therefore, if one interprets the
definition of \(\CZindex{\tau}\) given in \cite{siefring11intersection}
literally,\footnote{This is exactly Definition~3.9 of \cite{HWZpropertiesII}.
    In particular, in our situation where the Morse-Bott manifold is the whole
    contact manifold, and so \(\dim\ker\mathbf{A} = 2\), we have that the
    parity \(p(\gamma^k)\) (of the perturbed operator \(A \pm \epsilon\)) is
    \(1\). Moreover, comparing the notation of Siefring and Wendl, we have,
    for an integer \(k \in \Z \backslash \{0\}\),
    \[ \alpha^\Phi(\gamma^k) = \pm \alpha_\mp^\Phi(\gamma^{|k|}) =
	\pm\alpha_\mp^\Phi(\gamma^{|k|} \pm \epsilon), \] where the leftmost
	sign agrees with the distinction of \(\gamma\) as a positive or
	negative orbit. It then follows that \[ \mu^\Phi(\gamma^k) = 2
	    \alpha^\Phi(\gamma^k) + p(\gamma^k) = \pm
	    (2\alpha_\mp^\Phi(\gamma^{|k|}) \pm 1) =
    \pm\CZindex{\Phi}(\gamma^{|k|} \pm \epsilon), \] where the final equality
follows from equation (2.3) of \cite{wendl2010autoTrans}.} then the
constructions therein\footnote{To directly compare the notation that appears
    in \cite{siefring11intersection} with that used here, one can work through
    the definitions to show that
\[
	m_zm_w\max\left\{\frac{\alpha^\Phi(\gamma_z^{m_z})}{|m_z|},
	\frac{\alpha^\Phi(\gamma_w^{m_w})}{|m_w|}\right\} =
	-\Omega_\pm^\Phi(\gamma_z^{|m_z|},\gamma_w^{|m_w|}).
\]
    The numbers on the left hand side are used to define the intersection in
    \cite{siefring11intersection}, and those on the right are used to define
    the constrained intersection in \cite{wendl2010autoTrans}.} give the
constrained intersection number \(i_C(\cdot,\cdot)\). Therefore, positivity of
intersections holds for \(i_C\). A special case of Proposition~4.11 of
\cite{wendl2010autoTrans} is the inequality
\[
	i_U(u,v) \ge i_C(u,v),
\]
which implies positivity of intersections for \(i_U\) too.

\subsection{\(J\)-holomorphic planes in \(T^*S^2\)}
\label{sub:cotangentPlanes}

Hind proved\footnote{The paper \cite{hind2004Spheres} combined with results on
automatic transversality.} that, for any cylindrical almost complex structure
\(J\) on \(T^*S^2\), there are two transverse foliations
\(\cotangentPlanesModuli{}_\pm\) by finite-energy \(J\)-holomorphic planes.
They are distinguished by the sign of the intersection with the zero-section
\(0_{S^2} \subset T^*S^2\):
\[
	P_\pm \in \cotangentPlanesModuli{}_\pm \Rightarrow P_\pm \cdot 0_{S^2} = \pm 1.
\]
As remarked by Evans \cite[\S{}6.4]{evans2010delPezzo}, two planes
\(P_\pm \in \cotangentPlanesModuli{}_\pm\) of opposite parity intersect positively exactly
once if, and only if, they have distinct asymptotic orbits. Otherwise, they
are disjoint. However, the Siefring intersection ``sees'' this intersection,
even for planes with a common orbit:

\begin{lemma}[]
    \label{lem:cotangent-planes-intersection}
    \begin{enumerate}
	\item The unconstrained intersection of two planes of opposite parity
		is 1: \(i_U(P_+,P_-) = 1\).
	\item 	Let \(P_\pm, Q_\pm \in \cotangentPlanesModuli{}_\pm\) be two
		planes of the same parity. Then \(i_U(P_\pm,Q_\pm) = 0\).
    \end{enumerate}
\end{lemma}
\begin{proof}
	\begin{enumerate}
		\item 	As the Siefring intersection is homotopy invariant, we
		can perturb the plane \(P_+\) to one nearby, say \(P_+'\) in
		its moduli space, guaranteeing that \(P_+'\) and \(P_-\) have
		distinct orbits. Then \(i_U(P_+',P_-) = P_+' \cdot P_- = 1\),
		as already noted.
		\item 	As above, we can homotope one of the planes so that it
		has distinct asymptotic to the other, ensuring that the
		asymptotic contribution to the intersection product is zero
		and thus \(i_U(P_\pm,Q_\pm) = P_\pm \cdot Q_\pm = 0\), where
		the last equality follows since planes of the same parity form
		a foliation.
	\end{enumerate}
\end{proof}

\begin{proposition}[]
	\label{prop:cotangent-curves-intersect-planes}
	Let \(u : \dot{\Sigma} \to T^*S^2\) be a punctured \(J\)-holomorphic
curve. Then exactly one of the following is true:
\begin{enumerate}
	\item \(u\) is a cover of a plane in \(\cotangentPlanesModuli{}_\pm\); or,
	\item \(u\) intersects both types of plane positively. More precisely,
		\(i_U(u,P_\pm) > 0\).
\end{enumerate}
\end{proposition}
\begin{proof}
	Assume the first scenario is false, that is, \(u\) is not a cover of
any \(P_\pm \in \cotangentPlanesModuli{}_\pm\). Then, Theorem~\ref{thm:punc-pos-ints}
implies that \(i_U(u,P_\pm) \ge 0\). Note that this inequality must be strict
since \(\cotangentPlanesModuli{}_\pm\) is a foliation, and so \(u\) intersects one of the
planes \(P_\pm\) in at least one point.
\end{proof}

\begin{corollary}
    \label{cor:bottomLevelNonNegInts}
    Any unconstrained intersection of punctured \(J\)-holomorphic curves in
\(T^*S^2\) is non-negative.
\end{corollary}

\subsection{Analysis of holomorphic limit buildings}
\label{sec:buildingAnalysis}
\subsubsection{Consequences of neck stretching about a Lagrangian sphere}
\label{sec:stretchingAboutSphere}

The following result can be found in essential form in
\cite[p.315]{hind2004Spheres}. See also an expanded proof in
\cite[Lemma~7.5]{evans2010delPezzo}. It is a simple consequence of a curve
in a symplectisation having positive \(\mathrm{d}\lambda\)-energy and the
topology of \(\RP^3\). We recall the argument here as we will need a modified
version of it later.
\begin{lemma}[]
	\label{lem:connectedTopImpliesNoIntermediateLevels}
	Suppose that \(G = \splitcurve{G} : \Sigma^* \to X^*\) is a genus 0
holomorphic building such that \(\toplevel{\Sigma}\) is connected and
\(\toplevel{G}\) has only simple punctures. Then \(N = 0\), that is, \(G\)
has no non-trivial symplectisation levels.
\end{lemma}
\begin{proof}
    Let \(g = (a,u) : \dot{\Sigma} \to \R \times T_1^*L\) be a connected
component of the symplectisation level \(G^{(N)}\) --- the one connected to
the top level \(\toplevel{G}\). The maximum principle implies that \(g\) has
at least one positive puncture. On the other hand, since the
genus\footnote{The genus of a holomorphic building \(G : \Sigma^* \to X^*\) is
the genus of the surface \(\bar{\Sigma}\) obtained by gluing \(\Sigma^*\)
together along its punctures. We will always deal with genus 0 buildings.} of
\(G\) is 0 and \(\toplevel{G}\) is connected, \(g\) has \emph{at most}, and
therefore exactly, one positive puncture \(\gamma\). Moreover, it must have at
least one negative puncture, as otherwise \(\dot{\Sigma}\) would be the
complex plane and thus \(g\) would represent a contraction of a simple Reeb
orbit, which represents the non-trivial element in \(\pi_1(\RP^3) \cong
\Z_2\).

Next we consider the \(E_{\mathrm{d}\lambda}\)-energy of \(g\), defined
in \cite[\S{}5.3]{behwzSFTCompactness} as\footnote{In the notation of that
paper, we have \(\omega = \mathrm{d}\lambda\).}
\[
	E_{\mathrm{d}\lambda}(g) = \int_{\dot{\Sigma}} u^*\mathrm{d}\lambda.
\]
By Lemmata~5.4 and 5.16 of \cite{behwzSFTCompactness} we have that this energy is
non-negative and equal to the difference between the sum of the periods of the
positive punctures and the sum of those of the negative ones. Thus, since
\(\gamma\) is simple and all simple Reeb orbits in
\((T_1^*S^2,\lambda_\mathrm{can})\) have the same period, there is at most one
negative puncture, and therefore, exactly one. Whence we obtain
\(E_{\mathrm{d}\lambda}(g) = 0\). Theorem~6.11 of \cite{HWZpropertiesII} then
implies that \(g\) is a reparametrisation of a trivial cylinder. Repeating
this argument shows that every symplectisation level is composed solely of
trivial cylinders, a situation which is ruled out by the stability condition.
Hence the result is proved.
\end{proof}

The index formula of a punctured curve will be crucial in the proofs of this
section. The index\footnote{The name index comes from the fact that it is
    equal to the Fredholm index of a certain operator derived from the
Cauchy-Riemann equation (see \cite{bourgeoisThesis} for example).} of a curve
is the expected dimension of the moduli space it lives in, and thus it is
often also called the virtual dimension. Let \(u :
\dot{\Sigma} \to W\) be a punctured \(J\)-holomorphic curve with \(k_\pm\)
positive/negative punctures mapping into a symplectic manifold with
cylindrical ends \(W\). The index of \(u\) is given by\footnote{Observe that
    this is the \emph{unconstrained} index. Moreover, in the case that \(u\)
is a closed curve, it reduces to the index formula given in
Equation~\eqref{eq:closedIndexFormula}.}
\[
	\ind(u) = 2c_1^\tau(u) - \chi(\dot{\Sigma}) \pm
	    \sum_{i=1}^{k_\pm} \CZindex{\tau}(\mathbf{A}_{z_i} \mp \delta).
\]
In the notation of \cite{bourgeoisThesis} and \cite{hind2004Spheres}, this
becomes\footnote{See Remark~\ref{rem:CZindex}.}
\[
	\ind(u) = 2c_1^\tau(u) - \chi(\dot{\Sigma}) \pm
		\sum_{i=1}^{k_\pm} \left(\RSindex(\gamma_i^\pm) \pm
			\frac{1}{2} \dim(\gamma_i^+)\right),
\]
where \(\dim(\gamma)\) is the dimension of the moduli space of
\emph{unparametrised} Reeb orbits that \(\gamma\) lives in.
In our situation, we deal only with genus zero curves with asymptotics
that live in 2-dimensional families. Fixing
\(\tau\) to be the trivialisation that appears in Lemma~7 of \cite{hind2004Spheres}, then we
have \(\RSindex(\gamma_i^\pm) = 2\cov(\gamma_i^\pm)\), so the index
formula reduces to:
\begin{equation}
    \label{eq:indexFormula}
	\ind(u) = 2\left( k_+ + k_- - 1 + c_1^\tau(u) \pm
		\sum_{i=1}^{k_\pm} \cov(\gamma_i^\pm)\right).
\end{equation}

We primarily deal with the cases where one of \(k_\pm\) is zero, which
simplifies the formula further. In particular, in the case where \(u\) maps
into the top level \(\toplevel{X}\), we have \(k_+ = 0\) and so, combined with
the fact that \(\cov(\gamma) \ge 1\), we obtain the inequality:\footnote{We
caution the reader that this inequality is \emph{only} valid with respect to
the fixed trivialisation \(\tau\).}
\begin{equation}
	\label{eq:index-ineq}
	\ind(u) \le 2(c_1^\tau(u) - 1).
\end{equation}

\begin{lemma}[]
    \label{lem:relChern=absChern}
    Let \(F = \splitcurve{F}\) be a holomorphic building. Then the relative
first Chern numbers of every level except the top one vanish. In particular,
the relative first Chern number of the top level equals the first Chern number
of the homology class \([\bar{F}]\):
\begin{equation}
    \label{eq:relChern=absChern}
    c_1^\tau(\toplevel{F}) = c_1([\bar{F}]).
\end{equation}
\end{lemma}
\begin{proof}
The chosen trivialisation \(\tau\) extends to a global trivialisation of
\(TX^{(i)}\) and \(T\bottomlevel{X}\) for the symplectisation levels \(X^{(i)} =
\R \times T_1^*S^2\) and bottom level \(\bottomlevel{X} = T^*S^2\). Indeed, for
the symplectisation levels we have the global splitting
\[
    T(\R \times T_1^*S^2) = \langle Z,R\rangle \oplus \xi,
\]
where \(Z\) and \(R\) are the Liouville and Reeb vector fields respectively.
For \(\bottomlevel{X} = T^*S^2\) the extension is facilitated by the global
splitting of \(T(T^*S^2)\) into vertical and horizontal Lagrangian planes
\cite[Lemma~7]{hind2004Spheres}. Therefore, we have \(c_1^\tau(u) = 0\) for
any punctured curve mapping into \(X^{(i)}\) or \(\bottomlevel{X}\).
Combining this with the fact\footnote{This is a consequence of gluing the
levels \(F^{(i)}\) back together to obtain the cycle \([\bar{F}]\) in \(X\).} that,
for a holomorphic building \(F = \splitcurve{F}\), we have
\[
    \sum_{i=0}^N c_1^\tau(F^{(i)}) = c_1([\bar{F}]),
\]
we obtain the result.
\end{proof}

\begin{remark}
    Compare the above discussion with section~6.3 of \cite{evans2010delPezzo},
where Evans gives a proof using complex geometry and the compactification of
\(T^*S^2\) to the projective quadric surface.
\end{remark}

We now turn to recording a genericity result. This will be achieved by
perturbing the almost complex structure \(\toplevel{J}\) on the top level
\(\toplevel{X} \cong X \backslash L\) in a suitable open set. This set must be
chosen carefully though, as an arbitrary perturbation would either destroy the
\(c_1 \le 0\) curves in the divisor \(D' \subset X\), or the cylindrical
nature of \(\toplevel{J}\) in the neck region \((-\infty,\epsilon) \times
T_1^*L\). Let \(V\) be the open set \(\toplevel{X} \backslash
\left((-\infty,\epsilon] \times T_1^*L \cup D'\right)\). We are then free to
perturb \(\toplevel{J}\) in the set \(V\) to make it generic. This means that
any \(\toplevel{J}\)-holomorphic curve \(u\) mapping an injective point into
\(V\) is Fredholm regular, and thus \(\ind(u) \ge 0\).
\begin{lemma}[]
    \label{lem:neg-puncs-implies-regular}
    Let \(\toplevel{J}\) be as above and \(u : \dot{\Sigma} \to \toplevel{X}\) be a
somewhere injective \(J\)-holomorphic curve that is either closed and not
contained in \(D'\) or has at least one (negative) puncture. Then \(u\)
intersects \(V \subseteq \toplevel{X}\) and is thus Fredholm regular.
\end{lemma}
\begin{proof}
    First we deal with the closed case. Since the cylindrical part
\({(-\infty,\epsilon) \times T_1^*L}\) of \(\toplevel{X}\) is an exact
symplectic manifold, it contains no non-constant closed \(J\)-holomorphic
curves. Therefore, \(u\) must pass through \(V\).

    For the punctured case, a similar argument works. Suppose that the result
is false. Then the image of \(u\) is contained entirely in
\((-\infty,\epsilon) \times T_1^*L\), which is an exact, cylindrical
symplectic manifold with positive \emph{and} negative boundary. However,
\(u\) has no positive punctures, which contradicts the maximum principle.
Therefore, \(u\) must intersect \(V\).
\end{proof}

\begin{corollary}[]
    \label{cor:relChernGE1}
    Let \(u\) be a closed or punctured \(J\)-holomorphic curve in
\(\toplevel{X}\) that is not contained in the divisor \(D'\). Then,
    \begin{equation}
    	\label{eq:c1-ge-1}
    	c_1^\tau(u) \ge 1.
    \end{equation}
\end{corollary}
\begin{proof}
    If \(u\) is not simple, then by Theorems~2.35 and 6.34 of \cite{wendlSFT},
we may pass to its underlying simple curve \(\tilde{u}\).
Lemma~\ref{lem:neg-puncs-implies-regular} ensures that \(\tilde{u}\) is
Fredholm regular, and so
\[
    \ind(\tilde{u}) \ge 0.
\]
The inequality \eqref{eq:index-ineq} then implies that \(c_1^\tau(\tilde{u})
\ge 1\), so the result follows by linearity of \(c_1^\tau\) with respect to
covers.
\end{proof}

\subsubsection{Analysis of buildings: limits of \(c_1 = 1\) curves}
\label{sub:analysis-buildings-c1=1}

In this section we consider sequences of closed curves in a homology class \(E
\in H_2(X)\) satisfying \(E^2 = -1\) (which is equivalent to \(c_1(E) = 1\)
for embedded genus zero curves, by the adjunction formula). Recall that in
homology, \(L = \mathcal{E}_i - \mathcal{E}_j\) for some \(1 \le i \ne j \le
d\). Since the indices \(i\) and \(j\) won't play a role in what follows, we
do away with them and define\footnote{This notation is chosen so that
    \(\mathcal{E}^\pm \cdot L = \pm 1\). Although, note that \(\mathcal{G}^\pm
    \cdot L = \mp 1\), which might seem confusing, however the author believes
    that the other definition of \(\mathcal{G}^\pm = F - \mathcal{E}^\mp\)
required to achieve \(\mathcal{G}^\pm \cdot L = \pm 1\) is worse.}
\begin{equation}
    \label{eq:EplusMinus}
    \begin{aligned}
	\Eplus &= \mathcal{E}_j,\\
	\Eminus &= \mathcal{E}_i,
    \end{aligned}
    \qquad\text{and}\qquad
    \begin{aligned}
	\FEplus &= F - \Eplus,\\
	\FEminus &= F - \Eminus.
    \end{aligned}
\end{equation}

The main result of this section, proved in
Propositions~\ref{prop:c1=1-sft-limit} and \ref{prop:top-level-2-planes}, is:
\begin{proposition}[]
    Let \(E \in H_2(X)\) satisfy \(c_1(E) = 1\) and suppose we have a sequence
\(e_k \in \mathcal{M}_{0,0}(X,E;J_k)\) of embedded \(J\)-holomorphic curves
converging to a limit building \(F\) under neck stretching about \(L\). Then
\(F = (\toplevel{F},\bottomlevel{F})\) has no symplectisation levels,
and moreover \(\bottomlevel{F}\) is non-empty if, and only if, \(E \cdot L
= E \cdot (\Eminus - \Eplus) \ne 0\), in which case \(\bottomlevel{F}\)
consists of a union of planes of the same parity in the moduli spaces
\(\mathcal{M}_\pm\) in \(T^*S^2\). The parity of the planes coincides with the
sign of the intersection \(E \cdot L\).

Furthermore, in the case \(E \in \{\Eplusminus,\FEplusminus\}\), the buildings
\(F_{\Eplusminus}\) arising from limits of curves in the classes
\(\Eplusminus\) have identical top levels. The analogous result holds for
\(\FEplusminus\). We write these relationships as
\[
    \toplevel{(F_{\Eplus})} = \toplevel{(F_{\Eminus})} \text{ and }
    \toplevel{(F_{\FEplus})} = \toplevel{(F_{\FEminus})}.
\]
On the other hand, the four planes \(\bottomlevel{(F_{\Eplusminus})}\),
\(\bottomlevel{(F_{\FEplusminus})}\) in \(T^*S^2\) are pairwise distinct. 
\end{proposition}

\begin{remark}
    \label{rem:bottomPlanesAreNodalFibres}
    \begin{enumerate}
	\item In the case \(\bottomlevel{F} = \emptyset\) in the above result, this means
	    that \(\toplevel{F}\) is actually a \emph{closed} holomorphic curve in
	    \(\toplevel{X}\).
	\item In a later section, we will show that the two pairs of planes
	    \(\bottomlevel{(F_{\Eplusminus})} \cup \bottomlevel{(F_{\FEplusminus})}\)
	    are the nodal fibres of a Lefschetz fibration on \(T^*S^2\) such
	    that \(L = 0_{S^2}\) is a matching cycle.
    \end{enumerate}
\end{remark}

\begin{lemma}[]
    \label{lem:limits-of-c1-curves-have-connected-top-level}
    Suppose that \(F = \splitcurve{F}\) is the limiting building of a
sequence of curves in the moduli spaces \(\mathcal{M}_{0,0}(X,E;J_k)\), where
\(E \in H_2(X;\Z)\) satisfies \(c_1(E) = 1\). Then \(\toplevel{F}\) has
connected domain, \emph{i.e.}, it consists of exactly one component, and each
of its punctures is simple.
\end{lemma}
\begin{proof}
    If \(\toplevel{F}\) had more than one component then at least one
would satisfy \(c_1^\tau \le 0\), since
\[
    \sum_{\substack{\text{components }f \\ \text{of }\toplevel{F}}}
    c_1^\tau(f) = c_1^\tau(\toplevel{F}) = c_1([\bar{F}]) = c_1(E) = 1,
\]
by Lemma~\ref{lem:relChern=absChern}. However, this is impossible by
Corollary~\ref{cor:relChernGE1}. Therefore, \(\toplevel{F}\) consists of
exactly one component which satisfies \(c_1^\tau(\toplevel{F}) = 1\). Since
\(\toplevel{F}\) is Fredholm regular by
Lemma~\ref{lem:neg-puncs-implies-regular}, we must have \(\ind(\toplevel{F}) =
0\) by \eqref{eq:index-ineq}. This implies that \(k_- - \sum_{i=1}^{k_-}
\cov(\gamma_i^-) = 0\), which implies that the punctures of \(\toplevel{F}\)
are simple.
\end{proof}

Hence, by Lemma~\ref{lem:connectedTopImpliesNoIntermediateLevels}, \(F =
(\toplevel{F}, \bottomlevel{F})\) is a holomorphic building with only a top
and bottom level. Therefore, by Theorem~\ref{thm:intersection-cts} and
Proposition~\ref{prop:intersection-additive-covers} we have
\begin{equation}
	\label{eq:top-and-bottom-int-equals--1}
	\selfastint{\toplevel{F}} + \selfastint{\bottomlevel{F}} =
	i(F,F) = E^2 = -1.
\end{equation}
The adjunction formula implies that \(\selfastint{\toplevel{F}} \ge
c_N(\toplevel{F})\), and so we obtain the inequality
\[
	\selfastint{\bottomlevel{F}} \le -1 - c_N(\toplevel{F}).
\]

\begin{lemma}[]
    \label{lem:cN=-1-for-c1=1-curves}
    The normal Chern number \(c_N(\toplevel{F}) = -1\), and thus
    \[
	\selfastint{\bottomlevel{F}} \le 0.
    \]
\end{lemma}
\begin{proof}
	By definition,
\[
	c_N(u) = c_1^\tau(u^*TW) - \chi(\dot{\Sigma}) \pm \sum_{z \in
	\Gamma^\pm} \alpha_{\mp}^\tau(\gamma_z \mp \delta).
\]
We have the formulas \cite[\S{}3.2]{wendl2010autoTrans}
\begin{align*}
	2\alpha_\pm^\tau(\gamma + \epsilon) &= \CZindex{\tau}(\gamma +
	\epsilon) \pm p(\gamma + \epsilon), \\
	\CZindex{\tau}(\gamma \mp \delta) &= 2\cov(\gamma) \pm 1,
\end{align*}
where \(p(\gamma + \epsilon) := \alpha_+^\tau(\gamma+\epsilon) -
\alpha_-^\tau(\gamma + \epsilon)\) is called the \emph{parity}. From these we
obtain
\[
	\alpha_\mp^\tau(\gamma \mp \delta) = \cov(\gamma).
\]
Therefore, for a curve with only simple asymptotics,
\begin{equation}
	\label{eq:cN-simple-punctures}
	c_N(u) = c_1^\tau(u^*TW) - (2 - k_+ - k_-) \pm k_\pm = c_1^\tau(u^*TW)
	+ 2(k_+ - 1).
\end{equation}
In particular, for \(u = \toplevel{F}\) we have \(k_+ = 0\) and
\(c_1^\tau(\toplevel{F}) =  1\), resulting in \(c_N(\toplevel{F}) = -1\).
\end{proof}

\begin{proposition}[]
	\label{prop:c1=1-sft-limit}
	Let \(F = (\toplevel{F},\bottomlevel{F})\) be the limit
building of a sequence of curves in the moduli spaces
\(\mathcal{M}_{0,0}(X,E;J_k)\), where \(E \in H_2(X;\Z)\) satisfies
\(c_1(E) = 1\) and \(E \cdot L = \iota\). Then \(\toplevel{F}\) is a sphere
with exactly \(|\iota|\) punctures and
\[
	\bottomlevel{F} \text{ consists of exactly } 
	\begin{cases}
		\iota \quad P_+ \text{ planes}, &\text{ if } \iota \ge 0,
		\text{ or}\\
		|\iota| \quad P_- \text{ planes}, &\text{ if } \iota < 0.
	\end{cases}
\]
\end{proposition}

\begin{corollary}[]
    \label{cor:En-disjoins-L}
    Suppose that the sequence in the above proposition is given by the unique
curves in \(\mathcal{M}_{0,0}(X,E_n;J_k)\). Then \(\bottomlevel{F}\) is empty
and \(\toplevel{F}\) is a smooth closed curve in \(\toplevel{X}\). 
\end{corollary}

\begin{remark}
    We paraphrase this result by saying that the \(E_n\) curve disjoins from
\(L\) under neck stretching.
\end{remark}

\begin{corollary}[]
    \label{cor:lim-nodal-fibres}
    Suppose that \(E\) is one of the classes in Equation~\eqref{eq:EplusMinus}.
Then \(\bottomlevel{F}\) consists of exactly one plane in the moduli spaces
\(\cotangentPlanesModuli{}_\pm\) of planes in \(T^*S^2\) defined in
Section~\ref{sub:cotangentPlanes}. The sign is given by the sign of \(E \cdot
L\).
\end{corollary}

\begin{remark}
	It is important (albeit obvious) to note that the planes in \(T^*S^2\)
arising as limits of curves in the classes \(E\) and \(F - E\) have opposite
parity. That is, one is in \(\cotangentPlanesModuli{}_+\) and the other in
\(\cotangentPlanesModuli{}_-\).
\end{remark}

\begin{proof}[Proof of Proposition \ref{prop:c1=1-sft-limit}]
	The genus \(g\) of \(F\) is
0, and so each component of \(\bottomlevel{F}\) is a genus 0 punctured
curve. Moreover, as \(\toplevel{F}\) has only 1 component, the topology of
\(\bottomlevel{F}\) must be a union of discs --- anything else would
increase \(g\). Due to the classification
of simple \(J\)-holomorphic planes of \(T^*S^2\)
(Proposition~\ref{prop:cotangent-curves-intersect-planes}), we deduce that
\(\bottomlevel{F}\) is a union of \(P_\pm\) planes. We write this
as\footnote{Note that the \(\sqcup\) notation only means that the domain of
	\(\bottomlevel{F}\) is a disjoint union of Riemann surfaces, its
image need not be a disjoint union of planes. Moreover, the use of the plus
sign in the equation \(\bottomlevel{F} = m_+P_+ + m_-P_-\) is meant to reflect
the fact that the intersection product is additive over unions.}
\(\bottomlevel{F} = m_+P_+ + m_-P_- := \bigsqcup_{m_+}P_+ \sqcup
\bigsqcup_{m_-}P_-\), for some integers \(m_\pm \ge 0\). The results of
Lemma~\ref{lem:cotangent-planes-intersection} and additivity of the Siefring
intersection over disjoint unions
\cite[Proposition~4.3(3)]{siefring11intersection} then imply that
\[
	\selfastint{\bottomlevel{F}} = 2m_+m_- \ge 0.
\]
Combining this with Lemma~\ref{lem:cN=-1-for-c1=1-curves} yields that
\(\selfastint{\bottomlevel{F}} = 0\), and thus at least one of \(m_\pm\) is
zero. The equation
\[
	m_+ - m_- = \bottomlevel{F} \cdot L = E \cdot L = \iota
\]
then completes the proof.
\end{proof}

\begin{remark}
    \label{rem:repeatedSFTCompactness}
    Observe that, through repeated applications of the SFT compactness
theorem, we can produce a single neck-stretching sequence \(J_k \to
J_\infty\) such that, for every \(E \in \{E_n, \Eplus, \Eminus, \FEplus,
\FEminus\}\), the curves in the moduli spaces
\(\mathcal{M}_{0,0}(E;J_k)\) converge to \(J^*\)-holomorphic buildings
as \(k \to \infty\). Indeed, fix a particular class \(E\) and sequences
\(J_k \to J_\infty\) and \(u_k \in \mathcal{M}_{0,0}(E;J_k)\). The
compactness theorem produces a subsequence \(k\) such that \(u_k\)
converges. We can then apply SFT compactness to a sequence \(v_k \in
\mathcal{M}_{0,0}(E';J_k)\) for some \(E' \ne E\) to obtain another
subsequence, and so on and so forth.

    This idea can be generalised to include sequences of curves
in classes other than the exceptional ones used above, provided that the
sequences of curves are somehow fixed. For example, we could make a point
constraint. When dealing with the classes \(E\) above there was no need to do
this since the moduli spaces \(\mathcal{M}_{0,0}(E;J)\) are just single
points. 

    We shall use this principle repeatedly in the sequel, although we will
not always make reference to it to save cluttering the narrative. Suffice to
say that if we need to compare two \(J^*\)-holomorphic buildings arising
from sequences in the classes \(\alpha, \beta \in H_2(X)\), then this will be
implicitly done with respect to a single neck-stretching sequence for which
both the corresponding sequences of curves converge.
\end{remark}

We now turn to deriving relations between the limits of curves with homology
classes in \(\{\Eplus,\Eminus,\FEplus,\FEminus\}\).
\begin{proposition}[]
    \label{prop:top-level-2-planes}
    The holomorphic limit buildings \(F_{\Eplus} =
    (\toplevel{(F_{\Eplus})},\bottomlevel{(F_{\Eplus})})\), \(F_{\Eminus} =
    (\toplevel{(F_{\Eminus})}, \bottomlevel{(F_{\Eminus})})\), \(F_{\FEplus} =
    (\toplevel{(F_{\FEplus})},\bottomlevel{(F_{\FEplus})})\), and
    \(F_{\FEminus} =
    (\toplevel{(F_{\FEminus})}, \bottomlevel{(F_{\FEminus})})\) satisfy
\[
	\toplevel{(F_{\Eplus})} = \toplevel{(F_{\Eminus})} \qquad \text{and} \qquad
		\toplevel{(F_{\FEplus})} = \toplevel{(F_{\FEminus})}.
\]
\end{proposition}
\begin{proof}
	Corollary~\ref{cor:lim-nodal-fibres} and
Proposition~\ref{prop:intersection-additive-covers} imply that
\[
	0 = \Eminus \cdot \Eplus = i_U(\toplevel{(F_{\Eminus})},
	\toplevel{(F_{\Eplus})}) + i_U(\bottomlevel{(F_{\Eminus})}, \bottomlevel{(F_{\Eplus})}) =
	i_U(\toplevel{(F_{\Eminus})}, \toplevel{(F_{\Eplus})}) + 1.
\]
Therefore, by positivity of intersections and the fact that
\[
    c_1^\tau(\toplevel{(F_{\Eminus})} ) = c_1^\tau( \toplevel{(F_{\Eplus})} ) = 1,
\]
we must have that
\(\toplevel{(F_{\Eminus})}\) is a reparametrisation of \(\toplevel{(F_{\Eplus})}\), that
is, \(\toplevel{(F_{\Eminus})} = \toplevel{(F_{\Eplus})}\). Similarly, \(\toplevel{(F_{\FEminus})}
= \toplevel{(F_{\FEplus})}\).
\end{proof}

Recall the so-called Morse-Bott contribution to the intersection number
\cite[\S{}4]{wendl2010autoTrans}:
\[
	i_\mathrm{MB}^\pm(\gamma + \epsilon, \gamma' + \epsilon') :=
	\Omega^\tau_\pm(\gamma, \gamma') - \Omega^\tau_\pm(\gamma + \epsilon,
	\gamma' + \epsilon').
\]
A simple calculation shows that this integer is non-negative and that it
satisfies
\[
	i_U(u,v) = i_C(u,v) + \sum_{(z,w) \in \Gamma_u^\pm \times \Gamma_v^\pm}
	i_\mathrm{MB}^\pm(\gamma_z \mp \delta, \gamma_w \mp \delta).
\]

\begin{lemma}[]
	Let \(\gamma\) and \(\gamma'\) denote the asymptotic orbits of
\(\toplevel{(F_{\Eplus})}\) and \(\toplevel{(F_{\FEplus})}\) respectively. Then \(\gamma' \ne
\gamma\), that is, they are geometrically distinct orbits.
\end{lemma}
\begin{proof}
	Plugging \((u,v) = (\toplevel{(F_{\Eplus})},\toplevel{(F_{\FEplus})})\) into the above
formula yields
\[
	i_U(\toplevel{(F_{\Eplus})},\toplevel{(F_{\FEplus})}) = i_C(\toplevel{(F_{\Eplus})},\toplevel{(F_{\FEplus})}) +
	i_\mathrm{MB}^-(\gamma+\delta, \gamma'+\delta).
\]
Recall that each fibration \(\pi_J\) has two distinguished \(J\)-holomorphic
sections, one of class \(\fibrationSectionL\), and the other
\(\fibrationSectionR=H-E_0-\sum_{j=1}^d\mathcal{E}_j\), and that \(\Eplus\)
intersects \(S'\) and not \(S\), and vice versa for \(\FEplus\). It is then
trivial that \(\toplevel{(F_{\Eplus})}\) and \(\toplevel{(F_{\FEplus})}\) are
geometrically distinct, and so positivity of intersections gives
\(i_C(\toplevel{(F_{\Eplus})},\toplevel{(F_{\FEplus})}) \ge 0\) and thus
\[
	i_U(\toplevel{(F_{\Eplus})},\toplevel{(F_{\FEplus})}) \ge
	    i_\mathrm{MB}^-(\gamma+\delta, \gamma'+\delta).
\]
Observe that \([\Eplus] \cdot [\FEplus] = 1\), and, since \(\bottomlevel{(F_{\Eplus})}\) and
\(\bottomlevel{(F_{\FEplus})}\) are planes of opposite parity in \(T^*S^2\), 
\(i_U(\bottomlevel{(F_{\Eplus})}, \bottomlevel{(F_{\FEplus})}) = 1\). The additivity of the
unconstrained intersection product then implies that
\[
	i_U(\toplevel{(F_{\Eplus})}, \toplevel{(F_{\FEplus})}) = 0,
\]
Combining this with the above inequality, we obtain
\(i_\mathrm{MB}^-(\gamma+\delta, \gamma'+\delta) = 0\). Therefore, the result
follows from the next lemma.
\end{proof}

\begin{lemma}[]
	\label{lem:mbContributions}
	For a Morse-Bott degenerate asymptote \(\gamma\), we have
\[
	i_\mathrm{MB}^\pm(\gamma^k \mp \delta, \gamma^l \mp \delta) =
	\min\{l,k\} > 0.
\]
\end{lemma}
\begin{proof}
	This is a simple consequence of the definitions and the properties of
asymptotic operators satisfying \(\dim\ker\mathbf{A} = 2\). Indeed,
\begin{align*}
	i_\mathrm{MB}^\pm(\gamma^k \mp \delta, \gamma^l \mp \delta) &=
		\Omega^\tau_\pm(\gamma^k, \gamma^l) - \Omega^\tau_\pm(\gamma^k
		\mp \delta, \gamma^l \mp \delta) \\
	&= \min\{\mp l\alpha^\tau_\mp(\gamma^k), \mp
		k\alpha^\tau_\mp(\gamma^l)\} - 
		\min\{\mp l\alpha^\tau_\mp(\gamma^k \mp \delta),
		\mp k\alpha^\tau_\mp(\gamma^l \mp \delta)\}.
\end{align*}
Moreover, by the properties of the extremal winding numbers
\(\alpha^\tau_\pm\), we have (\emph{c.f.}~the proof of
Lemma~\ref{lem:unconstrained-breaking-vanishes} and
Remark~\ref{rem:linearity-of-extremal-winding})
\[
	\alpha^{\tau^k}_\pm(\gamma^k) =
		k\mathrm{wind}^\tau_{\mathbf{A}_\gamma}(0) \pm 1
	\qquad \text{and} \qquad
	\alpha^{\tau^k}_\pm(\gamma^k \pm \delta) =
		k\mathrm{wind}^\tau_{\mathbf{A}_\gamma}(0).
\]
Thus the result follows.
\end{proof}

Since \(\gamma' \ne \gamma\), and there is a unique
\(\cotangentPlanesModuli_\pm\) plane asymptotic to each orbit, we obtain the
following:
\begin{corollary}
    \label{cor:bottomLevelPlanesDistinct}
    The planes \(\{ \bottomlevel{(F_{\Eplus})}, \bottomlevel{(F_{\FEplus})},
\bottomlevel{(F_{\Eminus})}, \bottomlevel{(F_{\FEminus})} \}\) are pairwise
geometrically distinct. Specifically, the planes with common parity \(
\bottomlevel{(F_{\Eplus})}, \bottomlevel{(F_{\FEminus})} \in \cotangentPlanesModuli{}_+\)
and \( \bottomlevel{(F_{\FEplus})}, \bottomlevel{(F_{\Eminus})} \in
\cotangentPlanesModuli{}_-\) are distinct.
\end{corollary}

\subsubsection{Controlling the asymptotic orbits}
\label{subsub:asymptotics}

The aim of this section is to show that the almost complex structure on \(X\)
that undergoes neck stretching can be chosen so that the top level of the limiting
building of the sequence of \(\FEplus\) curves, \(\toplevel{(F_{\FEplus})}\), has
asymptotic Reeb orbit \(\gamma'\) equal to the image of the asymptote
\(\gamma\) of \(\toplevel{(F_{\Eplus})}\) under fibre-wise multiplication by
\(-1\). This is the property that will power our argument that the
\(\bottomlevel{J}\)-holomorphic cylinders in \(T^*S^2\) (obtained via neck
stretching in Section~\ref{sec:analysis-buildings-c1=2}) that intersect the
zero section do so along circles.

As mentioned in Section~\ref{sub:complex-structures}, fibre-wise multiplication
by \(-1\) in \(T^*S^2\) corresponds to complex conjugation under the
\(SO(3)\)-equivariant isomorphism \(T^*S^2 \cong (z_1^2+z_2^2+z_3^2=1)\). We
say these asymptotes are conjugate and  write this as \(\gamma' =
\overline{\gamma}\). We first make an explicit construction to force the top level curves
\(\toplevel{(F_{\Eplus})}\) and \(\toplevel{(F_{\FEplus})}\) to have conjugate
asymptotics. Recall the almost-K\"ahler structure
\((T^*S^2,\omega_\mathrm{can},\JTS)\), where \(\JTS\) is any of the
compatible, cylindrical almost complex structures constructed in
Section~\ref{sub:complex-structures}, and that the natural \(SO(3)\) action by
rotations is by exact almost-K\"ahler isometries. Note in particular that
Section~\ref{sub:complex-structures} allows us to produce a compatible
almost complex structure that is cylindrical outside of an arbitrarily small
neighbourhood of the zero section. We'll use this property to alter
\(\toplevel{J}\) in the neck region of \(\toplevel{X}\).

Denote by \(M\) the unit cotangent bundle \(T^*_1S^2\). Since both conjugation
and the Hamiltonian \(SO(3)\) action preserve the length function \(|p|\) on
\(T^*S^2\), they restrict to \(\R\)-equivariant maps on the symplectisation
\(\R \times M\). Moreover, the \(SO(3)\) action is by exact
symplectomorphisms, so it descends to give an action on the quotient \(M/S^1\)
of the unit cotangent bundle by the Reeb flow. Recall that \(M/S^1 \cong S^2\)
and, under this identification, the action of \(SO(3)\) is the usual one by
rotations.

We introduce some notation for convenience: as \(\toplevel{X}\) is a
manifold with the cylindrical end \((-\infty,\epsilon) \times M\), it makes
sense to define the subsets, for all \(r_0 < \epsilon\),
\begin{align*}
    X_{\le r_0 } &:= (-\infty,r_0] \times M, \text{ and,}\\
    X_{\ge r_0 } &:= \toplevel{X} \backslash (-\infty,r_0) \times M.
\end{align*}
Since \(\toplevel{(F_{\Eplus})}\) and \(\toplevel{(F_{\FEplus})}\) are
asymptotic to the distinct Reeb orbits \(\gamma, \gamma' \in M/S^1\), there
exists a number \(r_1 \ll 0\) and an open neighbourhood (which is a lift via
the composition \(X_{\le r_1} \to M \to M/S^1\) of a neighbourhood of \(\gamma
\in M/S^1\)) \(U \subset X_{\le r_1}\) of the half cylinder
\[
    C_{\Eplus} := \im\toplevel{(F_{\Eplus})} \cap X_{\le r_1}
\]
such that \(\im\toplevel{(F_{\FEplus})}\) is disjoint from
\(U\). We are now ready to state and prove the main result of this section.

\begin{lemma}[]
    \label{lem:conjugateAsymptotesACS}
    There exists a compatible almost complex structure \(J\) on
\(\toplevel{X}\) and a number \(r_2 < r_1\) that satisfy the following:
\begin{enumerate}
    \item \(J\) agrees with \(\toplevel{J}\) on the complement of \([r_2,r_1]
	\times M\), \emph{i.e.}~on \(X_{\ge r_1} \cup
	X_{\le r_2}\), and on the neighbourhood \(U\) of \(C_{\Eplus}\). In
	particular, \(J\) is cylindrical on \(X_{\le r_2}\).
    \item there exists a \(J\)-holomorphic plane asymptotic to
	\(\bar{\gamma}\) in the same relative homology class as
	\(\toplevel{(F_{\FEplus})}\).
\end{enumerate}
\end{lemma}
\begin{proof}
    The half-cylinder \(C_{\Eplus}\) is asymptotic to \(\gamma\), whilst the
\(C_{\FEplus} := \im\toplevel{(F_{\FEplus})} \cap X_{\le r_1}\) is
asymptotic to \(\gamma' \ne \gamma\). The asymptotic
convergence of \(C_{\Eplus}\) and \(C_{\FEplus}\) implies that their
projections to the manifold of Reeb orbits \(M/S^1\) yield disjoint
closed neighbourhoods \(N_\gamma\) and \(N_{\gamma'}\) of \(\gamma,\gamma' \in
M/S^1\). Choose a path \(\gamma_t \in M/S^1\) disjoint from \(N_\gamma\) such
that \(\gamma_0 = \gamma'\) and \(\gamma_1 = \bar{\gamma}\) and pick a
corresponding 1-parameter subgroup \(R_t \in SO(3)\) such that \(R_t(\gamma')
= \gamma_t\). Note that the asymptotic convergence of \(C_{\Eplus}\) and
\(C_{\FEplus}\) ensures that \(r_1 \ll 0\) can be chosen large enough to
ensure that \(R_t(N_{\gamma'})\) is disjoint from \(N_\gamma\) for all \(t \in
[0,1]\).

Now choose \(r_2 < r_1\) and a diffeomorphism \(\psi : [r_2,r_1] \to [0,1]\)
such that \(\psi' \equiv -1\) near the ends. Consider the diffeomorphism
\[
    R : X_{\le r_1} \to X_{\le r_1} : R(r,x) = \begin{cases}
	(r,R_{\psi(r)}(x)), &\text{ if }r \in [r_2,r_1],\\
	(r,R_1(x)), &\text{ if }r < r_2,
    \end{cases}
\]
and define the \emph{symplectic} half-cylinder \(C'_{\FEplus} :=
R(C_{\FEplus})\). Note that \(C'_{\FEplus}\) is
disjoint from the neighbourhood \(U\) of \(C_{\Eplus}\). We glue \(C'_{\FEplus}\) to \(\im\toplevel{(F_{\FEplus})}
\cap X_{\ge r_1}\) to obtain a symplectic plane \(P_{\FEplus}\) asymptotic to
\(\bar{\gamma}\). Linear interpolation from \(R\) to the identity map yields a
(relative) homology between \(\im\toplevel{(F_{\FEplus})}\) and
\(P_{\FEplus}\).

Since \(C'_{\FEplus}\) is disjoint from \(U\), we may choose a
compatible almost complex structure \(J\) on \([r_2,r_1] \times M\) that
makes \(C'_{\FEplus}\) holomorphic, and agrees with \(\toplevel{J}\) on \(U\).
Moreover, it can be chosen to agree with \(\toplevel{J}\) on \(\{r_1\} \times M
\cup \{r_2\} \times M\) since \(R_0 = \id\) and \(R_1\) is holomorphic with
respect to \(\toplevel{J}\). Extending \(J\) by \(\toplevel{J}\) on the
complement of \([r_2,r_1] \times M\) completes the construction.
\end{proof}

Since \((\toplevel{X},\omega^+)\) is symplectomorphic to \((X \backslash L,
\omega)\), we can push forward our new almost complex structure \(J\) to \(X
\backslash L\) and perform neck stretching to this new \(J\). What
this amounts to is stretching the neck around the contact-type hypersurface
\(T^*_{e^{r_2}}S^2 \subset X\). To make this completely precise, we need to extend \(J\)
over \(L\). However, this is easily achieved by altering the bump function
used in the construction of \(\JTS\) in Section~\ref{sub:complex-structures}.
The result is a new compatible almost complex structure on \(X\) that is
adapted to the hypersurface \(T^*_{e^{r_2}}S^2\), is equal to our
original almost complex structure on the set \(X_{\ge r_1}\), and, under
stretching the neck, converges to the almost complex structure of
Lemma~\ref{lem:conjugateAsymptotesACS} on \(\toplevel{X}\). To avoid
cluttering the notation, we shall continue to refer to the new split almost
complex structure obtained on \(X^* = \splitcurve{X}\) as \(J^*\), as we
will no longer make reference to the old one.

\begin{remark}
    Note that the new stretched almost complex structure obtained on
\(\bottomlevel{X} = T^*L\) may be different to the original one, since we have
\emph{increased} the portion of the neck \(\R \times T^*_1S^2 \cong T^*L
\backslash L\) on which \(\bottomlevel{J}\) is cylindrical. This amounts to
pushing forward the old \(\bottomlevel{J}\) by the Liouville flow for some
negative time.

Ultimately, though, this is of no importance, since the results of
Section~\ref{sub:cotangentPlanes} are valid for whatever cylindrical almost
complex structure we choose on \(T^*S^2\).
\end{remark}

Since we've changed our neck stretching sequence \(J_k\), morally we need to
redo the limit analysis of the previous section. Indeed, recall the open set
\(V = \toplevel{X} \backslash \left(((-\infty,\epsilon]\times T_1^*L) \cup
D'\right)\) from Section~\ref{sec:stretchingAboutSphere} in which we perturbed
to ensure genericity. The almost complex structure \(\toplevel{J}\)
constructed here satisfies a form of symmetry, which perturbing would destroy.
Therefore, \emph{a priori}, \(\toplevel{J}\)-holomorphic curves don't satisfy
the Fredholm regularity result of Lemma~\ref{lem:neg-puncs-implies-regular}.

Note that we \emph{can} perturb \(\toplevel{J}\) on the complement of the
images of \(\toplevel{(F_{\Eplus})}\) and \(P_{\FEplus}\) in \(V\). Moreover,
since a punctured \(\toplevel{J}\)-holomorphic curve in \(\toplevel{X}\)
either passes through this complement, or has identical image to one of
\(\toplevel{(F_{\Eplus})}\) or \(P_{\FEplus}\), we can apply
Lemma~\ref{lem:neg-puncs-implies-regular} to every punctured curve in
\(\toplevel{X}\) except \(\toplevel{(F_{\Eplus})}\) and \(P_{\FEplus}\). Thus,
if we can show that \(\toplevel{(F_{\Eplus})}\) and \(P_{\FEplus}\) are
Fredholm regular, then we can apply Lemma~\ref{lem:neg-puncs-implies-regular}
exactly as in Section~\ref{sub:analysis-buildings-c1=1}.

As a result, a convergent sequence of curves \(e_k \in
\mathcal{M}_{0,0}(\Eplus;J_k)\) converges to \(\toplevel{(F_{\Eplus})}\) in
the top level, and similarly, for curves in
\(\mathcal{M}_{0,0}(\FEplus;J_k)\), we obtain that they converge to the
constructed plane \(P_{\FEplus}\).

\begin{lemma}[]
    The \(\toplevel{J}\)-holomorphic planes \(\toplevel{(F_{\Eplus})}\) and
\(P_{\FEplus}\) are Fredholm regular.
\end{lemma}
\begin{proof}
    The case of \(\toplevel{(F_{\Eplus})}\) is trivial since \(\toplevel{J}\) was
unchanged in a neighbourhood of this curve. For \(P_{\FEplus}\) we apply
automatic transversality \cite[Theorem~1]{wendl2010autoTrans}. Since
\(P_{\FEplus}\) is embedded, the automatic transversality condition is:
\[
    \ind(P_{\FEplus}) > c_N(P_{\FEplus}).
\]
We compute that \(\ind(P_{\FEplus}) = 0\), which follows from
Equation~\eqref{eq:indexFormula}, Lemma~\ref{lem:relChern=absChern}, and that
\(P_{\FEplus}\) has a single simply covered asymptote. On the other hand,
\[
    c_N(P_{\FEplus}) = -1
\]
follows as in Lemma~\ref{lem:cN=-1-for-c1=1-curves}.
\end{proof}

\begin{proposition}[]
    \label{prop:toplevelPlanesHaveConjugateAsymptotes}
Let \(J_k\) be a neck stretching sequence converging to \(J^*\), and
\(e_k \in \mathcal{M}_{0,0}(\Eplus;J_k)\) a corresponding sequence converging
to a \(J^*\)-holomorphic building \(F = \splitcurve{F}\). Then the
top level \(\toplevel{F}\) consists only of the plane \(\toplevel{(F_{\Eplus})}\)
obtained in Proposition~\ref{prop:top-level-2-planes}, there are no
symplectisation levels, and the bottom level consists of exactly one plane in
\(\cotangentPlanesModuli_+\).

Similarly, a convergent sequence \(g_k \in \mathcal{M}_{0,0}(\FEplus;J_k)\)
converges to a building \(G = \splitcurve{G}\) whose top level
\(\toplevel{G}\) consists only on the plane \(P_{\FEplus}\) constructed in
Lemma~\ref{lem:conjugateAsymptotesACS}. As above, the bottom level then
consists of exactly one plane in \(\cotangentPlanesModuli_-\). In particular,
\(\toplevel{F}\) and \(\toplevel{G}\) have conjugate asymptotes.
\end{proposition}
\begin{proof}
Since \(\toplevel{J}\) is unchanged in a neighbourhood of
\(\toplevel{(F_{\Eplus})}\), there is nothing to check for the claimed
convergence \(e_k \to F\). We prove that
\(\toplevel{G} = P_{\FEplus}\). Note that \(P_{\FEplus}\) can be completed
into a \(J^*\)-holomorphic building \(G'\) by adding in the unique 
\(\cotangentPlanesModuli_-\) plane in \(T^*S^2\) with asymptotic orbit
\(\bar{\gamma}\). Then, both \(G\) and \(G'\) are \(J^*\)-holomorphic
buildings of the same homology class: \([\bar{G}] = \FEplus = [\bar{G'}]\).
Thus, they must intersect somewhere, since \((\FEplus)^2 = -1\). If
\(\toplevel{G}\) is geometrically distinct from \(\toplevel{(G')} =
P_{\FEplus}\), then positivity of intersections would force there to be a
negative intersection in the bottom level, contradicting
Corollary~\ref{cor:bottomLevelNonNegInts}. Hence, \(\toplevel{G} =
P_{\FEplus}\) and the result follows.
\end{proof}

We have constructed two \(J^*\)-holomorphic buildings \(F\) and \(G\)
in the homology classes \(\Eplus\) and \(\FEplus\), whose asymptotic orbits
are conjugate. To free up the notation, we shall continue to denote these
buildings by \(F_{\Eplus}\) and \(F_{\FEplus}\) respectively. Similarly we
obtain the \(J^*\)-holomorphic buildings \(F_{\Eminus}\) and \(F_{\FEminus}\),
which satisfy
\[
    \toplevel{(F_{\Eminus})} = \toplevel{(F_{\Eplus})} \text{ and } 
    \toplevel{(F_{\FEminus})} = \toplevel{(F_{\FEplus})},
\]
by Proposition~\ref{prop:top-level-2-planes}, and thus they have the same pair
of conjugate asymptotes \(\{\gamma,\bar{\gamma}\}\).

\begin{corollary}[]
    \label{cor:planesIntersectOnL}
    The unique point of intersection of the planes
\(\bottomlevel{(\building{\Eplusminus})} \in \cotangentPlanesModuli_\pm\) and
\(\bottomlevel{(\building{\FEplusminus})} \in \cotangentPlanesModuli_\mp\)
lies on the zero section.
\end{corollary}
\begin{proof}
    Conjugation sends \(\cotangentPlanesModuli_+\) planes to
\(\cotangentPlanesModuli_-\) and vice versa. Therefore, since there is a
unique \(\cotangentPlanesModuli_\pm\) plane asymptotic to each simple Reeb
orbit and the asymptotics of \(\bottomlevel{(\building{\Eplusminus})}\) and
\(\bottomlevel{(\building{\FEplusminus})}\) are conjugate, we have that
\[
    \overline{\bottomlevel{(\building{\Eplusminus})}} =
    \bottomlevel{(\building{\FEplusminus})}.
\]
Moreover, the zero section is the fixed locus of conjugation, so, since each
plane intersects the zero section, the claim follows.
\end{proof}

\begin{remark}
    This result partially justifies the claim made in
Remark~\ref{rem:bottomPlanesAreNodalFibres} that the bottom levels of the
buildings \(\building{\Eplusminus} \cup \building{\FEplusminus}\) form the
nodal fibres of a Lefschetz fibration on \(T^*L\) such that \(L\) is a
matching cycle.
\end{remark}

\subsubsection{Analysis of buildings: limits of \(c_1 = 2\) curves}
\label{sec:analysis-buildings-c1=2}

In this section we prove (Proposition~\ref{prop:sft-limit-fibre-curve}) that
limits of curves in the class of a fibre \(H-S \in H_2(X)\) have bottom
level (if non-empty) either a smooth cylinder, or one of the pairs of nodal
planes \(\bottomlevel{(F_{\Eplusminus})} + \bottomlevel{(F_{\FEplusminus})}\).
First, we need a basic result on unconstrained intersections of trivial
cylinders in symplectisations.\footnote{Compare with the non-degenerate case
    where
\[
    u_\gamma \ast u_\gamma =
    \begin{cases} 0, &\text{ if } \gamma \text{ is odd}\\
		  1, &\text{ if } \gamma \text{ is even}. \end{cases}
\]}
\begin{lemma}[]
    \label{lem:trivialCylinderIntersections}
    Let \(\gamma\) be a simply covered Morse-Bott degenerate
Reeb orbit living in a positive dimensional orbifold \(N\) of unparametrised
orbits of a contact manifold \(M\). Then the unconstrained self intersection
of the corresponding trivial cylinder \(u_\gamma : \R \times S^1 \to \R \times
M\) is zero:
\[
    i_U(u_\gamma,u_{\gamma}) = 0.
\]
\end{lemma}
\begin{proof}
    Since \(\dim(N) > 0\), we can homotope \(u_\gamma\) to
\(u_{\tilde{\gamma}}\) for some \(\tilde{\gamma} \ne \gamma\) in the same
family \(N\). Since \(i_U\) is homotopy invariant, we obtain the result, as
\(u_\gamma\) and \(u_{\tilde{\gamma}}\) are disjoint and asymptotic to
distinct orbits.
\end{proof}

\begin{remark}
    \label{rem:triv-cyl-self-int}
    In our situation, all simple Reeb orbits live in the same manifold \(N
\cong S^2\) of unparametrised orbits, and so the above result combined with
Proposition~\ref{prop:intersection-additive-covers}(2) yields, for any
integers \(k,l > 0\) and simple orbits \(\gamma\) and \(\gamma'\),
\[
	i_U(u_{\gamma^k},u_{(\gamma')^l}) = 0.
\]
From this and positivity of intersections (Theorem~\ref{thm:punc-pos-ints}),
we deduce that any curve in a symplectisation level intersects a trivial
cylinder non-negatively.
\end{remark}

Let \(F = \splitcurve{F}\) denote a \(J^*\)-holomorphic building arising as the
limit of a sequence of fibre curves. If the bottom level \(\bottomlevel{F}\)
is non-empty, the following lemma shows that its top-level \(\toplevel{F}\)
has to be the disjoint union of the limits that arose in
Proposition~\ref{prop:toplevelPlanesHaveConjugateAsymptotes}.
\begin{lemma}[]
    \label{lem:top-level-fibre-coincides-with-planes}
    Let \(F = \splitcurve{F}\) be the limit of a sequence of curves in the
class \([\bar{F}]=H-\fibrationSectionL \in H_2(X)\) of a fibre. Suppose that
\(\bottomlevel{F} \ne \emptyset\). Then \(\toplevel{F}\) is the disjoint union
of the planes \(\toplevel{(\building{\Eplus})}\) and
\(\toplevel{(\building{\FEplus})}\), where \(\building{\Eplus}\) and
\(\building{\FEplus}\) are the holomorphic buildings that arose in
section~\ref{subsub:asymptotics}. That is, \(\toplevel{F} =
\toplevel{(\building{\Eplus})} + \toplevel{(\building{\FEplus})} =
\toplevel{(\building{\Eplus})} \sqcup \toplevel{(\building{\FEplus})}\) is the
disjoint union of two planes with conjugate asymptotics
\(\{\gamma,\bar{\gamma}\}\).
\end{lemma}
\begin{proof}
    Since \(\bottomlevel{F} \ne \emptyset\), we can use
Proposition~\ref{prop:cotangent-curves-intersect-planes} to analyse how
\(\bottomlevel{F}\) intersects the planes in \(\cotangentPlanesModuli{}_\pm\).
Note that \(\bottomlevel{F}\) cannot possibly be composed exclusively of
(covers of) planes in only one of the families \(\cotangentPlanesModuli{}_+\)
or \(\cotangentPlanesModuli{}_-\), as otherwise the intersection of
\(\bottomlevel{F}\) with \(L\) would be non-zero, contradicting \([\bar{F}]
\cdot L = 0\). Thus, \(\bottomlevel{F}\) intersects both
\(\bottomlevel{(\building{\Eplus})} \in \cotangentPlanesModuli_+\) and
\(\bottomlevel{(\building{\FEplus})} \in \cotangentPlanesModuli_-\)
positively. That is
\[
    i_U(\bottomlevel{F}, \bottomlevel{(\building{\Eplus})}) > 0
    \qquad \text{and} \qquad
    i_U(\bottomlevel{F}, \bottomlevel{(\building{\FEplus})}) > 0.
\]

Now let \(\digamma = (\toplevel{\digamma}, \bottomlevel{\digamma})\)
denote either of the buildings \(\building{\Eplus}\) or
\(\building{\FEplus}\). \emph{A priori}, the building \(F\) may have
non-trivial symplectisation levels, so to make sense of the intersection
number of the buildings \(i(F,\digamma)\) we may need to extend \(\digamma\)
by \emph{trivial} symplectisation levels.\footnote{See Appendix~C.5 of
\cite{wendl2020Contact3folds} for a full discussion of why this operation is
well defined.} With this understood, Theorem~\ref{thm:intersection-cts},
Proposition~\ref{prop:intersection-additive-covers},
\(i_U(\bottomlevel{F},\bottomlevel{\digamma}) > 0\) and \([\bar{F}] \cdot
[\bar{\digamma}] = 0\) imply that
\[
    i_U(\toplevel{F},\toplevel{\digamma}) + \sum_{i=1}^N
i_U(F^{(i)},\digamma^{(i)}) < 0. \] Furthermore,
Remark~\ref{rem:triv-cyl-self-int} implies that, for each \(1 \le i \le N\),
\(i_U(F^{(i)},\digamma^{(i)}) \ge 0\), and thus
\(i_U(\toplevel{F},\toplevel{\digamma}) < 0\). Therefore, by positivity of intersections,
at least one component of \(\toplevel{F}\) covers \(\toplevel{\digamma}\).

Summarising, we have proved that some components of \(\toplevel{F}\) cover both
\(\toplevel{(\building{\Eplus})}\) \emph{and}
\(\toplevel{(\building{\FEplus})}\), and so there exist positive integers
\(m,n > 0\) such that
\[
    \toplevel{F} = m\toplevel{(\building{\Eplus})} +
    n\toplevel{(\building{\FEplus})} + \text{(other terms)}.
\]
However, this already exhausts the total \(\omega^+\)-area. Indeed, by
Corollary~2.11 of \cite{cieliebakMohnkeSFTCompactness}, we have that
\begin{align*}
    \omega([\bar{F}]) &= \int_{\toplevel{\Sigma}}(\toplevel{F})^*\omega^+ \\
	&\ge m\int_{\C} (\toplevel{(\building{\Eplus})})^*\omega^+ +
	n\int_{\C} (\toplevel{(\building{\FEplus})})^*\omega^+ \\
	&= m\omega([\Eplus]) + n\omega([\bar{F}] - [\Eplus]) \\
	&= n\omega([\bar{F}]) + (m-n)\omega([\Eplus]) \\
	&= m\omega([\bar{F}]) + (n-m)(\omega([\bar{F}]) - \omega([\Eplus])).
\end{align*}
Recall that Lemma~\ref{lem:compactify} states that \(\omega([\bar{F}]) \ge
\omega([\Eplus]) = l\), and so, in either of the cases \(n \ge m\), or \(m \ge
n\), we obtain
\[
    \omega([\bar{F}]) \ge m\omega([\Eplus]) + n\omega([\bar{F}] - [\Eplus]) \ge
    \omega([\bar{F}]),
\]
and thus \(n=m=1\). Moreover, this implies that there can be no other terms
in the expression for \(\toplevel{F}\). That is, \(\toplevel{F} =
\toplevel{(\building{\Eplus})} + \toplevel{(\building{\FEplus})}\), and so the
result is proved.
\end{proof}

\begin{lemma}[]
    \label{lem:fibre-no-intermediate-levels}
    Let \(F = \splitcurve{F}\) be the limit of a sequence of curves in the
class of a fibre \([\bar{F}]=H-S \in H_2(X)\) such that \(\bottomlevel{F} \ne
\emptyset\). Then \(N = 0\), that is, \(F = (\toplevel{F}, \bottomlevel{F})\)
consists only of a top and bottom level.
\end{lemma}
\begin{proof}
    Lemma~\ref{lem:top-level-fibre-coincides-with-planes} shows that
\(\toplevel{F} = \toplevel{(\building{\Eplus})} +
\toplevel{(\building{\FEplus})}\) and so any non-trivial component \(f\) of
\(F^{(N)}\) (the level adjacent to the top level \(\toplevel{F}\)) must have
exactly two positive punctures. Indeed, if it only had one, then the argument
of Lemma~\ref{lem:connectedTopImpliesNoIntermediateLevels} would imply that it
is a trivial cylinder, and thus \(F^{(N)}\) would be a union of trivial
cylinders, which is ruled out by the stability condition. Therefore, \(f\) has
exactly two positive punctures: \(\gamma\) and \(\bar{\gamma}\) of
\(\toplevel{(\building{\Eplus})}\) and \(\toplevel{(\building{\FEplus})}\).

    Since \(\toplevel{F} = \toplevel{(\building{\Eplus})} +
\toplevel{(\building{\FEplus})}\), we have that \(i_U(\toplevel{F},
\toplevel{(\building{\Eplus})}) = -1\), and so, to balance the equation
\(i(F,\Eplus) = [\bar{F}] \cdot [\Eplus] = 0\), there must be a positive
intersection between \(F\) and \(\building{\Eplus}\) in a different level. The
previous paragraph implies that this positive intersection is eaten up by the
non-trivial level \(F^{(N)}\), which then forces \(\bottomlevel{F}\) to take
on an illegal form. Indeed, extend \(\building{\Eplus}\) by trivial
intermediate levels, then, as \((\building{\Eplus})^{(N)}\) and \(f\) are
geometrically distinct, their intersection number is bounded below by the
Morse-Bott contributions from their asymptotic orbits, which are computed in
Lemma~\ref{lem:mbContributions}:
\[
    i_U(f,(\building{\Eplus})^{(N)}) \ge i_\mathrm{MB}^+(\gamma,\gamma) = 1.
\]
Combining this with Remark~\ref{rem:triv-cyl-self-int}, which says that
any further intersections appearing in intermediate levels are non-negative,
we obtain that
\begin{align*}
    0 = i(F,\Eplus) &\ge i_U(\toplevel{F}, \toplevel{(\building{\Eplus})}) +
    i_U(f,(\building{\Eplus})^{(N)}) +
    i_U(\bottomlevel{F}, \bottomlevel{(\building{\Eplus})}) \\
      &\ge i_U(\toplevel{F}, \toplevel{(\building{\Eplus})}) + 1 +
      i_U(\bottomlevel{F}, \bottomlevel{(\building{\Eplus})}) \\
      &= i_U(\bottomlevel{F}, \bottomlevel{(\building{\Eplus})}),
\end{align*}
and therefore \(i_U(\bottomlevel{F}, \bottomlevel{(\building{\Eplus})}) = 0\). Since
\(\bottomlevel{F} \ne \emptyset\),
Lemma~\ref{lem:cotangent-planes-intersection} implies that \(\bottomlevel{F}\)
must consist of covers of \(J\)-holomorphic planes of the same parity as
\(\bottomlevel{(\building{\Eplus})}\), which contradicts \([\bar{F}] \cdot L =
0\). Hence, there are no non-trivial symplectisation levels, and so the result
is proved.
\end{proof}

The following is the main result of the analysis of \(c_1=2\) buildings.
\begin{proposition}[]
    \label{prop:sft-limit-fibre-curve}
    Let \(F\) be a limiting building of a sequence of curves in the class
\([\bar{F}]=H-S \in H_2(X)\) of a fibre such that \(\bottomlevel{F} \ne
\emptyset\). Then \(F = ( \toplevel{F}, \bottomlevel{F} )\) where
\(\toplevel{F} = \toplevel{(\building{\Eplus})} +
\toplevel{(\building{\FEplus})}\) and \(\bottomlevel{F}\) is either
\begin{enumerate}
    \item one of the pairs of nodal curves \(\bottomlevel{(\building{\Eplus})} +
	\bottomlevel{(\building{\FEplus})}\) or \(\bottomlevel{(\building{\Eminus})} +
	\bottomlevel{(\building{\FEminus})}\), with \(\building{\Eplus},
	\building{\FEplus}, \building{\Eminus}, \building{\FEminus}\) as in
	section~\ref{subsub:asymptotics}; or,
    \item a smooth cylinder with the same asymptotes \(\{\gamma,
	\bar{\gamma}\}\) as \(\toplevel{(\building{\Eplus})} +
	\toplevel{(\building{\FEplus})}\).
\end{enumerate}
\end{proposition}
\begin{proof}
    The previous results show that \(F\) is a building with only a top and
bottom level, and that \(\toplevel{F} = \toplevel{(\building{\Eplus})} +
\toplevel{(\building{\FEplus})}\). This tells us that the positive asymptotics
of \(\bottomlevel{F}\) are exactly \(\gamma\) and \(\bar{\gamma}\) of
\(\toplevel{(\building{\Eplus})}\) and \(\toplevel{(\building{\FEplus})}\).
Thus, \(\bottomlevel{F}\) is a connected genus 0 holomorphic curve with
exactly two positive punctures. Therefore, either it is a smooth cylinder, or
it has nodes forming a chain of closed spheres connecting two
planes. However, there are no closed holomorphic
spheres in \(T^*L\) since this is an exact symplectic manifold, and so there
can be at most one node.

Observe that a nodal pair \(P_1 + P_2\) of \(J\)-holomorphic planes
satisfying \((P_1 + P_2) \cdot L = 0\) and having asymptotic orbits
\(\{\gamma,\bar{\gamma}\}\) must be exactly one of the pairs
\(\bottomlevel{(\building{\Eplus})} + \bottomlevel{(\building{\FEplus})}\) or
\(\bottomlevel{(\building{\Eminus})} + \bottomlevel{(\building{\FEminus})}\).
Therefore, either we are in case (1), or \(\bottomlevel{F}\) passes through a
point in \(T^*L\) not contained in the images of the curves
\(\bottomlevel{(\building{\Eplus})}, \bottomlevel{(\building{\FEplus})},
    \bottomlevel{(\building{\Eminus})},\) or
\(\bottomlevel{(\building{\FEminus})}\), and is thus a smooth cylinder.
\end{proof}

\subsection{Constructing the foliation}
\label{sec:foliationConstruction}
In this section we use the limit analysis of
section~\ref{sec:buildingAnalysis} to construct a \(\JTS\)-holomorphic
foliation of \(T^*S^2\) by cylinders. The process also picks out a particular
neck stretching sequence of almost complex structures \(J_k\) such that all
the curves of interest converge. First, suppose that we have a neck stretching
sequence \(J_k\), and a countable collection of sequences \((f_k^m)\) of
\(J_k\)-holomorphic curves with uniformly bounded energy. One then applies a
diagonal argument, a countable generalisation of that discussed in
Remark~\ref{rem:repeatedSFTCompactness}, using the SFT compactness theorem to
extract a subsequence \(J_k\) such that, for all \(m\), the sequences
\((f_k^m)\) converge to \(J^*\)-holomorphic buildings as \(k \to
\infty\).

We apply this process to the following sequences. Recall the \(-1\)-classes
\(\{\mathcal{E}_j \mid 1 \le j \le d\}\) defined in
Lemma~\ref{lem:compactify},
where classes of the form \(\mathcal{E}_i - \mathcal{E}_j\) support the
Lagrangian spheres in \(B_{d,p,q}\). The critical points
of the Lefschetz fibrations \(\pi_{J_k} : \XJ{J_k} \to \C\) correspond exactly to
the unique intersection points of \(\mathcal{E}_j \cdot (F - \mathcal{E}_j)\).
This allows us to partition the set of critical points into so-called
\emph{relevant} and \emph{irrelevant} sets determined by whether the
intersection \(\mathcal{E}_j \cdot L\) is non-zero or not. That is, the
relevant critical points correspond to the classes \(\{\Eplus,\Eminus\}\) as
defined in \eqref{eq:EplusMinus}, and the irrelevant ones correspond to the
remaining \(\mathcal{E}_j\) classes. Recall also the class \(E_n\), which
represents (the underlying simple curve of) a component of the exotic curve
\(\infinityCurve{J}\) (see Corollary~\ref{cor:c00FmoduliDescription}). Even
though curves in this class do not correspond to Lefschetz critical points, we
shall also call them irrelevant. With this in mind, for each \(1 \le j \le
d\), take the unique sequences of \(J_k\)-holomorphic \(-1\)-curves
\begin{align*}
    e_k^j &\in \mathcal{M}_{0,0}(\mathcal{E}_j;J_k),\\
    g_k^j &\in \mathcal{M}_{0,0}(F-\mathcal{E}_j;J_k),\\
    \varepsilon_k &\in \mathcal{M}_{0,0}(E_n;J_k),
\end{align*}
along with a countable number of sequences of fibre curves
\[
    f_k^m \in \mathcal{M}_{0,0}(F;J_k)
\]
determined by point constraints for some fixed dense set
\(\{x^m \in \weinsteinNbhd\}\) in the Weinstein neighbourhood
\(\weinsteinNbhd\) of \(L\). We extract a subsequence \(J_k\) for
which all these sequences converge to the \(J^*\)-holomorphic buildings
\(F^m\).

Note that Proposition~\ref{prop:c1=1-sft-limit} implies that, for sufficiently
large \(k\), the images of the \emph{irrelevant} curves stay bounded away from
\(L\). Moreover, a further diagonal argument ensures that the convergence of
all curves is monotonic, which will be useful in
Section~\ref{sec:convergenceOfMatchingPaths}.

To upgrade the dense set \(\{\bottomlevel{(F^m)}\}\) of
\(\JTS\)-holomorphic curves in \(T^*S^2\) to a foliation we apply a
bubbling argument, which is inspired by
\cite[\S{}6]{HWZ03threeSphere}.\footnote{The argument of
\cite[\S{}6]{HWZ03threeSphere} is much more complicated than what is required
for our situation.} The idea is choose \(x \in T^*S^2\) and to take a
subsequence of \((x^m)\) converging to \(x\) and analyse the corresponding
limits of the curves \(\bottomlevel{(F^m)}\) under SFT compactness. However,
this time the SFT compactness theorem is that relating to manifolds with
cylindrical ends \cite[Theorem~10.2]{behwzSFTCompactness}. To that
end, denote the moduli space\footnote{More precisely, this is the moduli space of
    genus 0 curves with exactly two positive asymptotes that represent the
    homology class
    \[
	(0,1,1) \in H_2(T^*S^2,\gamma\sqcup\bar{\gamma}) \cong
	H_2(T^*S^2) \oplus H_1(\gamma\sqcup\bar{\gamma}) \cong \langle 0_{S^2},
	\gamma,\bar{\gamma}\rangle \cong \Z^3.
    \]
} of \(\bottomlevel{J}\)-holomorphic cylinders with
fixed positive asymptotes \(\{\gamma,\bar{\gamma}\}\) by
\(\cylindersModuli\). We first compute the
self-intersection of a cylinder in \(\cylindersModuli\):

\begin{lemma}[]
    \label{lem:cylindersZeroIntersection}
    The \(\JTS\)-holomorphic cylinders in \(\cylindersModuli\) have
zero \emph{constrained} self-intersection. In other words, for a building in
the class \([\bar{F}]=H-S\) arising from
Proposition~\ref{prop:sft-limit-fibre-curve} we have
\[
    i_U(\bottomlevel{F},\bottomlevel{F}) = 2 \qquad\text{and}\qquad
    i_C(\bottomlevel{F},\bottomlevel{F}) = 0.
\]
\end{lemma}
\begin{proof}
    Since \(\toplevel{F} = \toplevel{(\building{\Eplus})} +
\toplevel{(\building{\FEplus})}\) and
\[
    i_U(\toplevel{(\building{\Eplus})}, \toplevel{(\building{\Eplus})}) = -1 =
    i_U(\toplevel{(\building{\FEplus})}, \toplevel{(\building{\FEplus})}),
\]
we have that
\begin{align*}
    0 &= [\bar{F}]^2 \\
      &= i_U(\toplevel{(\building{\Eplus})}, \toplevel{(\building{\Eplus})}) +
    i_U(\toplevel{(\building{\FEplus})}, \toplevel{(\building{\FEplus})}) +
    i_U(\bottomlevel{F},\bottomlevel{F}) \\
      &= i_U(\bottomlevel{F},\bottomlevel{F}) - 2,
\end{align*}
and so, \(i_U(\bottomlevel{F},\bottomlevel{F}) = 2\). Then, since
\[
    i_U(\bottomlevel{F},\bottomlevel{F}) =
    i_C(\bottomlevel{F},\bottomlevel{F}) + i_\mathrm{MB}^+(\gamma,\gamma) +
    i_\mathrm{MB}^+(\bar{\gamma},\bar{\gamma}) =
    i_C(\bottomlevel{F},\bottomlevel{F}) + 2,
\]
we obtain the result.
\end{proof}

Applying the SFT compactness theorem to sequences of curves in
\(\cylindersModuli\) yields holomorphic buildings of height \(k_-|1|k_+\), as
in Section~8 of \cite{behwzSFTCompactness}. Since \(T^*S^2\) is a manifold
with no negative cylindrical ends, we have that \(k_-=0\), so the resultant
buildings have a \emph{main level} \(F^{(0)} : \Sigma^{(0)} \to T^*S^2\) and
\(k_+\) \emph{upper levels} \(F^{(\nu)} : \Sigma^{(\nu)} \to \R \times
T_1^*S^2\).
\begin{lemma}[]
    A holomorphic building with non-empty main level in the SFT
compactification of \(\cylindersModuli\) has no non-trivial upper levels, and
so, is given by a curve in \(T^*S^2\). Moreover, this curve must be exactly
one of those in Proposition~\ref{prop:sft-limit-fibre-curve}.
\end{lemma}
\begin{proof}
    Denote the building by \(F = (F^{(0)},\ldots,F^{(k_+)})\). Its upper-most
level \(F^{(k_+)} : \Sigma^{(N)} \to \R \times T_1^*S^2\) is a
\(J^{(k_+)}\)-holomorphic curve in a symplectisation with positive asymptotes
given by \(\{\gamma,\bar{\gamma}\}\). If \(F^{(k_+)}\) is not the union
of trivial cylinders \(u_\gamma\sqcup u_{\bar{\gamma}}\), then it must
intersect them positively:
\[
    i_U(F^{(k_+)},u_\gamma\sqcup u_{\bar{\gamma}}) \ge
    i_\mathrm{MB}^+(\gamma,\gamma) + i_\mathrm{MB}^+(\bar{\gamma},\bar{\gamma})
    =2.
\]
Let \(P_\pm^\gamma \in \cotangentPlanesModuli_\pm\) denote the unique
\(\JTS\)-holomorphic plane asymptotic to \(\gamma\). Since \(F\) is the limit
of cylinders in \(\cylindersModuli\), we have that (extending \(P_\pm^\gamma +
P_\mp^{\bar{\gamma}}\) by trivial upper levels)
\[
    2 = i(F,P_\pm^\gamma + P_\mp^{\bar{\gamma}}) \ge i_U(F^{(0)},P_\pm^\gamma
    + P_\mp^{\bar{\gamma}}) + i_U(F^{(k_+)},u_\gamma\sqcup u_{\bar{\gamma}}),
\]
which implies that
\[
    i_U(F^{(0)},P_\pm^\gamma + P_\mp^{\bar{\gamma}}) \le 0.
\]
In light of Corollary~\ref{cor:bottomLevelNonNegInts} and
Proposition~\ref{prop:cotangent-curves-intersect-planes}, this yields
\[
    i_U(F^{(0)},P_\pm^\gamma) = 0 = i_U(F^{(0)},P_\pm^{\bar{\gamma}}),
\]
which is only possible if \(F^{(0)}\) is simultaneously a cover of all four
planes \(P_\pm^\gamma\), \(P_\pm^{\bar{\gamma}}\). However, since these planes
are distinct, this is impossible.

We have shown that \(F\) has no non-trivial upper levels, and so \(F^{(0)}\)
is a curve in \(T^*S^2\) with exactly two asymptotes
\(\{\gamma,\bar{\gamma}\}\). Therefore, we can apply 
Proposition~\ref{prop:sft-limit-fibre-curve} to complete the proof.
\end{proof}

We write \(\cylindersModuliBar\) to denote the subset of the SFT
compactification of \(\cylindersModuli\) consisting of buildings with non-empty
main level. This is a topological surface homeomorphic to the complex plane
\(\C\). Moreover, just as in Section~\ref{sec:compactLefFib},
\(\cylindersModuliBar\) can be equipped with a smooth structure (indeed the
unique one) that makes the natural map
\[
    \TSlefFib : T^*S^2 \to \cylindersModuliBar : x \mapsto \text{the curve
	passing through } x
\]
into a Lefschetz fibration with exactly two critical points corresponding to
the intersection points of the two pairs of planes \(P_\pm^\gamma +
P_\mp^{\bar{\gamma}}\).

The next goal is to prove that \(L\) is fibred by circles in the foliation
\(\cylindersModuliBar\), and thus the image of \(L\) under \(\TSlefFib\) is a
smooth path.

\begin{proposition}[]
    \label{prop:cylindersIntersectAlongCircles}
    Let \(u : (\C^\times,j) \to T^*S^2\) be a smooth, properly
embedded\footnote{All of the curves in \(\cylindersModuli\) are properly
embedded. Properness follows from their asymptotic behaviour, and embeddedness
follows from the adjunction formula \cite[\S{}4.1]{wendl2010autoTrans}.}
\(\JTS\)-holomorphic cylinder with conjugate asymptotic Reeb orbits
\(\{\gamma,\bar{\gamma}\}\), as in scenario (2) of
Proposition~\ref{prop:sft-limit-fibre-curve}, that intersects the zero section
\(L \subset T^*L\). Then this intersection is along a circle contained in
\(L\).
\end{proposition}
\begin{proof}
    The construction of the almost complex structure \(\JTS\) ensures that it
is anti-invariant under the action of conjugation on \(T^*L\). That is,
\[
    \overline{\JTS} = -\JTS.
\]
Composing with conjugation gives a
\((-j,\JTS)\)-holomorphic cylinder \(\bar{u}\). We aim to show that
\(i_C(u,\bar{u}) = 0\). To this end, since \(\cylindersModuliBar \cong \C\),
we can choose a homotopy from \(u\) to one of the nodal
cylinders \(P_+^\gamma + P_-^{\bar{\gamma}}\), where \(P_+^\gamma =
\bottomlevel{(\building{\Eplus})}\) and \(P_-^{\bar{\gamma}} =
\bottomlevel{(\building{\Eplus})}\) are the \(\cotangentPlanesModuli_\pm\)
planes asymptotic to \(\gamma\) and \(\bar{\gamma}\) respectively. Then, as
conjugation sends this pair of planes to itself, we can use the homotopy
invariance, additivity, and symmetry of \(i_C\) to compute:
\begin{align*}
    i_C(u,\bar{u})
    &= i_C(P_+^{\gamma} + P_-^{\bar{\gamma}},
	\overline{P_+^{\gamma} + P_-^{\bar{\gamma}}}) \\
    &= i_C(P_+^{\gamma} + P_-^{\bar{\gamma}},
	P_+^{\gamma} + P_-^{\bar{\gamma}}) \\
    &= i_C(P_+^{\gamma}, P_+^{\gamma})
	+ 2i_C(P_+^{\gamma}, P_-^{\bar{\gamma}})
	+ i_C(P_-^{\bar{\gamma}}, P_-^{\bar{\gamma}}) \\
    &= -1 + 2 -1 = 0.
\end{align*}

However, since the zero section \(L\) is the fixed locus of conjugation, there
is necessarily an intersection between \(u\) and \(\bar{u}\). In view of the
above, we deduce that \(u\) and \(\bar{u}\) have the same image --- they are
geometrically indistinct. This implies that conjugation restricts to a
\(j\)-anti-holomorphic involution of \((\C^\times,j)\) with
non-empty fixed locus. Moreover, it's compact since \(u\) is proper.
The fixed locus of such an involution is diffeomorphic to a circle, as is
proved in the next lemma. As \(u\) is an embedding, this completes the
proof.
\end{proof}

The following result is surely well known, and we include the proof only for
completeness.
\begin{lemma}[]
    \label{lem:fixedLocusOfAntiHoloInv}
    The fixed locus of an anti-holomorphic involution \(\iota\) of
\((\C^\times,j)\) is either empty, or non-empty and diffeomorphic to \(S^1\),
or two copies of \(\R\).
\end{lemma}
\begin{proof}
    The action of the diffeomorphism group of \(\C^\times\) is transitive on
complex structures, so we can assume that \(j = i\) is the standard complex
structure. Identify \((\C^\times,i)\) with \((\CP^1 \backslash
\{0,\infty\},i)\), so that we have an anti-holomorphic involution \(\iota\) of
\(\CP^1\) that preserves the set \(\{0,\infty\}\). Then \(z \mapsto
\iota(\bar{z})\) is holomorphic, and preserves \(\{0,\infty\}\), so it must be
of the form
\[
    \iota(\bar{z}) = \begin{cases}
        \alpha z, &\text{ if } \iota(0) = 0,\\
	\alpha/z, &\text{ if } \iota(\infty) = \infty,
    \end{cases}
\]
for some \(\alpha \in \C^\times\). In the first case \(\iota(z) =
\alpha\bar{z}\), if there is a fixed point \(z_0 \in \C^\times\), then we have
that \(\alpha = z_0/\bar{z}_0\), and
\[
    \mathrm{fix}(\iota|_{\C^\times}) = \{rz_0 \mid r \in \R\backslash\{0\}\}
    \cong \R\sqcup\R.
\]
On the other hand, if \(\iota(z) = \alpha/\bar{z}\) has a non-zero fixed point
\(z \in \C^\times\), then
\[
    \alpha = z\bar{z} = |z|^2
\]
must be a positive real number, and
\[
    \mathrm{fix}(\iota|_{\C^\times}) = (|z|^2 = \alpha) \cong S^1.
\]
\end{proof}

The next result is the pay-off for the work of this section.
\begin{corollary}[]
    \label{cor:LisMatchingCycle}
    The zero section \(L \subset T^*S^2\) is a matching cycle of the Lefschetz
fibration \(\TSlefFib\).
\end{corollary}
\begin{proof}
    Since \(L\) is fibred by circles, it lives submersively over its
projection by \(\TSlefFib\), which is a smooth embedded path \(p : [-1,1] \to
\cylindersModuliBar\) joining the two critical
values of \(\TSlefFib\). Denote the symplectic parallel transport of this
fibration by \(\phi_t\). By Lemma~1.17 of \cite{evans24kiaslectures} \(L\) is
the trace under \(\phi_t\) of its intersection with the fibre over \(p(0)\),
that is,
\[
    L = \bigcup_{t \in [-1,1]} \phi_t\left(L \cap \TSlefFib^{-1}(p(0))\right).
\]
Since the critical points of \(\TSlefFib\) lie on \(L\), we conclude that it
is indeed a matching cycle.
\end{proof}

\section{Isotopy to a matching cycle}
\label{ch:isotopy}
In this section we use the neck stretching analysis of
Section~\ref{ch:neckStretch} to show that a matching cycle \(\Sigma_k\) of of
\(\pi_{J_k}\) converges to the matching cycle of \(\TSlefFib\). In turn, this
yields a Lagrangian isotopy from \(L\) to \(\Sigma_k\) for sufficiently
large \(k\).

Before discussing convergence of matching cycles, we need to understand how to
construct a sequence of matching paths for \(\pi_{J_k}\) that converges (in
some sense) to the matching path of Corollary~\ref{cor:LisMatchingCycle}. This
is the subject of Section~\ref{sec:convergenceOfMatchingPaths}. Once this is
done, Section~\ref{sec:convergenceOfParTransp} shows that the convergence of
the matching cycles essentially follows from smooth dependence of ODEs on
their defining vector field and initial condition. Finally, 
Section~\ref{sec:constructingTheIsotopy} completes the proof that every
Lagrangian sphere \(L \subset B_{d,p,q} \subset X\) is Lagrangian isotopic to
a matching cycle of a fixed Lefschetz fibration \(\refFib\).

\subsection{Convergence of matching paths}
\label{sec:convergenceOfMatchingPaths}

Let \(x \in T^*S^2\) and recall from Section~\ref{sec:foliationConstruction}
that there is a unique sequence of curves \(f_k^x \in
\overline{\mathcal{M}}_{0,0}(X,F;J_k)\) that converge to a
\(\JTS\)-holomorphic building passing through \(x\). In this section, we'll
primarily be interested in the convergence to the bottom levels in the
Weinstein neighbourhood \(\weinsteinNbhd\) of \(L\).
Before we can talk about convergence of matching cycles of the fibrations
\(\pi_{J_k} : \XJ{J_k} \to \overline{\mathcal{M}}_{0,0}(\XJ{J_k},F;J_k)\) to the zero
section matching cycle of \(\TSlefFib : T^*S^2 \to \cylindersModuliBar\), we
need to make sense of what it means for matching paths in the bases
\(\overline{\mathcal{M}}_{0,0}(\XJ{J_k},F;J_k)\) to converge to the matching path of
Corollary~\ref{cor:LisMatchingCycle}. The global charts of
\(\overline{\mathcal{M}}_{0,0}(\XJ{J_k},F;J_k)\) constructed in
Section~\ref{sec:compactLefFib} are not fit for this purpose, since the neck
stretching analysis of Section~\ref{sec:neck-stretch} shows that, as \(k \to
\infty\), the subset \(\weinsteinModuli{J_k} \subset
\overline{\mathcal{M}}_{0,0}(\XJ{J_k},F;J_k)\) of 
curves that pass through \(\weinsteinNbhd\) shrinks to a
point.
We construct a new chart of \(\overline{\mathcal{M}}_{0,0}(\XJ{J_k},F;J_k)\)
with image contained in \(\weinsteinModuli{J_k}\) using one of the
\(\JTS\)-holomorphic \(\cotangentPlanesModuli_+\) planes in \(T^*S^2\). Choose
a simple Reeb orbit \(\gamma_0\) that is neither \(\gamma\) nor
\(\bar{\gamma}\). Fix the unique plane \(P_+^{\gamma_0} \in
\cotangentPlanesModuli_+\) asymptotic to \(\gamma_0\) along with a
parametrisation \(u : \C \to T^*S^2\) of it.

\begin{lemma}[]
    The plane \(P_+^{\gamma_0}\) intersects each curve in
\(\cylindersModuliBar\) exactly once, transversely. In particular, for each
\(C \in \cylindersModuliBar\),
\[
    i_U(P_+^{\gamma_0},C) = i_C(P_+^{\gamma_0},C) = 1.
\]
Therefore, the composition \(\TSlefFib \circ u : \C \to \cylindersModuliBar\)
is a global chart.
\end{lemma}
\begin{proof}
    The intersection claim follows easily from homotopy invariance of \(i_C\)
and the fact that \(\gamma_0 \ne \gamma,\bar{\gamma}\) (see
Section~\ref{sub:cotangentPlanes})
\[
    i_C(P_+^{\gamma_0},C) =
    i_C(P_+^{\gamma_0}, P_+^\gamma + P_-^{\bar{\gamma}}) =
    i_C(P_+^{\gamma_0}, P_-^{\bar{\gamma}}) = 1.
\]
By positivity of intersections, there is exactly one positive, transverse
intersection as claimed. The fact that this provides a global chart follows by
definition of the smooth structure on \(\cylindersModuliBar\).
\end{proof}

\begin{remark}
    Since \(PSL(2,\C)\) acts transitively on the configuration space of 3
points on \(\CP^1\), we can reparametrise \(u\) to ensure that the two
critical values of \(\TSlefFib\) correspond to the points \(\pm 1\).
\end{remark}

Recall that the neck stretching setup gives, for each \(k\), an almost complex
embedding \(\iota_k : (T_{\le e^k}^*S^2, \JTS) \to (X,J_k)\) with image the
Weinstein neighbourhood \(\weinsteinNbhd{}=T_{\le 1}^*L\) of \(L\). This
allows us to embed larger and larger portions of the plane \(P_+^{\gamma_0}\)
as the neck gets longer. The idea is that, eventually, the curves in
\(\weinsteinModuli{J_k}\) get close enough to their limiting bottom level
curves in \(\cylindersModuliBar\) so that the transverse intersection property
of the previous lemma holds for them too.

\begin{lemma}[]
    \label{lem:charts}
    There exists an integer \(K > 0\) such that for all \(k \ge K\) every
curve in \(\weinsteinModuli{J_k}\) intersects \(P_+^{\gamma_0}\) exactly once
transversely. Therefore, the parametrisation \(u\) of \(P_+^{\gamma_0}\)
facilitates a chart of the moduli spaces
\(\overline{\mathcal{M}}_{0,0}(\XJ{J_k},F;J_k)\) for all \(k \ge K\). Moreover,
\(K\) can be chosen large enough so that there are only the two
\emph{relevant} critical values of the maps \(\pi_{J_k}\) contained in the
image of this chart. That is, the two that correspond to the nodal curves
\(P_\pm^\gamma + P_\mp^{\bar{\gamma}}\).
\end{lemma}
\begin{proof}
    Fix \(x \in \weinsteinNbhd\), then there is a unique sequence \(f_k^x \in
\weinsteinModuli{J_k}\) of fibre curves that pass through \(x\). We first
prove that the result holds for this sequence. That is, that there exists
\(K_x > 0\) such that, for all \(k \ge K_x\), \(f_k^x\) transversely
intersects \(P_+^{\gamma_0}\) exactly once. Recall from
Section~\ref{sec:foliationConstruction} that the sequence \(J_k\) was chosen
to ensure that each of the sequences \(f_k^x\) converges\footnote{Note that
    some of the sequences \(f_k^x\) necessarily contain nodal curves, since
    there are two nodal curves in \(\cylindersModuliBar\). We have not defined
    exactly what it means for such a sequence (consisting of potentially
    non-smooth \(J_k\)-holomorphic curves) to converge to a holomorphic
    building. However, since we are only interested in the convergence
\emph{away from the nodes}, we will not go into the details.} to a
\(J\)-holomorphic building \(F^x : \Sigma^* \to X^*\),
where the bottom level \(\bottomlevel{(F^x)}\) is the element of
\(\cylindersModuliBar\) passing through \(x\). The nature of the convergence
\(f_k^x \to F^x\) implies that, for any \(\epsilon > 0\), there exists
\(K_{x,\epsilon}\) such that (up to reparametrisation) \(f_k^x\) is
\(\epsilon\)-close to \(\bottomlevel{(F^x)}\) in the \(C^1_\mathrm{loc}\)
topology.

Now, \(\bottomlevel{(F^x)}\) and \(P_+^{\gamma_0}\) intersect exactly once,
but since they are non-compact, we need to be slightly careful before claiming
that all nearby curves satisfy the same property. Fortunately, they are
non-compact in a very controlled way, due to the asymptotic convergence to
the Reeb orbits. The convergence \(f_k^x \to F^x\) guarantees,\footnote{See
Definition~2.7(d) of \cite{cieliebakMohnkeSFTCompactness}.} for all
sufficiently large \(k\), that the curves \(f_k^x\) are bounded away from the
cylindrical end of \(P_+^{\gamma_0}\). Therefore, the only intersections
between \(f_k^x\) and \(P_+^{\gamma_0}\) can occur in a compact set. Thus, by
standard differential topology, any such map transversely
intersects \(P_+^{\gamma_0}\) in exactly the same number of points as
\(\bottomlevel{(F^x)}\) --- that is, exactly once. Therefore, we deduce that
there exists \(\epsilon_x > 0\), and \(K_x = K_{x,\epsilon_x} > 0\), such
that, for all \(k \ge K_x\), \(f_k^x\) is \(\epsilon_x\)-close to \(\bottomlevel{(F^x)}\)
\emph{and} bounded away from the cylindrical end of \(P_+^{\gamma_0}\), from
which it follows that \(f_k^x\) intersects \(P_+^{\gamma_0}\) exactly once,
transversely.

We turn to showing that the result continues to hold as stated. This is a
consequence of the local description of the moduli spaces
\(\weinsteinModuli{J_k}\) and compactness of the Weinstein neighbourhood
\(\weinsteinNbhd\). First, we show that if the transverse intersection
property holds for a curve \(f_{K_x}^x \in \weinsteinModuli{J_{K_x}}\), then
it holds for all nearby curves \(f \in \weinsteinModuli{J_{K_x}}\). This is
achieved by finding a tubular neighbourhood \(N_x\) of \(f_{K_x}^x\) such that
the intersection \(P_+^{\gamma_0} \cap N_x\) is connected and \(C^1\)-close to
a fibre of \(N_x\). The claim then follows from standard differential
topology.

If the curve \(f_{K_x}^x\) is nodal, then the tubular neighbourhood \(N_x\) is
constructed via the gluing map. The nodal points of curves in
\(\weinsteinModuli{J_k}\) converge to the nodal points of
the curves in \(\cylindersModuliBar\), and since \(P_+^{\gamma_0}\) intersects
the curves in \(\cylindersModuliBar\) away from the nodes, we can increase \(K_x\) so that either
\(f_{K_x}^x\) is no longer nodal, or the node is bounded away from
\(P_+^{\gamma_0}\). Note that, away from nodal points the gluing map is a
diffeomorphism, and this is where the unique transverse intersection
\(f_{K_x}^x \cap P_+^{\gamma_0}\) occurs. By varying the gluing parameter, we
obtain arbitrarily small tubular neighbourhoods \(N_x\) of \(f_{K_x}^x\).

On the other hand, if \(f_{K_x}^x\) is smooth we
use the local description of the moduli space of smooth curves
\(\mathcal{M}_{0,0}(X,F;J_{K_x})\) to express nearby curves as sections of
the (trivial) normal bundle \(\nu f_{K_x}^x\). Thus, we obtain the tubular neighbourhood \(N_x\). In both the
smooth and the nodal cases, \(f_{K_x}^x\) is bounded away from the
cylindrical part of \(P_+^{\gamma_0}\), so we can choose the tubular
neighbourhood \(N_x\) small enough so that the intersection \(N_x \cap
P_+^{\gamma_0}\) can be isotoped to a fibre of \(N_x\). In both the smooth and
the nodal cases, this implies that the intersection property continues to hold
for all sufficiently nearby curves \(f \in \weinsteinModuli{J_{K_x}}\). In
particular, there exists a number \(\delta_x > 0\) such that, for all \(f \in
\weinsteinModuli{J_{K_x}}\) satisfying\footnote{This metric refers to any that
induces the manifold topology on \(\weinsteinModuli{J_{K_x}}\).}
\(d(f,f_{K_x}^x) < \delta_x\), \(f\) intersects \(P_+^{\gamma_0}\) exactly
once transversely.

In summary, so far we have proved that, for all \(x \in \weinsteinNbhd\),
there exists \(K_x > 0\) and \(\delta_x > 0\) such that, for all \(y \in
\weinsteinNbhd\), and \(f_{K_x}^y \in \weinsteinModuli{J_{K_x}}\) satisfying
\(d(f_{K_x}^y,f_{K_x}^x) < \delta_x\), we have that \(f_{K_x}^y\) intersects
\(P_+^{\gamma_0}\) exactly once, transversely. Furthermore, since the
convergence \(f_k^y \to F^y\) is \(C^1\)-monotonic, we can upgrade the
statement to hold for all \(k \ge K_x\). Specifically, for all \(k \ge K_x\)
and \(y \in \weinsteinNbhd\) such that \(d(f_{K_x}^y, f_{K_x}^x) < \delta_x\),
we have that \(f_k^y\) intersects \(P_+^{\gamma_0}\) as desired. We now
harness compactness to turn this into a global statement.

For a given \(x \in \weinsteinNbhd\), the curves near \(f_{K_x}^x\) form a
local foliation around \(x\). Therefore, we have that
\[
    U_x := \bigcup_{\substack{f \in \weinsteinModuli{J_{K_x}} \\
    d(f,f_{K_x}^x) < \delta_x}} \im f \cap \weinsteinNbhd
\]
is a neighbourhood of \(x\). Thus, we obtain a cover of \(\weinsteinNbhd\).
Since it's compact, we can pass to a finite subcover \(U_{x_1}, \ldots,
U_{x_k}\) and define
\begin{equation}
    K := \max\{K_{x_1}, \ldots, K_{x_k}\} < \infty.
\end{equation}
Then, as any point \(x \in \weinsteinNbhd\) is necessarily contained in some
\(U_{x_i}\), we have that \(d(f_{K_{x_i}}^x,f_{K_{x_i}}^{x_i}) <
\delta_{x_i}\), which implies that \(f_k^x\) intersects \(P_+^{\gamma_0}\)
exactly once, transversely, for all \(k \ge K \ge K_{x_i}\). As any curve \(f
\in \weinsteinModuli{J_k}\) is equal to \(f_k^x\) for some \(x \in
\weinsteinNbhd\), this completes the proof.

The claim about the critical values follows from the discussion on choosing
the sequence \(J_k\) in Section~\ref{sec:foliationConstruction}.
\end{proof}

As a result, for sufficiently large \(k\), we can choose a sequence of
matching paths \(p_k : [-1,1] \to \overline{\mathcal{M}}_{0,0}(\XJ{J_k},F;J_k)\)
joining the two relevant critical points, such that, in the coordinates of the
charts given by Lemma~\ref{lem:charts}, \(p_k\) converges to the matching path
\(p : [-1,1] \to \C\) of the Lefschetz fibration \(\TSlefFib\) associated to
\(L\). Our next goal is to show that we can choose \(k\) large enough so that
the matching cycle of \(\pi_{J_k}\) associated to \(p_k\) is Lagrangian
isotopic to \(L\).

\subsection{Convergence of the parallel transport}
\label{sec:convergenceOfParTransp}
\subsubsection{The parallel transport of \(\TSlefFib : T^*S^2 \to
\cylindersModuliBar\)}

    The problem of symplectic parallel transport involves lifting a vector
field defined over a compact path in the base to a horizontal vector field in
the total space and integrating it. This amounts to solving an ODE over a
compact family of fibres of the Lefschetz fibration \(\TSlefFib\). However,
since these fibres are non-compact, we need to justify why solutions to this
ODE exist for all time.

More precisely, fix an embedded path \(\gamma : [0,1] \to
\cylindersModuliBar\) whose image avoids the critical values of \(\TSlefFib\).
Taking symplectic orthogonal complements of the tangent spaces to the regular
fibres of \(\TSlefFib\) yields a field of \emph{horizontal} planes \(H \subset
TT^*S^2\), such that the restriction \(\ud\TSlefFib|_H : H \to
T\cylindersModuliBar\) is an isomorphism. Therefore, there exists a unique
horizontal lift \(\tilde{X}\) of the vector field \(X = \der{}{t}p\). This
defines an ODE on the space \(\TSlefFib^{-1}(\im \gamma)\).

\begin{lemma}[]
    \label{lem:TSparallelTransport}
    The symplectic parallel transport of the Lefschetz fibration \(\TSlefFib :
T^*S^2 \to \cylindersModuliBar\) is well-defined.
\end{lemma}
\begin{proof}
    The curves in the moduli space \(\cylindersModuliBar\) are finite energy
\(\JTS\)-holomorphic curves and so, by the results of Hofer, Wysocki, and
Zehnder \cite[Theorem~1.3]{HWZpropertiesIV}, Bourgeois
\cite[\S{}3.3]{bourgeoisThesis}, and Mora-Donato
\cite[Proposition~1.2]{moraThesis}, they satisfy exponential convergence to
their asymptotic Reeb cylinders. The fixed path \(\gamma :  [0,1] \to
\cylindersModuliBar\) yields a compact family \([0,1] \cong \kappa \subset
\cylindersModuliBar\) of such curves, and so, there exists a compact Liouville
subdomain of \(T^*S^2\) outside of which, all the curves in \(\kappa\) satisfy
the exponential convergence estimates.

Fix the standard Riemannian metric on \(T^*S^2\) coming from the embedding
\[
    T^*S^2 = \{(p,q) \in \R^3 \times \R^3 \mid |q|=1,\,
	\langle p,q \rangle = 0\},
\]
and consider the Hamiltonian \(H : T^*S^2 \to \R : H(p,q) =
\frac{1}{2}|p|^2\). The \(\omega_\mathrm{can}\)-orthogonal complement of the
tangent space of a Reeb cylinder is a contact plane, which is contained in
\(\ker \ud H\). Therefore, the exponential convergence of the curves implies
that the horizontal spaces converge exponentially to the contact planes. In
particular, the horizontal lift \(\tilde{X}\) satisfies
\[
    |L_{\tilde{X}} H| < Ce^{-ds},
\]
where \(C > 0\) and \(d > 0\) are constants, and \(s \in [s_0,\infty)\) is the
Liouville coordinate on \(T^*S^2\). An integral curve \(\alpha\) of
\(\tilde{X}\) escapes to infinity if, and only if, \(|H(\alpha)| \to \infty\).
Therefore, the above estimate shows that this is impossible, as we are
integrating over the compact set \([0,1]\).
\end{proof}

\subsubsection{The parallel transport of \(\pi_{J_k} : \XJ{J_k} \to \C\)}

Recall from Lemma~\ref{lem:charts} that, for large enough \(k\), we can
identify the bases of the Lefschetz fibrations \(\pi_{J_k}|_{\weinsteinNbhd}\)
and \(\TSlefFib|_{\weinsteinNbhd}\) with a bounded disc \(D \subset \C\).

\begin{lemma}[]
    \label{lem:transportConvergence}
    Let \(\gamma_k : [0,1] \to D\) be a sequence of embedded paths converging
\(C^1\) to \(\gamma : [0,1] \to D\). Suppose that each path avoids the
critical locus of the corresponding Lefschetz fibration, then parallel
transport \(\phi_{\gamma_k}\) of \(\pi_{J_k}|_{\weinsteinNbhd}\) over
\(\gamma_k\) is \(C^0\)-close to that of \(\TSlefFib|_{\weinsteinNbhd}\) over
\(\gamma\).
\end{lemma}
\begin{proof}
    Since we are dealing with compact families of curves, so one can show that
the fibres of \(\pi_{J_k}|_{\weinsteinNbhd}\) living over \(\gamma_k\)
converge uniformly to those of \(\TSlefFib|_{\weinsteinNbhd}\) over
\(\gamma\). It follows that the ODE defining the parallel transport
\(\phi_\gamma\) of \(\TSlefFib\) can be arbitrarily well approximated by that
defining \(\phi_{\gamma_k}\) by increasing \(k\). Therefore, by smooth
dependence of ODEs on the initial condition and the vector field that defines
it (see \cite[Appendix~B]{duistermaatKolk} for example) we obtain the result.
\end{proof}

\begin{lemma}[]
    \label{lem:matchingCycleC0Close}
    Let \(\gamma_k : [-1,1] \to D\) be a sequence of matching paths converging
to the matching path \(\gamma : [-1,1] \to D\) corresponding to the zero
section matching cycle of \(\TSlefFib\). Then, for sufficiently large \(k\),
the matching cycle \(\Sigma_k\) of \(\pi_{J_k}|_{\weinsteinNbhd}\) associated
to \(\gamma_k\) is \(C^0\)-close to the zero section \(L\).
\end{lemma}
\begin{proof}
    Since the critical points of \(\pi_{J_k}|_{\weinsteinNbhd}\) converge to
those of \(\TSlefFib|_{\weinsteinNbhd}\), smooth dependence of ODEs implies
that we have \(C^0\) convergence of the vanishing thimbles associated to the
paths \(\gamma_k^- := \gamma_k|_{[-1,0]}\) and \(\gamma_k^+ :=
\gamma_k|_{[0,1]}\) in a small neighbourhood of the critical points.
Therefore, we can apply Lemma~\ref{lem:transportConvergence} to the
restrictions \(\gamma_k^-|_{[-1+\epsilon,0]}\) and
\(\gamma_k^+|_{[0,1-\epsilon]}\) for some small \(\epsilon > 0\) to deduce
that the vanishing thimbles of \(\pi_{J_k}\) over the paths \(\gamma_k^\pm\)
can be made \(C^0\)-close to those of \(\TSlefFib\) over \(\gamma^\pm\) by
increasing \(k\).

Choose \(k\) large enough so that the vanishing thimbles are very close to
those of \(\TSlefFib\) and in particular, are contained in the Weinstein
neighbourhood \(\weinsteinNbhd\). To form the matching cycle a 
deformation of the symplectic structure on \(\weinsteinNbhd\) is made to
account for the fact that the vanishing cycles over \(\gamma_k(0)\) may not
agree (see \cite[Lemma~15.3]{seidelBook}). A Moser-type argument
(\cite[Lemma~7.1]{seidelBook}) is then used to map the resulting sphere back
to the original symplectic structure. Since the vanishing thimbles themselves
are \(C^0\)-close, the deformation (and resulting Moser isotopy) can be made
so that the matching sphere \(\Sigma_k\) remains \(C^0\)-close to \(L\). This
completes the proof.
\end{proof}

\subsection{Constructing the isotopy}
\label{sec:constructingTheIsotopy}

\begin{theorem}[]
    \label{thm:isotopyToMatchingCycle}
    There exists \(k\) sufficiently large such that \(L\) is Lagrangian
isotopic to a matching cycle of \(\pi_{J_k}\).
\end{theorem}
\begin{proof}
    Lemma~\ref{lem:matchingCycleC0Close} shows that the relevant matching
cycle \(\Sigma_k\) is a Lagrangian sphere contained in a neighbourhood of
\(L\) that is symplectomorphic to\linebreak \((T_{\le e^{R_k}}^*S^2,
e^{-R_k}\omega_\mathrm{can})\) for some positive number \(R_k > 0\).
Therefore, we can apply Hind's theorem \cite[Theorem~18]{hind12Stein} on
uniqueness of Lagrangian spheres in \(T^*S^2\) to obtain a Lagrangian isotopy
from \(\Sigma_k\) to \(L\). If necessary, by re-scaling by the
Liouville flow, we can ensure that the isotopy is supported in \(T_{\le
e^{R_k}}^*S^2\). Finally, since this neighbourhood of \(L\) symplectically embeds
into \(X\), we obtain the isotopy in \(X\) as desired.
\end{proof}

In summary, we have found a Lefschetz fibration \(\pi_{J_k} : \XJ{J_k} \to \C\) and
a Lagrangian isotopy supported in a small neighbourhood of \(L\) taking \(L\)
to a matching cycle \(\Sigma_k\). The next task is to ``undo the neck
stretch'' and find another Lagrangian isotopy taking \(\Sigma_k\) to a
matching cycle of some fixed Lefschetz fibration with
respect to which we will do the remaining computations.

We use the fact that the space \(\mathcal{J}(D')\) of compatible
almost complex structures for which each component of the divisor \(D' \subset
X\) is \(J\)-holomorphic is connected. In fact, we restrict ourselves even
further to the subset of \(\mathcal{J}(D')\) of almost complex structures that
are \emph{fixed} in a neighbourhood of the \(F\)-component of \(D'\). In more
detail, consider a neighbourhood \(N\) of this component of the form shown in
Figure~\ref{fig:fixedNeighbourhood}.
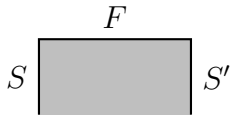
\begin{figure}\centering \begin{tikzpicture}[
    moment polytope,
]
    \fill (0,0) -- node [left] {\(\fibrationSectionL\)} (0,1) -- node[above]
	{\(F\)} (2,1) -- node [right] {\(\fibrationSectionR\)} (2,0) -- cycle;
    \draw (0,0) -- (0,1) -- (2,1) -- (2,0);
\end{tikzpicture}
    \caption{
	The neighbourhood \(N\) of the fibre component of \(D'\) in which we
	fix the almost complex structures to be equal to the usual product
	complex structure coming from \(S^2 \times S^2\).
    }
    \label{fig:fixedNeighbourhood}
\end{figure}
Remark~\ref{rem:productACS} ensures that we can choose\footnote{In fact, in
    Section~\ref{sec:mainTheoremProof} we shall explicitly construct a
    specific \(\Jref\).} \(\Jref \in \mathcal{J}(D')\) so that it agrees with the product
almost complex structure coming from \(S^2 \times S^2\) in \(N\). Let \(U := X
\backslash N\) and denote the subset of \(\mathcal{J}(D')\) of almost complex
structures that agree with \(\Jref\) in \(N=X\backslash U\) by
\(\mathcal{J}(U,\Jref)\). That is,
\[
    \mathcal{J}(U,\Jref) := \{ J \in \mathcal{J}(D') \mid J|_{X\backslash U} =
    J|_N = \Jref|_N = \Jref|_{X\backslash U}\}.
\]
This space is connected, which follows from S\'evennec's argument (see
\cite[Proposition~1.1.6]{audin94SympAndAlmostComplexManifolds} for example).
Observe that, for a suitable choice of neighbourhood \(N\), all the stretched
almost complex structures \(J_k\) are contained in \(\mathcal{J}(U,\Jref)\).

The idea is to use connectedness of \(\mathcal{J}(U,\Jref)\) to choose a
path \(J_t\) of almost complex structures from \(J_0 = J_k\) to \(J_1 =
\Jref\). Applying the results of Section~\ref{sec:fibreModuliSpaces} yields a
path of Lefschetz fibrations \(\pi_t = \pi_{J_t} : \XJ{J_t} \to \C\) from
\(\pi_0 = \pi_{J_k}\) to \(\pi_1 = \pi_{\Jref}\). This path of
fibrations induces a path in the configuration space \(C(\C,d)\) of \(d\)
points in \(\C\) corresponding to the positions of the critical values of
\(\pi_t\) (see Remark~\ref{rem:lefFibMarkedPoints}(1)).

We note two features of this path. Firstly, recall that the point at infinity in
the bases \(\overline{\mathcal{M}}_{0,0}(\XJ{J_t},F;J_t) \cong \C\)
corresponds to the exotic stable curve \(\infinityCurve{J_t}\).
Therefore, the points must remain in a bounded subset by the fact that
Lefschetz critical fibres never intersect \(\infinityCurve{J_t}\).
Secondly, since the almost complex structures \(J_t\) are fixed on \(N\) and
\(N\) is foliated by smooth \(J_t\)-holomorphic \(F\)-curves the
points also remain bounded away from \(0 \in \C\). As a result, this path in
\(C(\C,d)\) is really a path in \(C(A,d)\) where \(A \cong [0,1] \times S^1\)
is a compact annulus. Now, choose an isotopy of matching paths \(\gamma_t :
[-1,1] \to A\) for \(\pi_t\) such that \(\gamma_0 = \gamma_k\) is the matching
path for \(\Sigma_k\) from Lemma~\ref{lem:matchingCycleC0Close}. Forming the
corresponding matching cycles \(L_t\) yields a Lagrangian isotopy from \(L_0
= \Sigma_k\) to the matching cycle of \(\refFib\) over \(\gamma_1\). This
isotopy is disjoint from open neighbourhoods of \(\infinityCurve{J_t}\) and
the \(F\)-component of the divisor \(D' \subset X\). However, \emph{a priori}
it may pass through the sections \(\fibrationSectionL\) and
\(\fibrationSectionR\). This is a problem since we want the Lagrangian isotopy
to be supported in a subset of \(X\) that is symplectomorphic to a subset of
\(B_{d,p,q}\). The point of the next result is to avoid this behaviour by
explicitly altering the Lefschetz fibrations \(\pi_t\).

The curves in \(\overline{\mathcal{M}}_{0,0}(X,F;J)\) form singular foliations
\(\mathscr{F}_t\) on \(X\). The singular leaves are exactly the singular
Lefschetz fibres of \(\pi_t\), and one \emph{exotic} leaf given by
\(\infinityCurve{J_t}\). In small tubular neighbourhoods of the
\(J_t\)-holomorphic sections \(\fibrationSectionL\) and
\(\fibrationSectionR\), these foliations are smooth with leaves given by
symplectic \(2\)-discs. We denote these by \(\mathscr{D}_t\).
\begin{lemma}[]
    \label{lem:gompf}
    There exists an isotopy \(\mathscr{D}_{s,t}\) of foliations such that
\(\mathscr{D}_{0,t} = \mathscr{D}_t\) and \(\mathscr{D}_{1,t}\) are foliations
by symplectic \(2\)-discs intersecting the sections \(\fibrationSectionL\) and
\(\fibrationSectionR\) positively and symplectically orthogonally. The isotopy
can be chosen so that it is invariant in \(s\) outside of an arbitrarily small
neighbourhood of each section. As a result, we obtain singular foliations
\(\mathscr{F}_{s,t}\) of \(X\) such that, excluding a single exotic leaf,
\(\mathscr{F}_{s,t}\) forms a \emph{symplectic Lefschetz
fibration}\footnote{That is, Lefschetz fibrations whose fibres are only
symplectic submanifolds, and not necessarily \(J\)-holomorphic.} on \(X\).
\end{lemma}
\begin{proof}
    Put \(S_1 = \fibrationSectionL\) and \(S_2 = \fibrationSectionR\). Note
that near a section \(S_i\), the fibrations \(\pi_t\)
are foliations \(\mathscr{D}_t\) of symplectic 2-discs whose leaves intersect
\(S_i\) transversely and positively. Therefore, we can use a
variation of an argument of Gompf \cite[Lemma~2.3]{gompf95Construction} to
construct an isotopy of foliations \(\mathscr{D}_{s,t}\) such that: (1) the
point of intersection of a leaf and \(S_i\) is invariant in
\(s\); and (2) outside of a small neighbourhood of \(S_i\),
the leaves are also invariant in \(s\). Furthermore, the leaves
of \(\mathscr{D}_{1,t}\) intersect \(S_i\) symplectically
orthogonally. Consequently, we can glue the symplectic discs in
\(\mathscr{D}_{s,t}\) to the leaves of \(\mathscr{F}_t\) to obtain new
foliations \(\mathscr{F}_{s,t}\). By construction, these are constant in \(s\)
in a complement of small neighbourhoods of the sections \(S_i\), which implies
that (except for the one exotic fibre corresponding to the
\(\infinityCurve{J_t}\) curve) their leaves form Lefschetz fibrations, since
the existence of a Lefschetz chart at a critical point is a local condition.
This completes the proof.
\end{proof}

Let \(X_t\) denote \(X\) with the exotic leaf of \(\mathscr{F}_{1,t}\)
excised. Write \(\pi_{1,t} : X_t \to \C\) for the corresponding Lefschetz
fibration. Define the reference Lefschetz fibration by \(\refFib :=
\pi_{1,1}\), and observe that it does not depend on the Lagrangian sphere
\(L\). In addition to the tubular neighbourhood \(N\) defined earlier, choose
a tubular neighbourhood foliated by leaves of \(\mathscr{F}_{1,1}\) of the
exotic leaf, then define \(\mathring{X} \subset U \subset X\) to be the
excision of this as well as the sections \(\fibrationSectionL\) and
\(\fibrationSectionR\). The restriction
\(\refFib|_{\mathring{X}}\) is a Lefschetz fibration over the compact annulus.
\begin{corollary}[]
    \label{cor:isotopyToMatchingCycle}
    There exists a Lagrangian isotopy from \(L\) to a matching cycle of
\(\refFib\) that is supported in \(\mathring{X}\).
\end{corollary}
\begin{proof}
    Theorem~\ref{thm:isotopyToMatchingCycle} shows that it suffices to find
such an isotopy from the matching cycle \(\Sigma_k\) to one of \(\refFib\).
Observe that \(\Sigma_k\) is still a matching cycle for the deformed fibration
\(\pi_{1,0}\) since none of the structure changes near \(\Sigma_k\) throughout
the isotopy \(\mathscr{F}_{s,t}\). Therefore, we can construct the Lagrangian
isotopy using the matching paths \(\gamma_t\) defined earlier, except this
time we form the matching cycles with respect to the maps \(\pi_{1,t}\). To
see that this has support as claimed, note that the orthogonality condition
ensures that the parallel transport maps of \(\pi_{1,t}\) preserve the
sections \(\fibrationSectionL_i\). Therefore, they preserve their complements
too. Since \(\Sigma_k\) lives in this complement, the claim follows from this
and the discussion preceding Lemma~\ref{lem:gompf}.
\end{proof}

\section{Mapping class groups of surfaces and symplectomorphisms}
\label{ch:mcg}
In this section we harness the theory of mapping class groups of surfaces to
improve Corollary~\ref{cor:isotopyToMatchingCycle} to the main result
Theorem~\ref{thm:intro:main}. Corollary~\ref{cor:isotopyToMatchingCycle} tells
us that any Lagrangian sphere \(L \subset B_{d,p,q}\) is Lagrangian isotopic
to a matching cycle of the Lefschetz fibration \(\refFib\). To convert this
into an isotopy statement phrased in terms of Dehn twists, we need to
understand their relation to (isotopy classes of) matching paths. This will be
facilitated by the fact that the natural action of the mapping class group of
the \(d\)-punctured annulus \(A^{\times d} := [0,1] \times S^1 \backslash
\{d\text{ points}\}\) is transitive on matching paths, and so, it will suffice
to understand particularly simple matching paths.

\subsection{The mapping class group of the punctured annulus}

The mapping class group of a surface is a classical object with plenty of rich
theory. Farb and Margalit's book \cite{farbMargalit}  --- in particular,
Sections~1--4 --- contains everything we will use here. We recall the main
definition.
\begin{definition}[]
    Let \(S\) be an oriented surface (possibly with boundary and punctures).
The \emph{mapping class group of \(S\)}, \(\Mod(S)\), is defined to be the group
of connected components of the group of orientation preserving
diffeomorphisms\footnote{In fact, as Section~1.4.2 of \cite{farbMargalit}
    explains, we are free to consider homeomorphisms or diffeomorphisms
interchangeably.} of \(S\) that fix the boundary point wise. That is,
\[
    \Mod(S) := \pi_0(\mathrm{Diff}^+(S,\partial S)).
\]
\end{definition}

Recall from Section~\ref{sec:constructingTheIsotopy}, the restriction of the
Lefschetz fibration \(\refFib\) to \(\mathring{X}\) has base a \(d\)-punctured
annulus \(A\). Therefore, we seek the understand the mapping class group
\(\Mod(A)\).
\begin{figure}\centering \begin{tikzpicture}
    \draw (0,0) circle [radius=0.5cm];
    \draw (0,0) circle [radius=3cm];
    \foreach \x in {0.66,1.33,2.0}
	\draw[xshift=\x cm] (0.5cm,0) node {\(\times\)};
    \draw [dotted, xshift=0.5cm] (0.66cm,0) -- node [below] {\(\gamma_1\)}
	(1.33cm,0) -- node [below] {\(\gamma_2\)} (2.0cm,0);
\end{tikzpicture}
    \caption{
	The punctured annulus and the matching paths \(\gamma_i\).
    }
    \label{fig:puncturedAnnulus}
\end{figure}
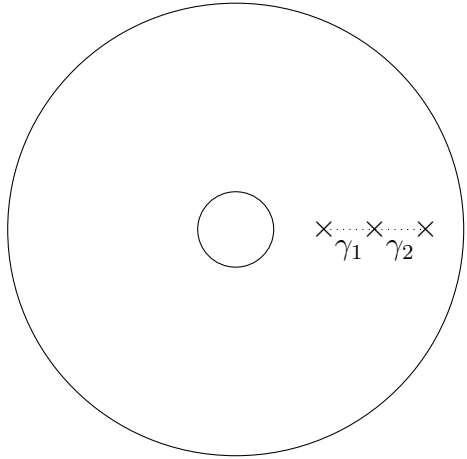
In Appendix~\ref{app:annulusMCG} we compute a presentation of \(\Mod(A)\):
\begin{example}[Appendix~\ref{app:annulusMCG}]
    The mapping class group of the punctured annulus is isomorphic to the
direct product of the \emph{annular braid group} with a copy of \(\Z\). This
is finitely generated, and each generator is one of three types: the Dehn twist
\(T\) about the central boundary, the so-called central twist \(\tau\), and
half-twists \(\sigma_i\) around the straight line paths \(\gamma_i\) joining the
adjacent \(i\) and \((i+1)\)th punctures. Explicitly, the
presentation is:
\[
    \Mod(A) = \left\langle \tau, \sigma_1,\ldots,\sigma_{n-1} \,\left|\,
    \begin{array}{c}
	(\tau\sigma_1)^2 = (\sigma_1\tau)^2,\\
	\tau\sigma_i = \sigma_i\tau \; \forall i>1,\\
	\sigma_i\sigma_{i+1}\sigma_i = \sigma_{i+1}\sigma_i\sigma_{i+1},\\
	\sigma_i\sigma_j = \sigma_j\sigma_i \;\forall |i-j|>1 
    \end{array}\right.\right\rangle \times \langle T \rangle.
\]
The behaviour of each generator is characterised by the diagrams in
Figure~\ref{fig:MCGgenerators}.
\begin{figure}\centering \begin{tikzpicture}
    \draw (0,0) circle [radius=0.5cm];
    \draw (0,0) circle [radius=1.5cm];
    \draw (1.0cm,0) node {\(\times\)};
    \draw [dotted] (1.0cm,0) -- (1.5cm,0);
    \draw [dashed,domain=0:360,variable=\t,smooth,samples=30]
	plot (\t:{(1-\t/360)*1.5 cm + (\t/360)*1.0 cm});
    \draw (0,-1.5cm) node [below=5pt] {\(\tau\)};
    \begin{scope}[xshift=4cm]
	\draw (0,0) circle [radius=0.5cm];
	\draw (0,0) circle [radius=1.5cm];
	\draw (1.0cm,0) node {\(\times\)};
	\draw [dotted] (0.5cm,0) -- (1.0cm,0);
	\draw [dashed,domain=0:360,variable=\t,smooth,samples=30]
	    plot (\t:{(1-\t/360)*0.5 cm + (\t/360)*1.0 cm});
	\draw (0,-1.5cm) node [below=5pt] {\(T\)};
    \end{scope}
    \begin{scope}[xshift=8cm]
	\draw (0,0) ellipse [x radius=1.5cm, y radius=0.75cm];        
	\draw (-0.75cm,0) node {\(\times\)} (0.75cm,0) node {\(\times\)};
	\draw [dotted] (-1.5cm,0) -- (-0.75cm,0);
	\draw [dotted] (1.5cm,0) -- (0.75cm,0);
	\draw [dashed] (-1.5cm,0)
	    .. controls (-0.75cm,-0.75cm) and (-0.1cm,-0.5cm) .. (0,-0.5cm)
	    .. controls (0.1cm,-0.5cm) and (0.75cm,-0.2cm) .. (0.75cm,0);
	\draw [dashed] (1.5cm,0)
	    .. controls (0.75cm,0.75cm) and (0.1cm,0.5cm) .. (0,0.5cm)
	    .. controls (-0.1cm,0.5cm) and (-0.75cm,0.2cm) .. (-0.75cm,0);
	\draw (0,-1.5cm) node [below=5pt] {\(\sigma_i\)};
    \end{scope}
\end{tikzpicture}
    \caption{
	Local characterisations of each generator of \(\Mod(A)\). In each
	case, dotted lines map to dashed.
    }
    \label{fig:MCGgenerators}
\end{figure}
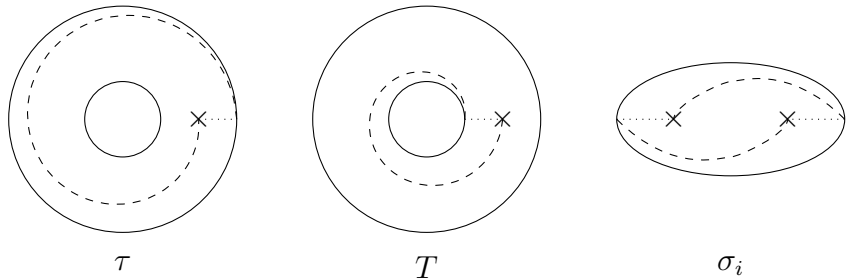
\end{example}

Since any Lagrangian sphere \(L \subset B_{d,p,q}\) is Lagrangian isotopic to
a matching cycle \(\Sigma_\gamma\) fibred over the matching path \(\gamma\),
the following basic fact will be crucial in proving
Theorem~\ref{thm:intro:main}.
\begin{fact*}
    The mapping class group \(\Mod(A)\) acts transitively on the set of
isotopy classes of matching paths \(\mathcal{P}\).
\end{fact*}

\begin{remark}
    In fact, this is true if \(A\) is replaced by any punctured oriented
surface, as can be proved by a cut-and-paste argument appealing to the
classification of surfaces, as the author learned from the MathOverflow answer
\cite{MCGtransitiveActionMathOverflow}.
\end{remark}

Therefore, \(\gamma\) is isotopic to the image of the so-called standard
matching path \(\gamma_1\) (see Figure~\ref{fig:puncturedAnnulus}) under an
element \(\rho \in \Mod(A)\). It follows that the matching cycle
\(\Sigma_\gamma\) is Lagrangian isotopic to \(\Sigma_{\rho(\gamma_1)}\). So to
prove Theorem~\ref{thm:intro:main}, we seek to a correspondence between
symplectomorphisms of \(B_{d,p,q}\) and mapping classes in \(\Mod(A)\).

That the half-twists \(\sigma_i\) correspond to generalised Dehn twists about
the Lagrangian spheres \(\Sigma_{\gamma_i}\) is well known (and proved in
\cite[16h]{seidelBook}). The Dehn twist \(T \in \Mod(A)\) acts
trivially on any class \([\gamma] \in \mathcal{P}\) since no matching path has
an end point lying on the central boundary component. So, our task is to
understand to what symplectomorphism the central twist \(\tau\) corresponds.
More precisely, is there a symplectomorphism \(\phi\) of \(B_{d,p,q}\) such
that \(\phi(\Sigma_{\gamma_1})\) is Lagrangian isotopic to
\(\Sigma_{\tau(\gamma_1)}\)?

\subsection{Some symplectomorphisms of \(B_{d,p,q}\)}
\label{sec:dpqSymplectomorphisms}

The aim of this section is to prove the following result.
\begin{proposition}[]
    There exists a compactly-supported symplectomorphism \(\pqtwist\) of
\(B_{d,p,q}\) arising from the symplectic monodromy of the
\(\frac{1}{p^2}(1,pq-1)\) singularity. The central twist \(\tau \in
\Mod(A)\) corresponds to \(\pqtwist\), that is,
\(\pqtwist(\Sigma_{\gamma_1})\) is Lagrangian isotopic to
\(\Sigma_{\tau(\gamma_1)}.\) Moreover, no iterate
\(\Sigma_{\tau^k(\gamma_1)}\), for \(k\ne 0\), is Lagrangian isotopic to
\(\Sigma_{\gamma_1}\).
\end{proposition}
As a corollary of this, we find that \(\pqtwist\) has infinite order in the
symplectic mapping class group \(\pi_0(\mathrm{Symp}_c(B_{d,p,q}))\).

We adapt Seidel's approach \cite[\S{}4.c]{seidel00graded} to the symplectic
monodromy to suit our situation. We work in
\(\C^3\) with coordinates \(z = (z_1,z_2,z_3)\).
Let \(\psi\) be a cut-off function\footnote{The purpose of which is to ensure
    that the symplectic parallel transport maps used in the
definition of monodromy are well defined.} satisfying
\[
    \psi(t^2) = \begin{cases}
        1, &\text{ if }t \le \frac{1}{3}\\
	0, &\text{ if }t \ge \frac{2}{3}.
    \end{cases}
\]
and define
\begin{equation*}
    \tilde{M}_w := \{(z_1,z_2,z_3) \in \C^3 \mid |z| \le 1,\,
    z_1z_2=z_3^{dp} + \psi(|z|^2)w\}
\end{equation*}
Note that, for some \(\epsilon > 0\) sufficiently small, the manifolds
\(\tilde{M}_w\) with \(0 < |w| \le \epsilon\) are
smooth symplectic manifolds diffeomorphic to the Milnor fibre
\((z_1z_2=z_3^{dp})\cap(|z|\le 1)\) of the \(A_{dp-1}\) singularity
(see Lemmata~4.9 and 4.10 of \cite{seidel00graded}). The
\(\tilde{M}_w\) are invariant under the \(\frac{1}{p}(1,-1,q)\) action on
\(\C^3\), and so (being a subgroup of the unitary group \(U(3)\)) we can take
the quotient to obtain symplectic manifolds \(M_w\) diffeomorphic to
\(B_{d,p,q}\).

The Milnor fibration associated to the singular point \(0 \in M_0\) is
defined to be restriction of the projection \(\C^3 \times S^1 \to S^1\):
\[
    \pi : M := \bigcup_{|w|=\epsilon} M_w \times \{w\} \to S^1,
\]
Pulling back the standard symplectic form
\(\omega_{\C^3}\) to \(M\) yields a closed 2-form \(\Omega\) whose
restriction to each fibre is symplectic. Therefore, we can define the
symplectic parallel transport of this fibration. Write \(\omega =
\Omega|_{M_\epsilon}\).

\begin{definition}
    Winding once (anticlockwise) around the base \(S^1\) yields the
\emph{symplectic monodromy map} \(f \in \Aut(M_\epsilon, \partial M_\epsilon,
\omega)\).
\end{definition}

\begin{lemma}[]
    \label{lem:pqMonodromyInducesCentralTwist}
    In the case \(d=1\), the monodromy, which we now denote by \(\pqtwist\),
of \(B_{p,q}\) induces the central twist \(\tau \in \Mod(A)\) in the mapping
class group of the punctured annulus \(A\). More precisely, let \(\gamma\) and
\(\gamma' = \tau(\gamma)\) be the dotted and dashed (respectively) vanishing
paths in the left panel of Figure~\ref{fig:MCGgenerators}. Form the associated
vanishing thimbles \(D_{\gamma}\) and \(D_{\gamma'}\). Then
\(\pqtwist(D_{\gamma})\) is Lagrangian isotopic to \(D_{\gamma'} =
D_{\tau(\gamma)}\).
\end{lemma} 
\begin{proof}
Consider the map \(\varpi : M \to \C : \varpi(z,w) = z_3^p\). We have
the following diagram of maps:
\[
\begin{tikzcd}
    M \arrow[d, "\varpi"] \arrow[r, "\pi"] & S^1 \\
    \C
\end{tikzcd}
\]
Restricting \(\varpi\) to each fibre \(M_w\) of \(\pi\) yields a Lefschetz
fibration \(\varpi_w := \varpi|{M_w}\) (except on the exotic fibre
\(\varpi_w^{-1}(0)\)). The core idea of this proof is thinking about how the
unique Lefschetz critical value \(x_w \in \C^\times\) of \(\varpi_w\) winds
around the origin as we go around the base \(S^1\) of \(\pi\).

For \(t \in [0,1]\) choose a smooth family \(\gamma_t\) of vanishing paths
such that \(\gamma_0 = \gamma\), \(\gamma_1 = \gamma' = \tau(\gamma_0)\),
\(\gamma_t(1) = x_w\), and \(\gamma_t\) restricted to an interval of the form
\([0,b]\) for \(b < 1\) agrees with \(\gamma_0|_{[0,b]}\). Let
\(\tilde{X}\) be the horizontal lift (with respect
to \(\pi\) and \(\Omega\)) of the vector field \(X(t) = 2\pi i e^{2\pi i
t}\) and consider its flow \(\mu_t\). Note that
\(\mu_t\) maps \(M_w\) to
\(M_{e^{2\pi it}w}\) and \(\mu_1|_{M_\epsilon} = \pqtwist\) by
definition. Form the vanishing thimbles \(D_{\gamma_t} \subset
M_{e^{2\pi it}w}\) and consider the Lagrangian isotopy
\[
    L_t := \mu_{-t}(D_{\gamma_t}),
\]
which is contained in \((M_\epsilon,\omega)\). Moreover,
\(L_0 = D_{\gamma_0}\) and \(L_1 = \mu_{-1}(D_{\gamma_1}) =
\pqtwist^{-1}(D_{\tau(\gamma_0)})\). Therefore, it follows that \(\pqtwist(D_\gamma) \simeq
D_{\gamma'}\), as desired.
\end{proof}

\begin{remark}
    Consider the circle action \(\sigma_t\) with weights
\((1,dpq-1,q)\) on \(\C^3\):
\[
    \sigma_t(z_1,z_2,z_3) = (e^{2\pi i t}z_1, e^{2\pi i(dpq-1)t}z_2,
    e^{2\pi i qt}z_3).
\]
In the quotient \(\C^3/\frac{1}{p}(1,-1,q)\) this descends to a circle action
that factors through \(\sigma_{t/p}\). The restriction of \(\sigma_{t/p}\) to
\(\partial M_w\) gives the boundary \(\partial M_w\) the structure of a
Seifert fibration. Consider the Hamiltonian \(H \in C^\infty(M_w,\R)\):
\[
    H(z) = \pi(|z_1|^2 + (dpq-1)|z_2|^2 + q|z_3|^2).
\]
Its time \(-1\) flow \(\phi_{-1}^H\) is a representative of the boundary Dehn
twist on \(M_w\) (induced by the Seifert structure on \(\partial M_w\)). Since
the Milnor fibration \(\pi : M \to S^1\) is equivariant with respect to the
circle actions \((\sigma_{t/p},e^{2\pi dq i})\) on \(M\) and \(e^{2\pi dq i}\)
on \(S^1\), one can apply an argument of Seidel
\cite[Lemma~4.16]{seidel00graded} (with \(\sigma_{t/p}\) in place of
\(\sigma_t\) and \(\beta=dq\)) to show that
\[
    [\phi_{-1}^H] = [f]^{dq}
\]
in the symplectic mapping class group of \((M_\epsilon,\omega)\).
\end{remark}

We can view the \(B_{p,q}\) monodromy explicitly inside \(B_{d,p,q}\) as
follows. Consider the singular manifold
\[
    \tilde{N} = \{(z_1,z_2,z_3) \in \C^3 \mid
    z_1z_2=z_3^p\prod_{i=2}^d(z_3^p-a_i)\},
\]
and its quotient \(N = \tilde{N}/\frac{1}{p}(1,-1,q)\). This is a partial
smoothing of the \(\frac{1}{dp^2}(1,dpq-1)\) singularity. The unique singular
point \(0 \in N\) is of type \(\frac{1}{p^2}(1,pq-1)\). Similarly to
above, we form the manifolds
\[
    \tilde{N}_t := \{(z_1,z_2,z_3) \in \C^3 \mid |z| \le 1,\, z_1z_2 = (z_3^p -
    \psi(|z|^2)e^{2\pi i t}a_1)\prod_{i=2}^d(z_3^p - \psi(|z|^2)a_i)\}
\]
where \(0 < a_1 < \ldots < a_d\) are sufficiently small real numbers, along
with their quotients \(N_t = \tilde{N}_t/\frac{1}{p}(1,-1,q)\). The family
\(N_t\) is a smoothing of the singular manifold \(N\) with a unique
\(\frac{1}{p^2}(1,pq-1)\) singularity. Therefore the monodromy of \(N_t\) is
the \(B_{p,q}\) monodromy. Observe that, \(N_t \cong B_{d,p,q}\). As in the
proof of Lemma~\ref{lem:pqMonodromyInducesCentralTwist}, projecting to
\(z_3^p\) yields a Lefschetz fibration (up to the \(p\)-covered fibre over
\(0\)) \(\varpi_t : N_t \to \C\), and the same proof shows that the monodromy
of \(N_t\) induces the central twist \(\tau \in \Mod(A)\). Since we primarily
work with \(N_0\), we abuse notation and continue to write \(\varpi : N_0 \to
\C\) for \(\varpi_0 : N_0 \to \C\).

Recall the matching path \(\gamma_1\) shown in
Figure~\ref{fig:puncturedAnnulus}.
\begin{lemma}[]
    \label{lem:pqMonodromyNonTrivialHomotopy}
    The \(B_{p,q}\) symplectic monodromy \(\pqtwist\) acts non-trivially on
\(\pi_2(B_{d,p,q})\). Moreover, it acts with order \(p\) on the matching cycle
\(\Sigma_{\gamma_1}\).
\end{lemma}
\begin{proof}
    The quotient map \(\tilde{c} : \tilde{N}_0 \to N_0\) is the
(degree \(p\)) universal cover of \(N_0 \cong B_{d,p,q}\).
Consider the Lefschetz fibration fibration \(\tilde{\varpi} : \tilde{N}_0 \to
\C : \tilde{\varpi}(z_1,z_2,z_3) = z_3\) and the canonical \(p\)-to-\(1\)
branched cover \(c : \C \to \C : c(x) = x^p\). Observe that the diagram
\begin{equation}
    \label{eq:commutingLefFibs}
\begin{tikzcd}
    \tilde{N}_0 \arrow[r, "\tilde{c}"] \arrow[d, "\tilde{\varpi}"] &
	N_0 \arrow[d, "\varpi"] \\
    \C \arrow[r, "c"] & \C
\end{tikzcd}
\end{equation}
commutes and realises the Lefschetz fibration \(\tilde{\varpi} : \tilde{N}_0
\to \C\) as the pullback under \(c\) of \(\varpi : N_0 \to
\C\). Combined with the fact that \(\tilde{c}\) respects the symplectic
structures, this implies that \(\tilde{c}\) commutes with the parallel
transport of \(\tilde{\varpi}\) and \(\varpi\). In particular, given a
matching path \(\gamma\) of \(\varpi\) and a lift \(\tilde{\gamma}\) via \(c\)
to a matching path for \(\tilde{\varpi}\), the matching cycle
\(\Sigma_{\tilde{\gamma}} \subset \tilde{N}_0\) will be a lift of
\(\Sigma_\gamma \subset N_0\) with respect to \(\tilde{c}\).

Since \(\tilde{c}\) is the universal cover, it induces isomorphisms on homotopy groups
\(\pi_i(\tilde{N}_0) \cong \pi_i(N_0)\) for all \(i > 1\).
Moreover, as \(\tilde{N}_0\) is simply connected, the Hurewicz theorem
implies that \(\pi_2(\tilde{N}_0) \cong H_2(\tilde{N}_0;\Z) \cong \Z^{dp-1}\).
Therefore, the calculation of \((\pqtwist)_*[\Sigma_{\gamma_1}] =
[\Sigma_{\tau(\gamma_1)}] \in \pi_2(N_0)\) reduces to computing
the homology class of \(\Sigma_{\tilde{\gamma}}\), where \(\tilde{\gamma}\)
lifts \(\tau(\gamma_1)\).
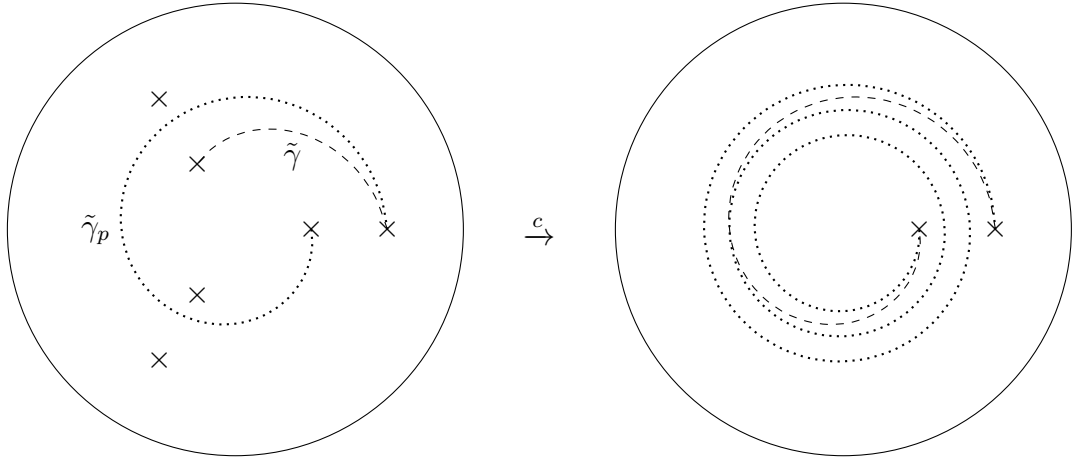
\begin{figure}\centering \begin{tikzpicture}
    \begin{scope}[shift={(-4,0)}]
	\draw (0,0) circle [radius=3];
	\foreach \t in {0,120,240}
	\draw (\t:1) node {\(\times\)} (\t:2) node {\(\times\)};
	\draw [dashed,domain=0:120,variable=\t,smooth,samples=10]
	    plot (\t:{(1-\t/120)*2 + (\t/120)});
	\draw (60:1.5) node [below] {\(\tilde{\gamma}\)};
	\draw[dotted,thick,domain=0:360,variable=\t,smooth,samples=30]
	    plot (\t:{(1-\t/360)*2 + (\t/360)});
	\draw (180:1.5) node [left] {\(\tilde{\gamma}_p\)};
    \end{scope}
    \draw (0,0) node {\(\xrightarrow{c}\)};
    \begin{scope}[shift={(4,0)}]
	\draw (0,0) circle [radius=3];
	\draw (1,0) node {\(\times\)} (2,0) node {\(\times\)};
	\draw [dashed,domain=0:360,variable=\t,smooth,samples=30]
	    plot (\t:{(1-\t/360)*2 + (\t/360)});
	\draw [dotted,thick,domain=0:1080,variable=\t,smooth,samples=90]
	    plot (\t:{(1-\t/1080)*2 + (\t/1080)});
    \end{scope}
\end{tikzpicture}
    \caption{
	The dashed path \(\tilde{\gamma}\) is a lift of the matching path
	\(\tau(\gamma_1)\). The dotted path \(\tilde{\gamma}_p\) is a lift of
	the path \(\tau^p(\gamma_1)\). The example drawn here is \(d=2\) and
	\(p=3\).
    }
    \label{fig:matchingPathLift}
\end{figure}

To this end, we fix a basis of \(H_2(\tilde{N}_0;\Z)\) to be the collection of
matching cycles corresponding to the following matching paths: the \(p-1\)
arcs, \(\mu_j : [0,1] \to \C : \mu_j(s) =
\sqrt[p]{a_1}e^{2\pi i(j+s)/p}\), for \(0 \le j < p-1\); and the \((d-1)p\)
paths \(\mu_{j,k} : [0,1] \to \C\), for \(0 \le j < p\) and \(1 \le k <
d\), where \(\mu_{0,k}\) is the straight path between the points
\(\sqrt[p]{a_k}\) and \(\sqrt[p]{a_{k+1}}\), and \(\mu_{j,k} = e^{2\pi i
j/p}\mu_{0,k}\). The illustration in Figure~\ref{fig:H2basis} should
dispel any confusion.
\begin{figure}\centering \begin{tikzpicture}[
]
    \begin{scope}[shift={(-4,0)}]
	\draw (0,-4) node {(a)};
        \draw (0,0) circle [radius=3];
        \foreach \t/\ttext/\pos in {0/0/below,120/1/left,240/2/below right}
    	\draw[dashed] (\t:1) node {\(\times\)} (\t:2) node {\(\times\)}
    	    (\t:1) -- node [\pos] {\(\mu_{\ttext,1}\)} (\t:2);
        \draw[dotted, thick] (1,0) arc [start angle=0, end angle=120, radius=1]
    	node[midway, above right] {\(\mu_0\)};
        \draw[dotted, thick] (120:1) arc [start angle=120, end angle=240,
    	radius=1] node [midway, left] {\(\mu_1\)};
    \end{scope}
    \begin{scope}[shift={(4,0)}]
	\draw (0,-4) node {(b)};
        \draw (0,0) circle [radius=3];
        \foreach \t in {0,120,240}
	    \draw (\t:1) node {\(\times\)} (\t:2) node {\(\times\)};
	\draw[dashed] (1,0) -- (2,0) (240:1) -- (240:2);
	\draw[dotted,thick,domain=0:120,variable=\t,smooth,samples=10]
	    plot (\t:{(1-\t/120)*1 + (\t/120)*1.5})
	    plot (\t+120:{(1-\t/120)*1.5 + (\t/120)*1});
	\draw (180:1.25) node [right=2mm] {\(\mu\)};
	\draw[dash dot,thick,domain=0:360,variable=\t,smooth,samples=30]
	    plot (\t:{(1-\t/360)*2 + (\t/360)});
	\draw (315:1.25) node [right] {\(\sigma_\mu(\mu_{0,1})\)};
    \end{scope}
    
\end{tikzpicture}
    \caption{
	(a) The matching paths corresponding to the chosen basis of
	\(H_2(\tilde{N}_0;\Z)\). (b) An \(A_3\)-configuration of matching paths:
	\(\mu_{0,1}\) (dashed), \(\mu=\sigma_{\mu_1}(\mu_0)\) (dotted), and
	\(\mu_{2,1}\) (dashed); along with the matching
	path \(\tilde{\gamma}_p = \sigma_\mu(\mu_{0,1})\) (dash dot). The case
	drawn here has \(d=2\), and \(p=3\).
    }
    \label{fig:H2basis}
\end{figure}
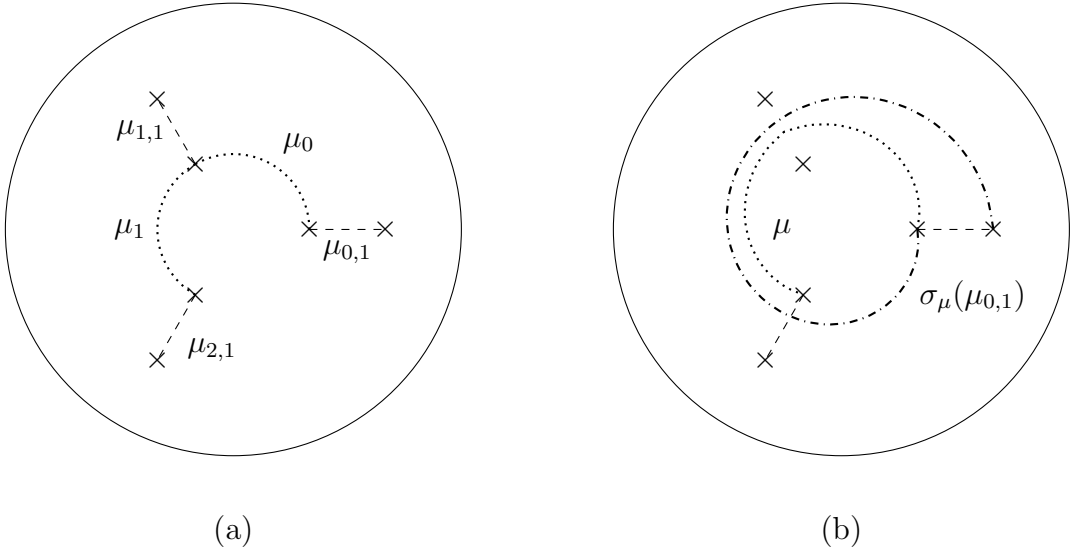

We have that \(\Sigma_{\tilde{\gamma}}\) is isotopic to the Dehn twist
\(\tau_{\Sigma_{\mu_0}}(\Sigma_{\mu_{0,1}})\) of \(\Sigma_{\mu_{0,1}}\) about
\(\Sigma_{\mu_0}\), so we can use the Picard-Lefschetz formula to calculate the
homology class:
\[
    [\tau_{\Sigma_{\mu_0}}(\Sigma_{\mu_{0,1}})] = [\Sigma_{\mu_{0,1}}] +
    ([\Sigma_{\mu_0}]\cdot[\Sigma_{\mu_{0,1}}])[\Sigma_{\mu_0}] =
    [\Sigma_{\mu_{0,1}}] + [\Sigma_{\mu_0}].
\]
This proves that \((\pqtwist)_*\) acts non-trivially on \(H_2(\tilde{N}_0;\Z) \cong
\pi_2(B_{d,p,q})\).

To show that \((\pqtwist^p)_*[\Sigma_{\gamma_1}] = [\Sigma_{\gamma_1}]\) it
suffices to check that \([\Sigma_{\tilde{\gamma}_p}] = [\Sigma_{\mu_{0,1}}]\),
where \(\tilde{\gamma}_p\) is the lift of \(\tau^p(\gamma_1)\) drawn in
Figure~\ref{fig:matchingPathLift}. Denote by \(\mu\) the matching path
obtained by successively half-twisting \(\mu_0\) along \(\mu_i\) for \(i > 0\)
(see Figure~\ref{fig:H2basis}(b) for an example with \(d=2\), and \(p=3\)):
\[
    \mu := \sigma_{\mu_{p-2}}\sigma_{\mu_{p-3}}\cdots\sigma_{\mu_1}(\mu_0).
\]
Then, \(\tilde{\gamma}_p\) is isotopic to \(\sigma_\mu^2(\mu_{0,1})\), which
implies that the matching cycle \(\Sigma_{\tilde{\gamma}_p}\) is Lagrangian
isotopic to the Dehn twisted sphere
\(\tau_{\Sigma_\mu}^2(\Sigma_{\mu_{0,1}})\), which is \emph{smoothly} isotopic
to \(\Sigma_{\mu_{0,1}}\). This implies that \([\Sigma_{\tilde{\gamma}_p}] =
[\Sigma_{\mu_{0,1}}]\), which completes the proof.
\end{proof}

The previous lemma shows that \(\pqtwist^p(\Sigma_{\gamma_1}) \simeq
\Sigma_{\tau^p(\gamma_1)}\) is homotopic to \(\Sigma_{\gamma_1}\). However,
they are not Lagrangian isotopic as we now show.
\begin{proposition}[]
    The matching cycle \(\pqtwist^p(\Sigma_{\gamma_1}) \simeq
\Sigma_{\tau^p(\gamma_1)}\) is \emph{not} Lagrangian isotopic to
\(\Sigma_{\gamma_1}\).
\end{proposition}
\begin{proof}
    Suppose that the result is false. Then, denoting \(L = \Sigma_{\gamma_1}\)
and \(L' = \Sigma_{\tau^p(\gamma_1)}\), there exists a \emph{Hamiltonian}
isotopy \(\phi_t^{H} \in \mathrm{Ham}(N_0,\omega)\) satisfying
\[
    \phi_1^H(L) = L'.
\]
Then \(H\) lifts to the Hamiltonian \(\tilde{H} := H \circ \tilde{c} : \tilde{N}_0 \to
\R\), whose flow \(\phi_t^{\tilde{H}}\) covers that of \(H\):
\[
    \tilde{c} \circ \phi_t^{\tilde{H}} = \phi_t^H.
\]
Therefore, by choosing the lift \(\tilde{L} := \Sigma_{\mu_{0,1}}\) of \(L\)
we must have that \(\tilde{L}' := \phi_1^{\tilde{H}}(\tilde{L})\) is a lift of
\(L'\). Moreover, by the diagram in Equation~\eqref{eq:commutingLefFibs},
\(\tilde{\varpi}(\tilde{L}')\) is (the image of) a lift of the matching path
\(\tau^p(\gamma_1)\). For topological reasons, this lift is exactly
\(\tilde{\gamma}_p\) as shown in Figure~\ref{fig:matchingPathLift}.

Picking up where we left off in the proof of
Lemma~\ref{lem:pqMonodromyNonTrivialHomotopy}, we deduce that \(\tilde{L}'\)
is Lagrangian isotopic to \(\tau_{\Sigma_\mu}^2(\Sigma_{\mu_{0,1}})\). As
noted, this is smoothly isotopic to \(\tilde{L} = \Sigma_{\mu_{0,1}}\), but a
famous result of Seidel \cite{seidel99SympKnotting} shows that it is not
Lagrangian isotopic. Indeed, an \(A_3\)-configuration of Lagrangian spheres
one can use to apply this result is given by the collection of matching cycles
corresponding to the matching paths \(\mu_{0,1}\), \(\mu
=\sigma_{\mu_{p-2}}\sigma_{\mu_{p-3}}\cdots\sigma_{\mu_1}(\mu_0)\), and
\(\mu_{p-1,1}\). Applying Seidel's theorem shows that the Lagrangian Floer
cohomology \(HF(\tilde{L}',\Sigma_{\mu_{p-1,1}})\) is non-zero,\footnote{Note
how in Figure~\ref{fig:H2basis} the matching path \(\sigma_\mu(\mu_{0,1})\)
intersects \(\mu_{p-1,1}\), whereas \(\mu_{0,1}\) doesn't.} so
\(\tilde{L}'\) cannot be Hamiltonian isotoped to \(\tilde{L}\) since
\(HF(\tilde{L},\Sigma_{\mu_{p-1,1}}) = 0\). This yields a contradiction, and
hence our initial assumption that \(L=\Sigma_{\gamma_1}\) and
\(L'=\Sigma_{\tau^p(\gamma_1)}\) were Lagrangian isotopic is false.
\end{proof}

\begin{corollary}[]
    The \(B_{p,q}\) monodromy has infinite order in the symplectic mapping
class group of \(B_{d,p,q}\).
\end{corollary}
\begin{proof}
    A basic extension of the above proof shows that \(L_k :=
\Sigma_{\tau^{kp}(\gamma_1)}\) is not Lagrangian isotopic to \(L_0 =
\Sigma_{\gamma_1}\) for all \(k \ne 0\). Since \(L_k \simeq
\pqtwist^{kp}(L_0)\) this proves the result.
\end{proof}
\begin{remark}
    This is not a surprising result; it parallels the story of weighted
homogeneous hypersurface singularities. What would be interesting is
investigating whether it is truly a symplectic phenomenon or not. That is, is
the monodromy of infinite order in the \emph{smooth} mapping class group?

Note that in the case \((d,p,q) = (n+1,1,0)\) --- which is the \(A_n\) du
Val singularity --- this is truly a symplectic phenomenon by Brieskorn's
simultaneous resolution \cite{brieskron68simultaneousResolution}. On the other
hand, recent work of Konno, Lin, Mukherjee, and Mu\~noz-Ech\'aniz
\cite{KLMME24monodromy} shows that the monodromy diffeomorphism of the Milnor
fibration for every weighted homogeneous hypersurface singularity
\emph{excluding} the ADE singularities has infinite order in the smooth
mapping class group. Their theorem does not apply to the \(B_{d,p,q}\) case
since \(b^+(B_{d,p,q}) = 0\) and \(\pi_1(B_{d,p,q}) \ne 1\).
\end{remark}

\subsection{Proof of the main theorem}
\label{sec:mainTheoremProof}

Consider the subgroup \(G\) of the symplectic mapping class group
\(\pi_0(\mathrm{Symp}_c(B_{d,p,q}))\) generated by the \(d-1\) Dehn twists
about the standard spheres \(L_i = \Sigma_{\gamma_i}\) and the
\(\frac{1}{p^2}(1,pq-1)\) symplectic monodromy \(\pqtwist\). We now state the
main theorem in its most precise form:
\begin{theorem}[]
    \label{thm:main}
    For every Lagrangian sphere \(L \subset B_{d,p,q}\) there exists \(\phi
\in G\) such that \(L\) is Lagrangian isotopic to \(\phi(L_1)\).
\end{theorem}

We'll do this by combining the main result of Section~\ref{ch:isotopy} with a
proof that the Lefschetz fibration \(\varpi : N_0 \to \C\) in the previous
section compactifies in a suitable sense to \(\refFib\).
We show that we can choose the reference almost complex structure \(\Jref\) on
\(X = X_{d,p,q}\) so that the fibres of the map \(\pi_{\Jref} : X \to \C\) are
compactifications of those of \(\varpi : N_0 \to \C\).\footnote{Recall that
    \(\refFib\) is the deformation of \(\pi_{\Jref}\) under the Gompf argument
of Lemma~\ref{lem:gompf}.} The upshot of this is that all the results proved
in Section~\ref{sec:dpqSymplectomorphisms} transfer immediately to
\(\refFib\).

\begin{lemma}[]
    \label{lem:lefFibLocalModel}
    There exists a symplectic embedding \(\iota : N_0 \hookrightarrow
X=X_{d,p,q}\) and an almost complex structure \(\Jref \in \mathcal{J}(D')\) on
\(X\) such that the fibres of \(\varpi : N_0 \to \C\) compactify to
\(\Jref\)-holomorphic fibres of the map \(\pi_{\Jref} : X \to \C\).
\end{lemma}
\begin{proof}
    The existence of this embedding will follow from showing that \(N_0\) has
an almost toric structure with base diagram isomorphic to that shown in
Figure~\ref{fig:bdpqFundamentalDomain}. Indeed, the compactification \(X\) was
constructed using only the almost toric base diagram. After this, we then
investigate how the almost complex structure \(J\) on \(N_0\), inherited from
its embedding in \(\C^3\), and the fibres of \(\varpi\) behave with respect to
the Hamiltonian \(H : N_0 \to \R\) used in the symplectic cut construction.

As in Lemma~7.2 of \cite{evansLTF}, the Hamiltonian system \(\mathbf{H} =
(|z_3|^2,\frac{1}{2}(|z_1|^2-|z_2|^2))\) on \(N_0\) has a fundamental
domain whose image under action coordinates is that shown in
Figure~\ref{fig:bdpqAlternativeFundamentalDomain}(a).
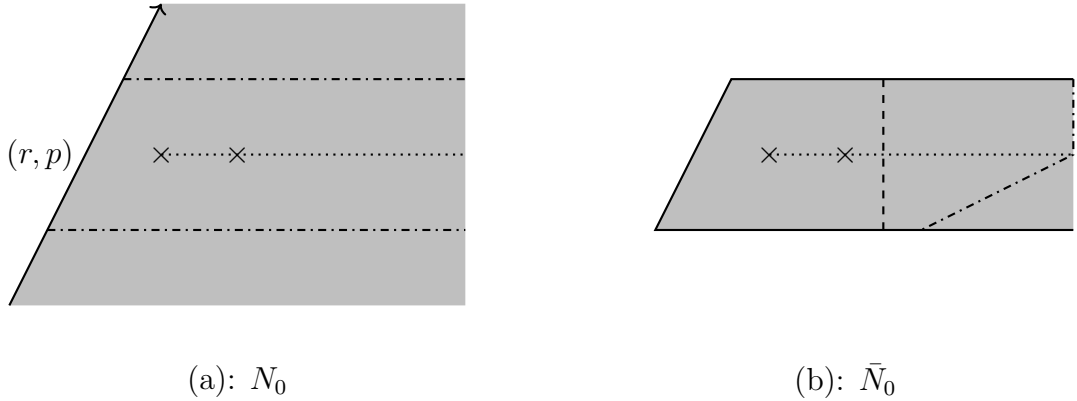
\begin{figure}\centering
\begin{tikzpicture}[
    moment polytope,
]
    \fill (0,0) -- (2,4) -- (6,4) -- (6,0) -- cycle;
    \draw[->] (0,0) -- node [left] {\((r,p)\)} (2,4);
    \draw[dotted] (2,2) node {\(\times\)} -- (3,2) node {\(\times\)} -- (6,2);
    \draw[dash dot] (0.5,1) -- (6,1) (1.5,3) -- (6,3);
    \draw (3,-1) node {(a): \(N_0\)};

\begin{scope}[shift={(8,0)}]
    \filldraw (6,1) -- (0.5,1) -- (1.5,3) -- (6,3);
    \draw[dotted] (2,2) node {\(\times\)} -- (3,2) node {\(\times\)} -- (6,2);
    \draw[dashed] (3.5,3) -- (3.5,1);
    \draw (3,-1) node {(b): \(\bar{N}_0\)};
    \draw[dash dot] (6,3) -- (6,2) -- (4,1);
\end{scope}

\end{tikzpicture}
    \caption{
	A fundamental domain derived from the Hamiltonian system \(\mathbf{H}
	: N_0 \to \R^2\). Here \(r\) is the unique integer \(0 < r < p\) such
	that \(rq \equiv 1 \mod p\). Cuts made are shown in dash dot.
    }
    \label{fig:bdpqAlternativeFundamentalDomain}
\end{figure}
In particular, the \(y\)-coordinate is the Hamiltonian \(H(z) =
\frac{1}{2}(|z_1|^2-|z_2|^2)\) which generates the circle action \(t \cdot
(z_1,z_2,z_3) = (e^{i t}z_1, e^{-i t}z_2,z_3)\).
After applying a \(SL_2(\Z)\) transformation to make the toric boundary
vertical and rotating the branch cut \(180^\circ\) clockwise, one obtains
exactly the fundamental domain of Figure~\ref{fig:bdpqFundamentalDomain}.
Under this correspondence, the vertical cuts made in
Figure~\ref{fig:BdpqCuts}(b) (which correspond to the \(J\)-holomorphic sections
of the Lefschetz fibrations \(\pi_J : X \to \C\)) translate to 
horizontal cuts in the above fundamental domain. This means that the
symplectic cut construction is done with respect to the Hamiltonian \(H\).
Recall that a fibre \(\varpi^{-1}(w)\) is given by the equation
\[
    z_1z_2 = \prod_{i=1}^d(w - \psi(|z|^2)a_i),
\]
which is preserved by the Hamiltonian flow \(\phi^H_t\). Therefore, the
intersection of a fibre with a regular level \(H^{-1}(r)\) is given by an
orbit of \(\phi^H_t\). This implies that, after taking the symplectic cuts,
the fibres of \(\varpi\) become symplectic spheres in \(X\).

As for the almost complex structure, note that in the region where the cut is
made, \(N_0\) is a complex submanifold of \(\C^3/\frac{1}{p}(1,-1,q)\).
Indeed, \(N_0 \cap (\frac{2}{3}\le|z|^2\le 1) = M_0 \cap
(\frac{2}{3}\le|z|^2\le 1)\) and \(M_0\) is holomorphic. Therefore, as the
circle action \(\phi^H_t\) is a subgroup of the \(U(3)\) action, this implies
that \(J\) is invariant under \(\phi^H_t\) and thus descends to give an almost
complex structure \(\bar{J}\) on the cut manifold \(\bar{N}_0\) shown in
Figure~\ref{fig:bdpqAlternativeFundamentalDomain}(b). To produce an almost
complex structure on \(X\) we need to make a cut corresponding to the
horizontal cut shown in Figure~\ref{fig:BdpqCuts}(a). However,
\emph{a priori}, the symplectic sphere introduced by this cut may not be
\(\bar{J}\)-holomorphic. Therefore, we define \(\Jref\) as follows. Fix a
closed subset of \(\bar{N}_0\) containing all of the focus-focus critical
points of the form indicated by the vertical dashed line in
Figure~\ref{fig:bdpqAlternativeFundamentalDomain}(b), and define \(\Jref\) to be
\(\bar{J}\) here. Then, in a neighbourhood of the vertical cut passing through
the branch cut define
\(\Jref\) to be that coming from the standard product complex structure in
\(S^2 \times S^2\) (compare Section~\ref{sec:constructingTheIsotopy} and
Figure~\ref{fig:fixedNeighbourhood}). Finally, on the remaining region pick
\(\Jref\) arbitrarily so that it makes the horizontal toric boundary
\(\Jref\)-holomorphic (which is possible by
Lemma~\ref{app:lem:divisorExtension}).\footnote{Of course, one also needs to
    handle resolving the singularities introduced by the horizontal cuts, as
    in Lemma~\ref{lem:ZCF}. However, this can be done so that the resulting
    resolution loci are \(\Jref\)-holomorphic without affecting the property
that the fibres of \(\varpi : N_0 \to \C\) away from the exotic fibre over
\(0\) compactify to \(\Jref\)-holomorphic spheres in \(X\).}

The final claim that the fibres of \(\varpi\) compactify to fibres of
\(\pi_{\Jref}\) is equivalent to saying that the compactified spheres live in
the homology class \(F = H-S\) (in the basis computed in
Lemma~\ref{lem:compactify}). Let \(A \in H_2(X)\) be the homology class of the
compactified fibres. The claim follows from the fact that \(A \cdot E_i = 0\), \(A \cdot
\mathcal{E}_j = 0\), \(A \cdot F = 0\), and \(A \cdot S = 1\).
\end{proof}

\begin{corollary}[]
    \label{cor:lefFibsCommute}
    The symplectic embedding \(\iota : N_0 \to X\) induces a commutative
square
\[
    \begin{tikzcd}
	N_0 \arrow[r, "\iota"] \arrow[d, "\varpi"] & X \arrow[d,
	"\pi_{\Jref}"] \\
	\C \arrow[r] & \C
    \end{tikzcd}.
\]
Furthermore, this induces a bijection of sets of isotopy classes of matching
paths of \(\varpi\) and \(\pi_{\Jref}\).
\end{corollary}
\begin{proof}
    The existence of the commutative square is essentially just a rephrasing
of Lemma~\ref{lem:lefFibLocalModel}. Notice that the restriction
\(\pi_{\Jref}|_{X \backslash \im\iota}\) consists only of regular fibres
by the choice of closed set on which \(\Jref\) agrees with the
almost complex structure \(\bar{J}\) coming from symplectic cutting \(N_0\).
This induces the claimed bijection of isotopy classes of matching paths,
which, said another way, means that every matching cycle of \(\pi_{\Jref}\) is
Lagrangian isotopic to a matching cycle of \(\varpi\).
\end{proof}

\begin{proof}[Proof of Theorem~\ref{thm:main}]
    The symplectic completion of \(N_0\) is symplectomorphic to \(B_{d,p,q}\).
Therefore, by applying the negative Liouville flow, we can assume that \(L
\subset N_0\) and thus, after compactifying, \(L \subset X\). Therefore, by
Corollary~\ref{cor:isotopyToMatchingCycle}, \(L\) is Lagrangian isotopic to a
matching cycle of \(\refFib\) in \(\mathring{X}\). The
results of Section~\ref{sec:dpqSymplectomorphisms} and
Corollary~\ref{cor:lefFibsCommute} imply that this matching cycle is Lagrangian isotopic
to a composition of Dehn twists and the \(B_{p,q}\) symplectic monodromy
applied to the matching cycle \(L_0=\Sigma_{\gamma_1}\). Finally, since the
isotopy is supported in \(\mathring{X}\), we can find a contact-type
hypersurface \(M\) in \(X \backslash D'\) such that the isotopy is contained
in the \emph{interior} \(X_ M \subset X \backslash D'\) of \(M\) and so that
the completion of \(X_M\) is symplectomorphic to \(B_{d,p,q}\). See
Figure~\ref{fig:hypersurface} for an example. Thereby we view the isotopy as
taking place in \(B_{d,p,q}\).
\begin{figure}\centering
    \begin{tikzpicture}[
	moment polytope,
    ]
        \filldraw (4,1) -- (0.5,1) -- (1.5,3) -- (6,3) -- (6,2) -- cycle;
	\draw[dotted] (2,2) node {\(\times\)} -- (3,2) node {\(\times\)} -- (6,2);
	\draw[dashed] (0.75,1.5) -- (3.5,1.5);
	\draw[dashed] (3.5,1.5) arc [radius=0.5, start angle=-90, end
	    angle=90];
	\draw[dashed] (3.5,2.5) -- (1.25,2.5);
    \end{tikzpicture}
    \caption{
	The dashed curve represents the image of the contact-type hypersurface
	\(M \subset X \backslash D'\) chosen to ensure that the completion of
	the interior is symplectomorphic to \(B_{d,p,q}\). We can ensure that
	the Lagrangian isotopy is contained in the interior of this
	hypersurface by altering the curve so that it doesn't touch any of the
	boundary edges except the leftmost one.
    }
    \label{fig:hypersurface}
\end{figure}
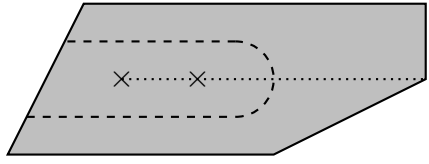
\end{proof}

\appendix
\section{Almost complex structures on symplectic divisors}
\label{app:acsOnDivisor}

We construct the almost complex structures necessary for the
arguments of Section~\ref{sec:fibreModuliSpaces}. Let \((M,\omega)\) be a
closed symplectic 4-manifold and \(S \subset M\) a symplectic divisor. In
particular, throughout the whole of this section, we will assume the
following:
\begin{itemize}
    \item the components \(S_i\) of \(S\) are closed, embedded symplectic
	submanifolds;
    \item intersecting components \(S_i\) and \(S_j\) do so symplectically
	orthogonally;\footnote{This condition is not strictly necessary, but
	it makes the proofs simpler and it is satisfied in the situation we
    handle in this paper.} and,
    \item there are no triple intersections, \emph{i.e.}~for distinct
	components \(S_i,S_j,S_k\), we have \(S_i \cap S_j \cap S_k =
	\emptyset\).
\end{itemize}
We will construct a compatible almost complex structure on \(M\) that
realises each component of \(S\) as a \(J\)-holomorphic submanifold. The idea
is to first construct \(J\) at the points of intersection, then extend in
neighbourhoods of each component, and finally to the whole of \(M\). All the
arguments presented here are variations on standard material, see for example
\cite[Chapters~2--3]{mcduffSalamonIntro} and \cite[Part~III]{daSilvaLectures}.

    First consider the simplest situation where \(S = S_1 \cup S_2\) is
composed of two symplectic surfaces that intersect symplectically
orthogonally exactly once. Note that, by the symplectic neighbourhood theorem
\cite[Theorem~3.4.10]{mcduffSalamonIntro}, a neighbourhood of \(S_1\) in \(M\)
is isomorphic to a neighbourhood of the zero section in the symplectic normal
bundle \(\nu{}S_1\) of \(S_1\). Thus, we may choose a symplectic
trivialisation of \(\nu{}S_1\) in a neighbourhood of the unique intersection
point \(\{x\} = S_1 \cap S_2\). In other words, we have an embedding \(\phi
: U \to M\) of a neighbourhood \(U\) of \(0 \in (\R^4,\omega_0)\), where
\(\omega_0 = \begin{pmatrix} 0 & I \\ -I & 0 \end{pmatrix}\) is the standard
symplectic form, such that \(\phi(\R^2 \times \{0\}) \subset S_1\).
Furthermore, the symplectic orthogonality condition between \(S_1\) and
\(S_2\) ensures that \(\ud\phi(0)(\{0\} \times \R^2) = T_xS_2\). The next step
is to adjust the map \(\phi\) to a symplectic embedding that maps the
symplectic planes \(\R^2 \times \{0\}\) and \(\{0\} \times \R^2\) into \(S_1\)
and \(S_2\) respectively, thus symplectically identifying the intersection
\(S_1 \cap S_2\) with the standard intersection of coordinate planes in
\(\R^4\). This is achieved by the following result.

\begin{lemma}[]
    \label{app:lem:orthogonalDarboux}
    Let \(W \subset (\R^{2n},\omega_0) = (\C^n,\omega_0)\) be an embedded
symplectic submanifold of the standard symplectic vector space that intersects
the coordinate plane \(\C^k \times \{0\}\) exactly once symplectically
orthogonally at \(0 \in \C^n\). Then there exists a symplectomorphism of a
neighbourhood of \(0 \in \C^n\) that maps the planes \(\C^k \times \{0\}\) and
\(\{0\} \times \C^{n-k}\) onto \(\C^k \times \{0\}\) and \(W\) respectively.
\end{lemma}
\begin{proof}
    Choose a symplectic embedding \(i : B^{2(n-k)} \subset \C^{n-k} \to \C\) of a
ball parametrising \(W\) and satisfying \(i(0) = 0\). In the following, we will
denote the planes \(\C^k \times \{0\}\) and \(\{0\} \times \C^{n-k}\) by
\(\C_{z_1}\) and \(\C_{z_2}\) respectively. The symplectic neighbourhood
theorem allows us to view a neighbourhood of the image of \(i\) as \(i^*\nu W\). Note
that the normal fibre \(\nu_0 W\) over \(0\) is exactly \(\C_{z_1}\). Moreover,
the symplectic form \(\omega_0\) defines a connection \(H \subset Ti^*\nu W\)
on \(i^*\nu W\) by taking the \(\omega_0\) complement of the vertical
distribution. This implies that the symplectic form \(\omega_0\) splits over
\(i^*\nu W\) into vertical and horizontal components: \({\omega_0 =
\omega' \oplus \omega''}\). Let \(\langle X_j,Y_j \rangle\) be
the standard symplectic basis of \(T_0 \C_{z_1} = \C_{z_1} = (T_0 W)^\omega\)
and extend this to a symplectic frame of \(i^*\nu W\) via parallel transport.
Using complex coordinates \(z_1 = (x_1 + i y_1, \ldots, x_k+i y_k)\),
define a map
\(\varphi : B^{2n} \subset \C^n = \C_{z_1} \times \C_{z_2} \to i^*\nu W\) by
\[
    \varphi(z_1,z_2) = \left(\sum_{j=1}^k x_jX_j(z_2) + y_jY_j(z_2),
	z_2\right).
\]
Note that \(\varphi\) is a bundle isomorphism \(\C^k \times B^{2(n-k)} \cong
i^*\nu W\). Moreover, since \(\nabla X_j = 0\) and \(\nabla Y_j = 0\),
\(\varphi\) pulls back the connection \(H\) on \(i^*\nu W\) to the canonical
one \(TB^{2(n-k)} \subset T\C^k \oplus TB^{2(n-k)}\) on \(\C^k \times
B^{2(n-k)}\). Consequently, \(\varphi\) respects the splitting of \(i^*\nu W\)
induced by \(H\), satisfies \(\varphi^*\omega' = \omega_{\C^k}\), and
\(\varphi^*\omega'' = \omega_{B^{2(n-k)}}\), which yields
\[
    \varphi^*\omega_0 = \varphi^*(\omega' \oplus \omega'')
    = \omega_{\C^k} \oplus \omega_{B^{2(n-k)}} = \omega_0.
\]
Hence, \(\varphi\) is the required symplectomorphism.
\end{proof}

\begin{corollary}[]
    \label{app:cor:acsOnIntersections}
    Let \(x\) be an intersection point of a symplectic divisor \({S \subset
(M,\omega)}\), then there exists an integrable compatible almost complex
structure on a neighbourhood of \(x\) that preserves each of the pieces of the
divisor that intersect at \(x\).
\end{corollary}
\begin{proof}
    The assumptions on the divisor \(S\) ensure that intersections between
components are isolated and satisfy the assumptions of the local result of
Lemma~\ref{app:lem:orthogonalDarboux}. Therefore, we can push forward the
standard complex structure \(i\) on \(\C^2\) via a chart given by
Lemma~\ref{app:lem:orthogonalDarboux}. Since \(i\) preserves the coordinate
planes, the result follows.
\end{proof}

    We are now a position to choose a compatible almost complex structure
\(J\) in a neighbourhood of all the intersection points of the divisor \(S
\subset (M,\omega)\). Then next step is to extend this to each of the
components \(S_i \subset S\) so that each \(S_i\) is \(J\)-holomorphic.

\begin{lemma}[]
    \label{app:lem:divisorExtension}
    Let \(W \subset (M,\omega)\) be an embedded symplectic
submanifold.\footnote{Here, \(W\) and \(M\) may be of any dimensions.} Then,
by the symplectic neighbourhood theorem, \(W\) has a neighbourhood
symplectomorphic to its symplectic normal bundle \((\nu W,\omega)\). Let
\(H \subset T\nu W\) be the canonical symplectic connection induced by
\(\omega\), and suppose that we are given a compatible almost complex structure
that splits \(J = J_H \oplus J_V\) according to the connection \(H\), and is
defined over an open neighbourhood \(U \subset \nu W\) of an open set of
\(W\). Then, for any open subset \(U' \subset \overline{U'} \subset U\), there
exists a compatible almost complex structure \(J' \in \mathcal{J}(T\nu W,
\omega)\) that agrees with \(J\) on \(U'\) and preserves the splitting
\(T\nu W = H \oplus V\). In particular, \(W\) is a \(J'\)-holomorphic
submanifold.
\end{lemma}
\begin{proof}
    The non-degeneracy of the symplectic form \(\omega\) splits the bundle
\(T\nu W\) into horizontal and vertical components \((H,\omega_H)\) and
\((V,\omega_V)\), that is,
\[
    {(T\nu W, \omega) = (H,\omega_H) \oplus (V,\omega_V)}.
\]
The hypothesis on \(J\) says that it decomposes into two
compatible almost complex structures \(J_H \in \mathcal{J}(H|_U,\omega_H)\) and
\(J_V \in \mathcal{J}(V|_U,\omega_V)\). Therefore, to extend \(J\) to all of
\(\nu W\), it suffices to extend both \(J_H\) and \(J_V\). This follows from
the non-emptiness and contractibility of the space of compatible almost
complex structures on any symplectic vector bundle, see
\cite[Proposition~2.6.4]{mcduffSalamonIntro} for example. The non-emptiness
and contractibility of \(\mathcal{J}(H,\omega_H)\) and
\(\mathcal{J}(V,\omega_V)\), ensure that we can find global sections \(J'_H \in
\mathcal{J}(H,\omega_H)\) and \(J'_V \in \mathcal{J}(V,\omega_V)\) that
agree with \(J_H\) and \(J_V\) over \(\overline{U'}\). Finally, since \(J' :=
J'_H \oplus J'_V\) preserves the horizontal bundle \(H\), and, for any point
\(x \in W\), \(H_{(x,0)} = T_xW\), we have that \(W\) is a \(J'\)-holomorphic
submanifold. This completes the proof.
\end{proof}

\begin{remark}
    \label{rem:productACS}
    Suppose that the normal bundle \((\nu W, \omega)\) above is trivial,
\emph{i.e.}~isomorphic to \((W \times \R^{2m}, \omega_W \oplus
\omega_{\R^{2m}})\). Then the horizontal distribution is simply \(W \times
\{x\} \subset W \times \R^{2m}\). Therefore, the extended almost complex
structure \(J'\) preserves those subspaces, making the canonical sections
\(s_x : W \to W \times \R^{2m} : s_x(w) = (w,x)\) \(J'\)-holomorphic.

Returning to the situation of Section~\ref{sec:fibreModuliSpaces}, since a
neighbourhood of the \(F\)-component of the divisor \(D\) is symplectomorphic
to a neighbourhood of a horizontal sphere in \(S^2 \times S^2\), we can choose
\(J\) near the points of intersection with the sections to be that coming from
the corresponding intersections of horizontal and vertical spheres in \(S^2
\times S^2\). Hence, the above implies that we can choose \(J\) near the
\(F\)-component of the divisor \(D\) so that nearby curves in
\(\mathcal{M}_{0,0}(F;J)\) all intersect the sections \(\omega\)-orthogonally.
\end{remark}

\begin{corollary}[]
    \label{app:cor:acsOnDivisor}
    Given a symplectic divisor \(S\) in a symplectic 4-manifold
\((M,\omega)\), there exists a compatible almost complex structure \(J \in
\mathcal{J}(M,\omega)\) that realises each component \(S_i\) of \(S\) as a
\(J\)-holomorphic curve.
\end{corollary}
\begin{proof}
    Corollary~\ref{app:cor:acsOnIntersections} and
Lemma~\ref{app:lem:divisorExtension} yield a compatible almost complex
structure \(J'\) defined in a neighbourhood of \(S\) such that each component
\(S_i\) is \(J'\)-holomorphic. Indeed, the set \(U'\) in the statement of
Lemma~\ref{app:lem:divisorExtension} can be chosen to be a shrunken
neighbourhood of that given by Corollary~\ref{app:cor:acsOnIntersections}.
Therefore, all that remains to do is extend \(J'\) to all of \(M\), which
again follows from \cite[Proposition~2.6.4]{mcduffSalamonIntro}.
\end{proof}

\section{The mapping class group of the punctured annulus}
\label{app:annulusMCG}

The following argument is based on Adrien Brochier's MathOverflow answer
\cite{annulusMCGMathOverflow}. Let \(S\) be a topological surface, and let
\(C(S,n)\) denote the configuration space of \(n\) points on \(S\). Denote the
2-disc by \(D\) and recall that the braid group \(Br_n\) on \(n\) strands is
defined to be the fundamental group of \(C(D,n)\). Similarly, denoting the
compact annulus by \(A\), the \emph{annular} braid group on \(n\) strands
\(Br_n(A)\) is defined to be \(\pi_1(C(A,n))\).

\begin{lemma}[]
    The annular braid group has the following presentation:
\begin{equation}
    Br_n(A) = \left\langle \tau, \sigma_1,\ldots,\sigma_{n-1} \,\left|\,
    \begin{array}{c}
	(\tau\sigma_1)^2 = (\sigma_1\tau)^2,\\
	\tau\sigma_i = \sigma_i\tau \; \forall i>1,\\
	\sigma_i\sigma_{i+1}\sigma_i = \sigma_{i+1}\sigma_i\sigma_{i+1},\\
	\sigma_i\sigma_j = \sigma_j\sigma_i \;\forall |i-j|>1 
    \end{array}\right.\right\rangle.
\end{equation}
\end{lemma}
\begin{proof}
    Recall the Artin presentation of \(Br_{n+1}\):
\begin{equation}
    Br_{n+1} = \left\langle \sigma_0,\ldots,\sigma_{n-1} \,\left|\,
    \begin{array}{c}
	\sigma_i\sigma_{i+1}\sigma_i = \sigma_{i+1}\sigma_i\sigma_{i+1},\\
	\sigma_i\sigma_j = \sigma_j\sigma_i \;\forall |i-j|>1 
    \end{array}\right.\right\rangle.
\end{equation}
By ``filling in'' the central hole of \(A\) with a punctured disc, we
construct an injective homomorphism \(i: Br_n(A) \to Br_{n+1}\) sending each
annular braid to the corresponding planar one which fixes the 0th
strand.\footnote{The strand associated with the central puncture in \(D\).}
Consider the map \(\phi : Br_{n+1} \to \Sigma_{n+1} = \Sym\{0,\ldots,n\}\)
sending each braid to its permutation of the punctures. The image of \(i\) is
\(\phi^{-1}(\mathrm{Stab}(0))\), which has generators \(\langle \sigma_0^2,
\sigma_1,\ldots,\sigma_{n-1} \rangle\). Setting \(\tau = \sigma_0^2\), we
obtain the claimed generating set for \(Br_n(A)\). The relation
\((\tau\sigma_1)^2 = (\sigma_1\tau)^2\) follows from the braid
relation \(\sigma_0\sigma_1\sigma_0 = \sigma_1\sigma_0\sigma_1\).
\end{proof}

Denote the mapping class group of a surface \(S\) as \(\Mod(S)\), interpreted
as in \cite[\S{}2]{farbMargalit}. As with \(\Mod(D^{\times n})\), where
\(D^{\times n}\) is the \(n\)-punctured disc \(D^{\times n} = D \backslash \{n
\text{ points}\}\), \(\Mod(A^{\times n})\) is closely related to its braid
group \(Br_n(A)\).
\begin{proposition}[]
    Let \(T_\alpha\) denote the isotopy class of the Dehn twist about a simple
closed curve \(\alpha\) near the central boundary of \(A^{\times n}\). Then
\begin{equation}
    \Mod(A^{\times n}) \cong Br_n(A) \times \langle T_\alpha \rangle
       	= Br_n(A) \times \Z.
\end{equation}
\end{proposition}
\begin{proof}
    The result follows from the Birman exact sequence
\cite[Theorem~9.1]{farbMargalit}:
\begin{equation}
    1 \longrightarrow \pi_1(C(A,n)) = Br_n(A) \longrightarrow \Mod(A^{\times n})
    \longrightarrow \Mod(A) \longrightarrow 1.
\end{equation}
We may apply this theorem since the identity component of
\(\mathrm{Homeo}^+(A,\partial{}A)\) is contractible
\cite[Theorem~1(D)]{earleSchatz70}, and so in particular,
\({\pi_1(\mathrm{Homeo}^+(A,\partial{}A)) = 1}\). Since \(\Mod(A) = \langle
T_\alpha \rangle = \Z\) is free, this sequence splits and we obtain the
result.
\end{proof}

\printbibliography[heading=bibintoc,title=References]

\end{document}